\documentclass{article}

\usepackage{microtype}
\usepackage{graphicx}
\usepackage{subcaption}
\usepackage{booktabs} 
\usepackage{bm}
\usepackage{hyperref}

\usepackage{natbib}
\usepackage[preprint]{icml2026}

\usepackage{colortbl}
\usepackage{amsmath}
\usepackage{amssymb}
\usepackage{mathtools}
\usepackage{amsthm}
\allowdisplaybreaks[1]

\usepackage[normalem]{ulem}

\usepackage[capitalize,noabbrev]{cleveref}

\theoremstyle{plain}
\newtheorem{theorem}{Theorem}[section]

\newtheorem{lemma}[theorem]{Lemma}
\newtheorem{corollary}[theorem]{Corollary}
\theoremstyle{definition}

\newtheorem{assumption}[theorem]{Assumption}
\theoremstyle{remark}

\usepackage[textsize=tiny]{todonotes}
\usepackage{stfloats}
\usepackage{multirow}

\newcommand*{\rr}{\mathbb{R}}
\newcommand*{\abs}[1]{\left|#1\right|}
\newcommand*{\norm}[1]{\left\|#1\right\|}

\newcommand*{\innerproduct}[2]{\left\langle #1,#2\right\rangle}

\DeclareMathOperator*{\argmin}{arg\,min}

\newcommand*{\ma}{\mathcal{A}}
\newcommand*{\md}{\mathcal{D}}
\newcommand*{\mh}{\mathcal{H}}
\newcommand*{\mo}{\mathcal{O}}
\newcommand*{\ep}{\mathbb{E}}

\begin{document}

\icmltitlerunning{Byzantine-Robust Distributed SGD}
\twocolumn[
  \icmltitle{Byzantine-Robust Distributed SGD: A Unified Analysis and Tight Error Bounds} 

  \icmlsetsymbol{equal}{*}

  \begin{icmlauthorlist}
    \icmlauthor{Boyuan Ruan}{boyuan}
    \icmlauthor{Xiaoyu Wang}{xiaoyu}
    \icmlauthor{Ya-Feng Liu}{yafeng}
  \end{icmlauthorlist}

  \icmlaffiliation{boyuan}{School of Advanced Interdisciplinary Sciences, University of Chinese Academy of Sciences, Beijing 100049, China}
  \icmlaffiliation{xiaoyu}{School of Mathematical Sciences, University of Chinese Academy of Sciences, Beijing 100049, China}
  \icmlaffiliation{yafeng}{Ministry of Education Key Laboratory of Mathematics and Information Networks, School of Mathematical Sciences, Beijing University of Posts and Telecommunications, Beijing 102206, China}

  \icmlcorrespondingauthor{Xiaoyu Wang}{wangxiaoyu@ucas.ac.cn}
  \icmlcorrespondingauthor{Ya-Feng Liu}{yafengliu@bupt.edu.cn}

  \vskip 0.3in
]



\printAffiliationsAndNotice{}  

\begin{abstract}
Byzantine-robust distributed optimization relies on robust aggregation rules to mitigate the influence of malicious Byzantine workers. Despite the proliferation of such rules, a unified convergence analysis framework that accommodates general data heterogeneity is lacking. In this work, we provide a thorough convergence theory of Byzantine-robust distributed stochastic gradient descent (SGD), analyzing variants both with and without local momentum.  We establish the convergence rates for nonconvex smooth objectives and those satisfying the Polyak-\L{}ojasiewicz condition under a general data heterogeneity assumption. Our analysis reveals that while stochasticity and data heterogeneity introduce unavoidable error floors, local momentum provably reduces the error component induced by stochasticity. Furthermore, we derive matching lower bounds to demonstrate that the upper bounds obtained in our analysis are tight and characterize the fundamental limits of Byzantine resilience under stochasticity and data heterogeneity.  Empirical results support our theoretical findings.

\end{abstract}

\section{Introduction}
Distributed machine learning has emerged as a dominant paradigm for large-scale learning tasks, enabling efficient model training across dispersed data sources \cite{alistarh2018byzantine, kairouz2021advances, wang2021field, guerraoui2024byzantine}. In a standard centralized distributed machine learning setting, a server coordinates multiple workers that compute on local data to collaboratively train a global model. Despite its scalability, this paradigm introduces significant vulnerabilities. Specifically, the presence of Byzantine workers \cite{lamport1982byzantine}—arising from malicious attacks, hardware malfunctions, or data poisoning—can introduce arbitrarily incorrect updates that undermine the entire training process. This threat has motivated extensive research into Byzantine-resilient algorithms \cite{blanchard2017machine, chen2017distributed, yin2018byzantine, xie2018generalized, li2019rsa, karimireddy2021learning, pillutla2022robust,allouah2023fixing, guerraoui2024byzantine}. These algorithms replace the standard averaging step in distributed optimization with robust aggregation rules designed to filter outliers and preserve the integrity of the learning process. However, the design and analysis of such rules face two major challenges.

    The first challenge is \emph{data heterogeneity}. Early analyses of robust aggregation rules typically assume data homogeneity, where samples across workers are independent and identically distributed \cite{blanchard2017machine, xie2018generalized, yin2018byzantine, wu2020federated}. While data homogeneity simplifies the detection of Byzantine updates, it rarely holds in real-world scenarios such as federated learning \citep{konevcny2016federated, brendan2017communication, kairouz2021advances}, where local data distributions are inherently heterogeneous. This heterogeneity causes honest updates to naturally diverge, which confuses the aggregation rules in distinguishing honest updates from malicious ones and in turn degrades model accuracy~\citep{data2021byzantine, allouah2023robust, shi2025optimal}.

    The second challenge is \emph{stochastic noise}. Given the massive size of modern datasets, workers typically rely on mini-batch gradients rather than full-batch gradients during local computations \cite{nemirovski2009robust, bottou2018optimization}. Stochastic noise in mini-batch gradients further enlarges discrepancies among worker updates and misleads robust aggregation rules. The effect of stochastic noise cannot be mitigated unless local momentum or variance reduction techniques are introduced~\cite{karimireddy2021learning, farhadkhani2022byzantine, allouah2023fixing, gorbunov2023variance}.

    Therefore, understanding the joint effect of data heterogeneity and stochastic noise on Byzantine-resilient algorithms 
    requires a unified and general analysis. In particular, we focus on two state-of-the-art algorithmic frameworks: \textbf{robust distributed SGD (R-DSGD)} and \textbf{robust distributed SGD with local momentum (R-DSGD-M)}. These are natural, Byzantine-resilient extensions of the widely-used distributed SGD (DSGD) algorithm, with or without momentum. To contextualize our work, we first provide a brief review of the relevant literature on Byzantine-robust distributed optimization.
    \begin{table*}[t]
        \centering
        \caption{A comparison of prior work with our convergence results of R-DSGD and R-DSGD-M. Here, $\sigma^2$ refers to the stochastic noise bound, $\mu$ is the PL condition parameter, and $\Delta_0:=f_{\mh}(x^0)-f_{\mh}^*$ represents the initial gap on the global objective function. The \textbf{bold terms} indicate the Byzantine error, which we prove to be tight for the first time. }\label{list of previous analysis}
        \renewcommand{\arraystretch}{1.5}
        \begin{tabular}{|c|c|c|c|c|}
    \hline
    Settings & Algorithm & Heterogeneity & Convergence & Reference \\
    \hline
    \multirow{5}{*}{\begin{tabular}{c}
    
    smooth\\
    \end{tabular}}
    
    & R-DSGD & $(G,0)$ &
    $\mo\left(\frac{\sigma^2+\Delta_0}{\sqrt{T}}+\kappa\sigma^2+\kappa G^2\right)$
    & \cite{guerraoui2024robust} \\
    
    & \cellcolor{gray!15} R-DSGD & \cellcolor{gray!15} $(G,B)$ & \cellcolor{gray!15}
    $\mo\biggl(\frac{\sigma^2+\Delta_0}{\sqrt{T}}
    +\bm{\frac{\kappa\sigma^2+\kappa G^2}{1-\kappa B^2}
    }\biggr)$
    & \cellcolor{gray!15} Corollary \ref{coro:r-dsgd:nonconvex} \\
    \cline{2-5}
    & R-DSGD-M & $(G,0)$ &
    $\mo\left(\frac{(1+\kappa)\sigma^2+\Delta_0}{\sqrt{T}}+\kappa G^2\right)$
    & \cite{allouah2023fixing} \\
    
    & R-DGD-M & $(G,B)$ &
    $\mo\left(\frac{\Delta_0}{T}+\frac{\kappa G^2}{1-\kappa B^2}\right)$
    & \cite{gupta2025reconciling} \\
    
    & \cellcolor{gray!15} R-DSGD-M & \cellcolor{gray!15} $(G,B)$ & \cellcolor{gray!15}
    $\mo\left(\frac{(1+\kappa)\sigma^2+\Delta_0}{\sqrt{T}}
    +\bm{\frac{\kappa G^2}{1-\kappa B^2}}\right)$
    &\cellcolor{gray!15} Corollary \ref{coro:r-dsgd-m:convergence} \\
    \hline
    \multirow{4}{*}{\begin{tabular}{c}
    
    smooth\\
    and PL
    \end{tabular}}
    & R-DSGD & $(G,0)$ &
    $\mo\left(\frac{(1+\kappa)\sigma^2+\kappa G^2}{T}
    +\frac{\kappa\sigma^2+\kappa G^2}{\mu}
    \right)$
    & \cite{guerraoui2024robust} \\
    
    & \cellcolor{gray!15} R-DSGD & \cellcolor{gray!15} $(G,B)$ & \cellcolor{gray!15}
    $\mo\left(\frac{(1+\kappa)\sigma^2+\kappa G^2}{T}
    +\bm{\frac{\kappa\sigma^2+\kappa G^2}{\mu(1-\kappa B^2)}}
    \right)$
    & \cellcolor{gray!15} Theorem \ref{convergence for PL R-DSGD} \\
    \cline{2-5}
    & R-DSGD-M & $(G,0)$ &
    $\mo\left(\frac{(1+\kappa)\sigma^2}{T}
    +\frac{\kappa G^2}{\mu}\right)$
    & \cite{guerraoui2024robust} \\
    
    & \cellcolor{gray!15} R-DSGD-M &\cellcolor{gray!15} $(G,B)$ &
    \cellcolor{gray!15}$\mo\left(\frac{(1+\kappa)\sigma^2}{T}
    +\bm{\frac{\kappa G^2}{\mu(1-\kappa B^2)}}\right)$
    &\cellcolor{gray!15} Theorem \ref{convergence for PL R-DSGD-M} \\
    \hline
    \end{tabular}

    \end{table*}
    
    \textbf{Related work.} Distributed optimization without Byzantine attacks has been extensively studied under various assumptions regarding data heterogeneity  \cite{li2020federated, karimireddy2020scaffold, woodworth2020minibatch, zhang2021fedpd, wang22wireless, zakerinia2024communication}.
    Among these, one of the most general conditions to date is the $(G,B)$-bounded dissimilarity~\cite{karimireddy2020scaffold}. To systematically analyze Byzantine-resilient algorithms under data heterogeneity, recent works have proposed unified conditions for robust aggregation rules \cite{karimireddybyzantine,farhadkhani2022byzantine,allouah2023fixing}. One prominent example is the $(b,\kappa)$-robustness condition \cite{allouah2023fixing}, which qualifies the ability of an aggregation rule to estimate the average of honest workers' outputs despite the presence of $b$ Byzantine workers. 
    
    Under this framework, several studies have established convergence guarantees for R-DSGD and R-DSGD-M (robust distributed SGD with and without momentum). However, existing results are limited by specific constraints. For instance, the convergence results of R-DSGD and R-DSGD-M are presented in  \citet{guerraoui2024robust}; yet, both are restricted to the simpler case of $(G,0)$-bounded dissimilarity. Furthermore, while  \citet{gupta2025reconciling} have explored the combination of communication compression and Byzantine robustness under the $(G,B)$-bounded dissimilarity, their analysis is restricted to deterministic algorithms. \emph{To the best of our knowledge, no prior work has provided a thorough analysis of R-DSGD and R-DSGD-M with $(b,\kappa)$-robust aggregation rules under the general $(G,B)$-bounded dissimilarity assumption.} 

    Prior research confirms that stochastic noise and data heterogeneity introduce unavoidable error floors in Byzantine-robust convergence. Regarding stochastic noise, \citet{karimireddy2021learning} have proved that any permutation-invariant robust algorithm must suffer a Byzantine error of $\mo(\frac{b}{n}\sigma^2)$. Concerning data heterogeneity, \citet{shi2025optimal} have established the lower bound $\mo\left(\kappa G^2\right)$ on the unavoidable error for a class of first-order algorithms under the $(G,0)$-bounded dissimilarity, while \citet{allouah2023robust} have derived the lower bound $\mo\left(\frac{b}{n-(2+B^2)}G^2\right)$ for any deterministic Byzantine-robust algorithm under the general $(G,B)$-bounded dissimilarity. However, these results do not fully match the error upper bound of $\mo(\kappa\sigma^2)$ and $\mo(\frac{\kappa G^2}{1-\kappa B^2})$ achieved in prior work~\cite{allouah2023robust, gupta2025reconciling}. 
\emph{Specifically, existing lower bounds are unable to fully characterize the dependence on $\kappa$ for general $(b,\kappa)$-robust aggregation rules under the $(G,B)$-bounded dissimilarity.} 
    
    \textbf{Our contributions.}
    This work provides the first tight analysis of R-DSGD and R-DSGD-M with $(b, \kappa)$-robust aggregation rules and under the general $(G,B)$-bounded dissimilarity assumption.
    In summary, our contributions are twofold:
    \begin{itemize}
        \item \textbf{Convergence analysis (upper bounds):} We establish that R-DSGD converges at a rate of $\mo(1/\sqrt{T})$ for nonconvex and smooth objective functions, subject to an error floor of $\mo\left(\frac{\kappa\sigma^2}{1-\kappa B^2} +\frac{\kappa G^2}{1-\kappa B^2}\right)$. Furthermore, we demonstrate that while introducing local momentum does not improve the convergence rate, it effectively eliminates the error term $\mo\left(\frac{\kappa\sigma^2}{1-\kappa B^2}\right)$ induced by stochastic noise. We extend our analysis to the case where the global objective function satisfies the PL condition and show that our framework directly recovers the results in the absence of Byzantine workers. To the best of our knowledge, these are the first theoretical results derived explicitly under the $(G,B)$-bounded dissimilarity assumption.
   
        \item \textbf{Tightness of analysis (lower bounds):} We prove that the heterogeneity-induced error $\mo\left(\frac{\kappa G^2}{1-\kappa B^2}\right)$ is unavoidable for both R-DSGD and R-DSGD-M. In addition, we show that the noise-induced error $\mo\left(\frac{\kappa\sigma^2}{1-\kappa B^2}\right)$ is unavoidable for R-DSGD. These lower bound results confirm the tightness of our upper bounds and provide the sharpest characterization of error limits to date under the $(G,B)$-bounded dissimilarity assumption.
    \end{itemize}
     A detailed comparison between our theoretical results and prior work is provided in Table~\ref{list of previous analysis}. Finally, we evaluate the empirical performance of R-DSGD and R-DSGD-M under various settings, demonstrating that our theoretical results are consistent with the empirical observations. 
    
    \textbf{Paper outline.} The rest of the paper is organized as follows. Section 2 reviews the problem setup and the algorithmic framework. Section 3 presents our main theoretical convergence results of R-DSGD and R-DSGD-M. Section 4 provides tight error lower bounds. Section 5 reports empirical results on benchmark learning tasks. Section 6 summarizes our contributions and outlines further research directions.
    
    \textbf{Notation.} Throughout this paper, $\norm{\cdot}$ denotes the Euclidean norm. We use the notation $[n]$ to represent the set  $\left\lbrace 1, 2,\ldots, n \right\rbrace$. For any real value $a\in\rr$, denote $\lfloor a\rfloor$ as the largest integer that is smaller than $a$. Superscripts indicates iteration indices (e.g., $x^t$ denotes the model at round $t$), while subscripts indicate worker indices or scalar components. We use the big-O notation $\mo(\cdot)$ to hide irrelevant constant factors and higher-order terms. Finally, $\ep[\cdot]$ denotes the total expectation over all sources of randomness.    
    \section{Problem Setup and Algorithmic Framework}
    \subsection{Problem Setup}
    We consider a centralized server-worker distributed network consisting of a central server and $n$ workers,  where each worker only communicates with the central server. The data distribution of the $i$-th worker is denoted by $\md_i$. For a model parameterized by $x\in\rr^d$ and a loss function $\ell$, the local objective function on the $i$-th worker is 
    defined as \[f_i(x)=\ep_{\xi_i\sim\md_i}\left[\ell(x,\xi_i)\right].\]
    We assume that among $n$ workers, $b$ of them may be Byzantine. Let $\mh$ denote the unknown index set of honest workers with cardinality $\abs{\mh}=h:=n-b$. Since the data held by Byzantine workers may be unreliable or adversarial, the goal is to minimize the global objective function over the honest workers $f_{\mh}$ \cite{allouah2023fixing, guerraoui2024robust}:
    \begin{equation}\label{the problem}
        \min_{x\in\rr^d} f_{\mh}(x):=\frac{1}{h}\sum_{i\in\mh} f_i(x).
    \end{equation}
    Note that if $b=0$, then $f_{\mh}=\frac1n\sum_{i=1}^nf_i$, recovering the standard objective function in distributed optimization.
    
    Throughout the paper, we assume that the function $f_{\mh}$ is bounded from below by $f_{\mh}^* :=\inf_{x\in\rr^d}f_{\mh}>-\infty$.
    We also impose the following standard assumptions in first-order optimization \cite{bottou2018optimization, nesterov2018lectures}. 
    \begin{assumption}[Smoothness]\label{smoothness}
        The global objective function $f_{\mh}$ is differentiable, and there exists a constant $L>0$ such that $f_{\mh}$ is $L$-smooth, i.e., for all $x,y\in\rr^d$,  
        \[f_{\mh}(x)\leq f_{\mh}(y)+\innerproduct{\nabla f_{\mh}(x)}{y-x}+\frac{L}{2}\norm{y-x}^2.\]
        This condition is equivalent to $\nabla f_{\mh}$ being Lipschitz continuous with constant $L$.
    \end{assumption}
    \begin{assumption}[Polyak-\L ojasiewicz (PL) condition]\label{pl}
        There exists a constant $\mu>0$ such that for all $x\in\rr^d$, 
        \[\norm{\nabla f_{\mh}(x)}^2\geq 2\mu(f_{\mh}(x)-f_{\mh}^*).\]
    \end{assumption}
    Data heterogeneity across workers arises from differences in the underlying distributions $\left\{\md_i\right\}$ and is reflected in the discrepancies among the local objective functions $\{f_i\}$. To model this heterogeneity among honest workers, extensive research on Byzantine-resilient algorithms relies on the \textit{bounded heterogeneity assumption} \cite{gorbunov2023variance,allouah2023fixing, cheng2024momentum, yang2025tension}: there exists a constant $G\geq 0$ such that for all $x\in\rr^d$, 
\[
    \frac{1}{h}\sum_{i\in\mh}\norm{\nabla f_i(x)-\nabla f_{\mh}(x)}^2\leq G^2.
\]
    However, this assumption is too restrictive to capture the benign scenarios such as over-parameterization, where honest local objective functions may be different but share the same minimizer. Furthermore, this bounded heterogeneity assumption fails in the simple least-squares regression \cite{allouah2023robust}. To address these limitations, we adopt the more general $(G,B)$-bounded dissimilarity assumption \cite{karimireddy2020scaffold, noble2022differentially}, which is standard in federated learning and has recently been applied to analyze Byzantine-resilient algorithms \cite{allouah2023robust, gupta2025reconciling, allouah2025adaptive}. 
    
    \begin{assumption}[$(G,B)$-bounded dissimilarity]\label{heterogeneity}
        There exist constants $G,B\geq 0$ such that for all $x\in\rr^d$,
        \[\frac{1}{h }\sum_{i\in\mh}\norm{\nabla f_i(x)-\nabla f_{\mh}(x)}^2\leq G^2+B^2\norm{\nabla f_{\mh}(x)}^2.\]
    \end{assumption}
    Intuitively, $G$ quantifies the discrepancy between the minimizers of the local and global objectives, while $B$ captures the scale difference between local and global gradients.
    \begin{algorithm}[t]
    \caption{R-DSGD-M }\label{R-DSGD-M}
    \begin{algorithmic}[1]
    \item[\textbf{Input}:]  initial model $x^0$, initial local momentum $m_i^0=0$ for all honest workers, momentum parameters $\{\beta_t\}$, stepsizes $\{\gamma_t\}$, robust aggregation rule $\ma$.
    \FOR{$t = 1, 2, \ldots$}
        \STATE \textit{Server} broadcasts $x^{t-1}$ to \textit{all workers};
        \FOR{every \textit{honest worker} $i \in \mh$ \textit{in paralleel}}
        \STATE Compute a stochastic gradient $g_{i}^{t}$ at $x^{t-1}$;
        \STATE Update the local momentum $m_i^t=\beta_t m_i^{t-1}+(1-\beta_t)g_i^t$;
        \STATE Send the local momentum $m_i^{t}$ to the server;
        \ENDFOR
        \STATE Every \textit{Byzantine worker} sends an arbitrary vector $m_j^t$ to the server;
        \STATE \textit{Server} aggregates the update vector using the aggregation rule $\ma$: $g^t=\ma(m_1^t,m_2^t,\ldots,m_n^t)$;
        \STATE \textit{Server} updates the model: $x^{t}=x^{t-1}-\gamma_t g^t$.
    \ENDFOR
    \item[\textbf{Output}:] model parameters $\{x^0,x^1,\ldots\}$.
    \end{algorithmic}
    \end{algorithm}
    \subsection{The R-DSGD-M Framework}    
    To solve problem (\ref{the problem}), we employ the widely-used framework R-DSGD-M outlined in Algorithm \ref{R-DSGD-M}. Specifically, the process begins with an initial model $x^0$ at the server and an initial momentum vector $m_i^0$ at each honest worker $i$. At the $t$-th iteration, the server broadcasts the current model $x^{t-1}$ to all workers. Upon receiving $x^{t-1}$, each honest worker computes a stochastic gradient $g_i^t$, and then updates its local momentum $m_i^t$.
    The honest workers then transmit their updated local momentum to the server, whereas the Byzantine workers may send arbitrary or adversarial vectors. The server subsequently applies an aggregation rule $\ma: \rr^{d\times n}\to\rr^d$ to compute the global update $g^t$ and performs a gradient descent step to obtain the new model $x^t$. When the momentum parameter is set to $\beta_t\equiv 0$, this R-DSGD-M framework reduces to R-DSGD.

    We impose the following assumption on the local stochastic gradients $\{g_i^t\}$ computed by honest workers, ensuring that $g_i^t$ is an unbiased estimator of $\nabla f_i(x^{t-1})$ with bounded variance. This is a standard assumption in the stochastic optimization literature \cite{bottou2018optimization, lan2020first}.
    \begin{assumption}[Stochastic first-order oracle]\label{sfo}
        The $i$-th honest local worker has access to a stochastic oracle that generates a stochastic approximation $g_i(x,\xi_i)$ for $\nabla f_i(x)$ at any $x\in\rr^d$ such that 
        \begin{align*}
            &\ep_{\xi_i\sim\md_i} [g_i(x,\xi_i)]=\nabla f_i(x),\\
            &\ep_{\xi_i\sim\md_i}\left[\norm{g_i(x,\xi_i)-\nabla f_i(x)}^2\right]\leq\sigma^2.
        \end{align*}
    Furthermore, we assume that the stochastic gradients generated by different workers are mutually independent.
    \end{assumption}
    
    In classical distributed SGD, the aggregation rule $\ma$ is simply the average of the $n$ local updates:
\[
    \ma(m_1^t,m_2^t,\ldots,m_n^t)=\frac{1}{n}\sum_{i=1}^nm_i^t.
\]
However, Byzantine workers can arbitrarily manipulate the simple average, thereby damaging the training process \cite{blanchard2017machine}. To mitigate the influence of such outliers, various robust aggregation rules have been proposed, including Krum \cite{blanchard2017machine}, coordinate-wise median (CwM) \cite{yin2018byzantine}, coordinate-wise trimmed mean (CwTM) \cite{yin2018byzantine}, geometric median (GM) \cite{chen2017distributed,wu2020federated}, and centered clipping (CC) \cite{karimireddy2021learning}.
Various conditions have been proposed to analyze the performance of these rules, such as $(b,\lambda)$-resilient averaging \cite{farhadkhani2022byzantine} and $(\delta_{\max},c)$-agnostic robust aggregation \cite{karimireddybyzantine}. Recently, the concept of $(b,\kappa)$-robustness is introduced in \citet{allouah2023fixing} to encompass these existing definitions.
    \begin{assumption}[$(b,\kappa)$-robustness]\label{aggregation}
        For a given number $b<n/2$ of Byzantine workers, there exists a constant $\kappa>0$ such that for any vectors $x_1, x_2,\ldots,x_n\in\rr^d$ and any subset $\mh\subset[n]$ with cardinality $\abs{\mh}=n-b$, the aggregation rule $\ma$ satisfies 
        \[\norm{\ma(x_1,x_2,\ldots,x_n)-\bar{x}_{\mh}}^2\leq\frac{\kappa}{h }\sum_{i\in\mh}\norm{x_i-\bar{x}_{\mh}}^2,\]
        where $\bar{x}_{\mh}=\frac1h\sum_{i\in\mh}x_i$.
    \end{assumption}
    The $(b,\kappa)$-robustness condition requires that the aggregated output $\ma(x_1,x_2,\ldots,x_n)$ remains close to the ideal average $\bar{x}_{\mh}$, with the error bounded by the dispersion of the honest inputs scaled by $\kappa$. Ideally, one desires $\kappa=0$ to ensure perfect aggregation for any selection of $\mh$. However, as demonstrated in \cite{allouah2023fixing}, an aggregation rule can be $(b,\kappa)$-robust only if $\kappa\geq \frac{b}{n-2b}$. 

    \textbf{On the scope of the $(b,\kappa)$-robustness framework.} We note that a distinct class of Byzantine robust algorithms leverage trust score instead of robust aggregation rules to discriminate malicious updates from honest ones \cite{xie2019zeno, Molodtsov2026bant}, which falls out of the scope of Assumption \ref{aggregation}. However, this kind of methods typical
    relies on availability of a clean auxiliary at the server. Since the existence of such trail dataset is often unavailable in practice, we focus on the $(b,\kappa)$-robustness framework, which contains the most commonly studied and theoretically grounded approaches in the ``blind'' server setting.

\section{Convergence Analysis of R-DSGD and R-DSGD-M}\label{sec:alg:convergence}
In this section, we present the convergence guarantees for R-DSGD and R-DSGD-M. We analyze these algorithms equipped with $(b,\kappa)$-robust aggregation rules under the general heterogeneity assumption of $(G,B)$-bounded dissimilarity. The big-$O$ notation $\mo(\cdot)$ in our theoretical guarantees omits the possible dependence on $L$ and $\mu$, as we aim to highlight the effects of $\kappa$, $G$, and $B$. The complete details are provided in the Appendices A and B.

\subsection{Analysis of R-DSGD}
We begin by establishing theoretical guarantees for R-DSGD on general nonconvex and smooth functions, followed by an improved analysis under the additional PL condition. The full proofs are provided in Appendix \ref{Proof of the Convergence Result for R-DSGD}.
    \begin{theorem}[\textbf{Nonconvex case}]\label{convergence for nonconvex R-DSGD}
        Suppose that Assumptions \ref{smoothness}, \ref{heterogeneity}-\ref{aggregation} hold, and let $0<\delta<1/4$ be any fixed constant. If $\kappa B^2<(1-4\delta)/(1+2\delta)$ and the stepsizes satisfy $\gamma_t\equiv \gamma \leq \delta/L$, then for any $T>0$, the iterates $\{x^t\}_{t=0}^{T-1}$ generated by R-DSGD satisfy
        \begin{align*}
            &\frac1T\sum_{t=0}^{T-1}\ep\left[\norm{\nabla f_{\mh}(x^{t})}^2\right]\\
            &\leq \frac{\Delta_0}{C_1\gamma T}+\frac{2\gamma L\sigma^2}{C_1h }
            +C_2\left(\kappa\sigma^2+\kappa G^2\right),
        \end{align*}
        where $\Delta_0=f_{\mh}(x^0)-f_{\mh}^*$, $
            C_1 =\frac12-2\delta-\kappa B^2\left(\frac12+\delta\right)=\mo\left(1-\kappa B^2\right) $,  and $C_2=\left(\frac12+\delta\right)C_1^{-1}$.
    \end{theorem}
    Theorem \ref{convergence for nonconvex R-DSGD} establishes a general bound on the average squared gradient norm. By selecting a small constant stepsize, we derive the explicit convergence rate below.
    \begin{corollary}\label{coro:r-dsgd:nonconvex}
        Under the conditions of Theorem \ref{convergence for nonconvex R-DSGD}, for any $T\geq 1$, let $\gamma = \gamma_0/\sqrt{T}$ where $\gamma_0 \leq \delta/L$.
        
        Then the iterates $\{x^t\}_{t=0}^{T-1}$ generated by R-DSGD satisfy
        \begin{align*}
            &\frac1T\sum_{t=0}^{T-1}\ep\left[\norm{\nabla f_{\mh}(x^{t})}^2\right]\\
            &\leq\mo\biggl(\frac{\Delta_0+\sigma^2}{(1-\kappa B^2)\sqrt{T}}+\frac{\kappa G^2}{1-\kappa B^2}+\frac{\kappa\sigma^2}{1-\kappa B^2}\biggr).
        \end{align*}
    \end{corollary}
    Corollary \ref{coro:r-dsgd:nonconvex} demonstrates that R-DSGD achieves a convergence rate of $\mathcal{O}(1/\sqrt{T})$ for nonconvex smooth objectives. However, the inaccuracy of the $(b,\kappa)$-robust aggregation rule introduces a \emph{non-vanishing Byzantine error} of order:
    \begin{equation}\label{robustness lower bound for nonconvex}
        \mo\left(\frac{\kappa\sigma^2}{1-\kappa B^2} + \frac{\kappa G^2}{1-\kappa B^2}\right).
    \end{equation}
    This term reflects the irreducible error caused by Byzantine tolerance, scaling with the heterogeneity parameters $G$ and $B$, and the noise variance $\sigma^2$.
        
        We next examine the behavior of the algorithm when the global objective satisfies the PL condition. The formal description is provided in Theorem~\ref{convergence for PL R-DSGD specific}.
        \begin{theorem}[\textbf{PL case}]\label{convergence for PL R-DSGD}
        Suppose that Assumptions \ref{smoothness}-\ref{aggregation} hold. Given $T>2$, assume that the stepsizes are chosen as 
        \begin{equation}\label{eq:stepsize}
            \gamma_t=\begin{cases}
            \gamma_0, & \mathrm{if}\;t<\lfloor T/2\rfloor,\\
            \mo(\gamma_0/t), & \mathrm{if}\;t\geq \lfloor T/2\rfloor,
        \end{cases}
        \end{equation}
        where $\gamma_0$ is the initial stepsize.
        Then the last iterate $x^{T-1}$ generated by R-DSGD satisfies
        \begin{align*}
            &\ep\left[f_{\mh}(x^{T-1})\right]-f_{\mh}^*  \leq\mo\biggl(
            \frac{\Delta_0\exp(-\mu(1-\kappa B^2)T)}{T^2}\\&\;+\frac{(1+\kappa)\sigma^2+\kappa G^2}{(1-\kappa B^2)^2T}+\frac{\kappa G^2}{\mu(1-\kappa B^2)}+\frac{\kappa \sigma^2}{\mu(1-\kappa B^2)}\biggr).
        \end{align*}
    \end{theorem}

    Theorem \ref{convergence for PL R-DSGD} demonstrates that under the PL condition, the convergence rate of R-DSGD improves to $\mathcal{O}(1/T)$, surpassing the rate established in Corollary~\ref{coro:r-dsgd:nonconvex}. Nevertheless, similar to the general nonconvex setting, the function value remains subject to an unavoidable error proportional to\begin{equation}\label{robustness lower bound for pl}
        \mo\left(\frac{\kappa G^2}{\mu(1-\kappa B^2)} + \frac{\kappa \sigma^2}{\mu(1-\kappa B^2)}\right).
    \end{equation}
    
    \subsection{Analysis of R-DSGD with Momentum}
 We now turn our attention to R-DSGD with local momentum (R-DSGD-M). Theorem \ref{convergence for R-DSGD with momentum} addresses the nonconvex smooth case, while Theorem \ref{convergence for PL R-DSGD-M} covers the additional PL condition. Detailed proofs are available in Appendix \ref{Proof for Convergence Result of R-DSGD with Momentum}.
    \begin{theorem}[\textbf{Nonconvex case}]\label{convergence for R-DSGD with momentum}
        Suppose that Assumptions \ref{smoothness}, \ref{heterogeneity}-\ref{aggregation} hold, and assume that $\kappa B^2<\mo(1)$.  If the stepsizes and momentum parameters are chosen as $\gamma_t\equiv\gamma \leq 1/(32L)$ and $\beta_t\equiv\beta=1-32\gamma L$, then for any $T>0$, the iterates $\{x^t\}_{t=0}^{T-1}$ generated by R-DSGD-M satisfy
        \begin{align*}
            &\frac1T\sum_{t=0}^{T-1}\ep\left[\norm{\nabla f_{\mh}(x^t)}^2\right]\\
            &\leq \mo\left(\frac{\Delta_0}{(1-\kappa B^2)\gamma T}+\frac{\gamma (1+\kappa) \sigma^2}{1-\kappa B^2}+\frac{\kappa G^2}{1-\kappa B^2}\right).
        \end{align*}
    \end{theorem}
    \begin{corollary}\label{coro:r-dsgd-m:convergence}
        Under the conditions of Theorem \ref{convergence for R-DSGD with momentum},  for any $T\geq 1$, choosing $\gamma=\gamma_0/\sqrt{T}$ with $\gamma_0\leq 1/(32L)$ yields
        \begin{align*}
            &\frac1T\!\sum_{t=0}^{T-1}\!\ep\!\left[\norm{\nabla f_{\mh}(x^t)}^2\right]\!\leq\!\mo\!\left(\!\frac{\Delta_0\!+\!(1+\kappa)\sigma^2}{(1-\kappa B^2)\sqrt{T}}\!+\!\frac{\kappa G^2}{1-\kappa B^2}\!\right).
        \end{align*}
    \end{corollary}
    \begin{theorem}[\textbf{PL case}]\label{convergence for PL R-DSGD-M}
        Suppose that Assumptions~\ref{smoothness}-\ref{aggregation} hold, and assume that $\kappa B^2<\mo(1)$. Given any $T>2$, assume that the stepsizes are chosen as in ~\eqref{eq:stepsize} and the momentum parameters are chosen as $\beta_t=1-32\gamma_tL$. Then for any $T>2$, the last iterate $x^{T-1}$ generated by R-DSGD-M satisfies
        \begin{align*}
           \ep\left[f_{\mh}(x^{T-1})\right]\!-\!f_{\mh}^*  &\leq\mo\biggl(\frac{\Delta_0\exp(-\mu(1-\kappa B^2)T)}{T^2}\\&\quad+\!\frac{(1+\kappa)\sigma^2}{(1-\kappa B^2)^2T}\!+\!\frac{\kappa G^2}{\mu(1-\kappa B^2)}\biggr).
        \end{align*}
    \end{theorem}

    \textbf{Comparison with R-DSGD.} Theorems \ref{convergence for R-DSGD with momentum} and \ref{convergence for PL R-DSGD-M} indicate that while the local momentum does not improve the asymptotic convergence rates ($\mathcal{O}(1/\sqrt{T})$ and $\mathcal{O}(1/T)$ respectively), it significantly mitigates the Byzantine error. Specifically, the error floors reduce to 
    \begin{equation}\label{robustness lower bound for nonconvex for momentum}
        \mo\left(\frac{\kappa B^2}{1-\kappa G^2}\right)\; \quad \text{and}\;  \quad \mo\left(\frac{\kappa B^2}{\mu(1-\kappa G^2)}\right),
    \end{equation}
    for general nonconvex and PL objectives, respectively. Note that, unlike the bounds for R-DSGD in \eqref{robustness lower bound for nonconvex} and \eqref{robustness lower bound for pl}, these terms \textit{do not depend on the stochastic noise variance} $\sigma^2$. This suggests that local momentum can effectively dampen the impact of stochastic noise on the solution's accuracy.
    \subsection{Discussions and General Remarks}
    \textbf{The most general analysis.} In the special case where $B=0$, our analysis of R-DSGD in Theorem \ref{convergence for nonconvex R-DSGD} and R-DSGD-M in Theorem \ref{convergence for R-DSGD with momentum} recover the results of \citet{allouah2023fixing} and \citet{guerraoui2024robust}. In addition, when $\sigma^2=0$, using a constant stepsize $\gamma=\mo(1/T)$ in Theorem \ref{convergence for nonconvex R-DSGD} yields an improved $\mo(1/T)$ rate, matching the deterministic result of \citet{gupta2025reconciling}.
    
  However, our results constitute a nontrivial extension of these works. Since both data heterogeneity and stochastic noise contribute to discrepancies among honest updates, decoupling their effects is crucial—particularly when analyzing local momentum. While prior studies \cite{allouah2023fixing, guerraoui2024robust} address this under the simpler $G$-bounded assumption (where heterogeneity is constant), we tackle the more complex $(G,B)$-bounded dissimilarity setting. In this setting, heterogeneity scales with $\norm{\nabla f_{\mh}(x)}^2$, intrinsically coupling it with the optimization trajectory. This dependency is further complicated by the iteration-dependent stepsizes in the PL case. These make the analysis more challenging and prevent the direct application of existing frameworks tailored to constant heterogeneity.
    
    \textbf{No-attack accuracy.}
    In the absence of Byzantine attacks,  \citet{yang2025tension} have characterized the behavior of an aggregation rule  via $\hat{\kappa}$-accuracy:
    \[\norm{\ma(x_1,x_2,\ldots,x_n)-\bar{x}}^2\leq\frac{\hat{\kappa}}{n}\sum_{i=1}^n\norm{x_i-\bar{x}}^2,\]
    where $\bar{x}=\frac{1}{n}\sum_{i=1}^nx_i$. They show that any $(b,\kappa)$-robust aggregation rule is $\kappa$-accurate with $\kappa\geq\frac{b}{n-b}$. This highlights an inherent trade-off: robust aggregation rules introduce a bias compared to simple averaging--even when all workers are honest--as the necessary cost of their defensive capability. 

Our work extends this analysis in several key aspects. While  \citet{yang2025tension} have focused on deterministic gradients (R-DGD) with the $(G,0)$-bounded dissimilarity, we address the more challenging stochastic setting (R-DSGD and R-DSGD-M) under the general $(G,B)$-bounded dissimilarity. Our unified bounds are comprehensive, seamlessly recovering no-attack results by simply substituting the robustness parameter $\kappa$ with the accuracy parameter $\hat{\kappa}$.

    \textbf{The condition $\kappa B^2<\mo(1)$.} The requirement that $\kappa B^2$ remains bounded by a constant (also appearing in the work of \citet{allouah2023robust} for R-DGD) admits an intuitive interpretation. 
    If the data is highly heterogeneous (i.e., large $B^2$), the algorithm requires a more accurate aggregation rule (i.e., smaller $\kappa$) to maintain convergence. Conversely, if the aggregation rule is inaccurate (i.e., large $\kappa$), the algorithm can only tolerate limited heterogeneity.

    \section{Lower Bounds for Byzantine Errors}
    \begin{figure*}[!t]
        \centering
        \begin{subfigure}[b]{0.33\textwidth}
        \centering
        \includegraphics[width=\textwidth]{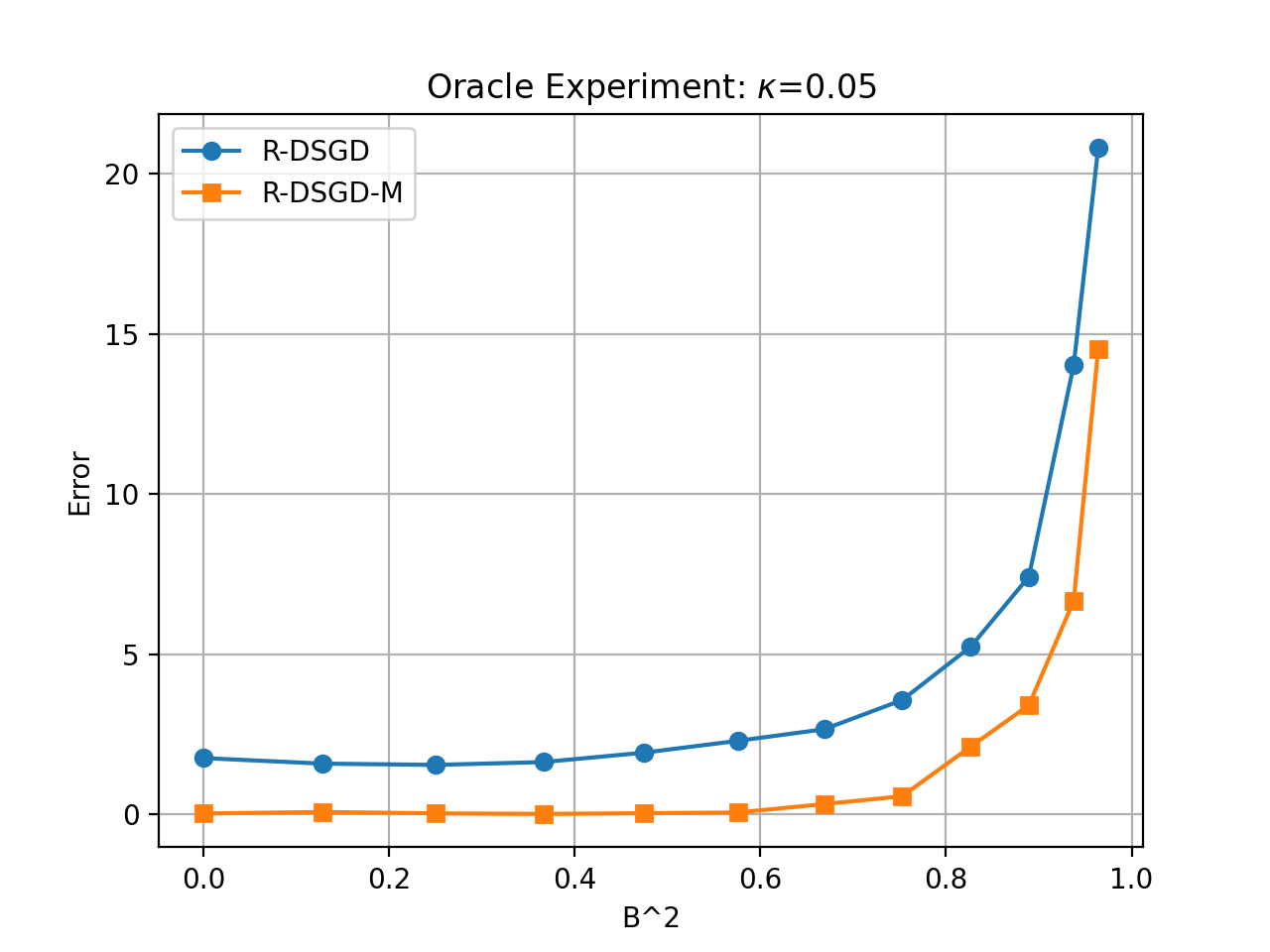}
    \end{subfigure}
    \hfill %
    \begin{subfigure}[b]{0.33\textwidth}
        \centering
        \includegraphics[width=\textwidth]{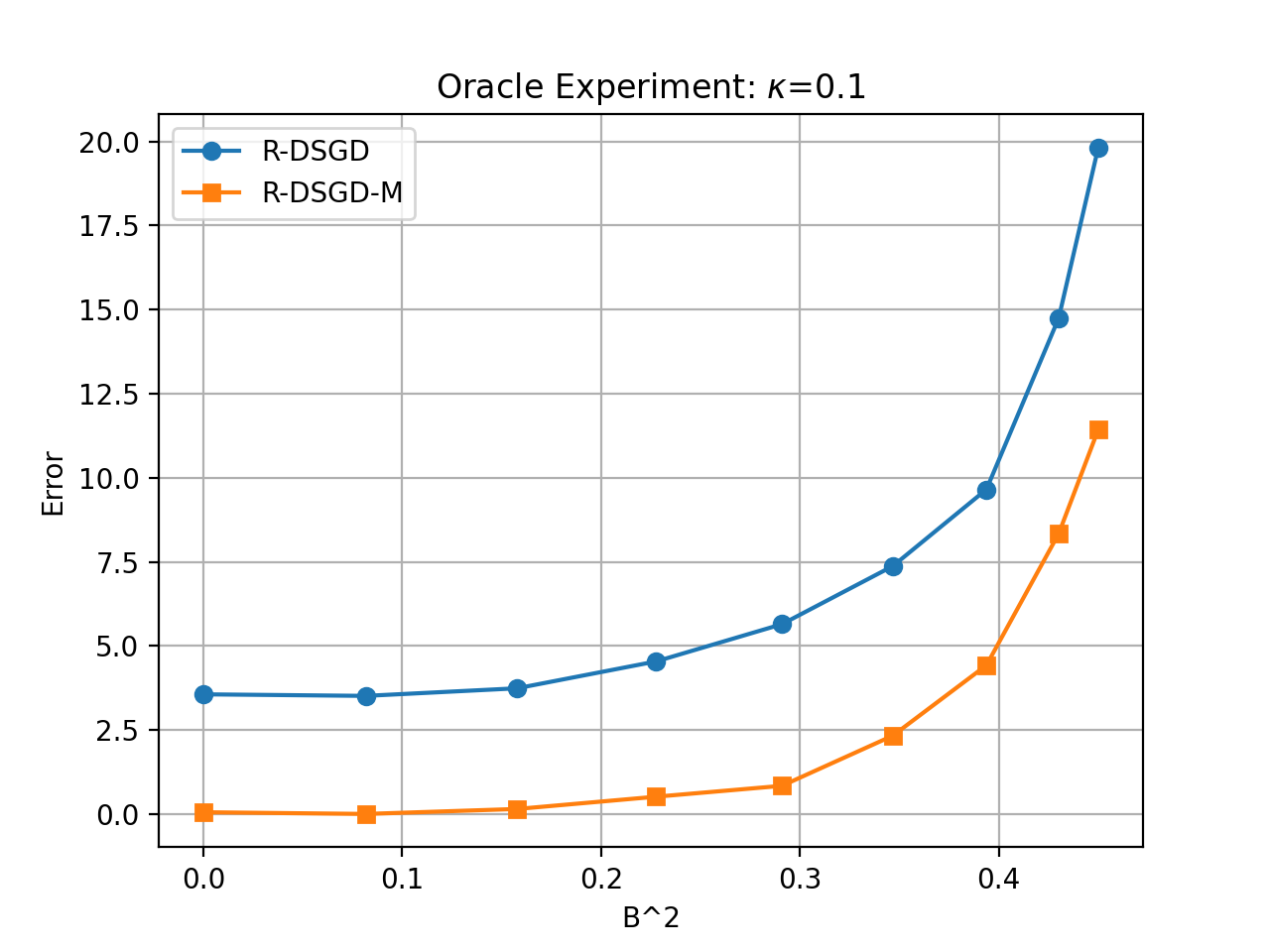}
    \end{subfigure}
    \hfill %
    \begin{subfigure}[b]{0.33\textwidth}
        \centering
        \includegraphics[width=\textwidth]{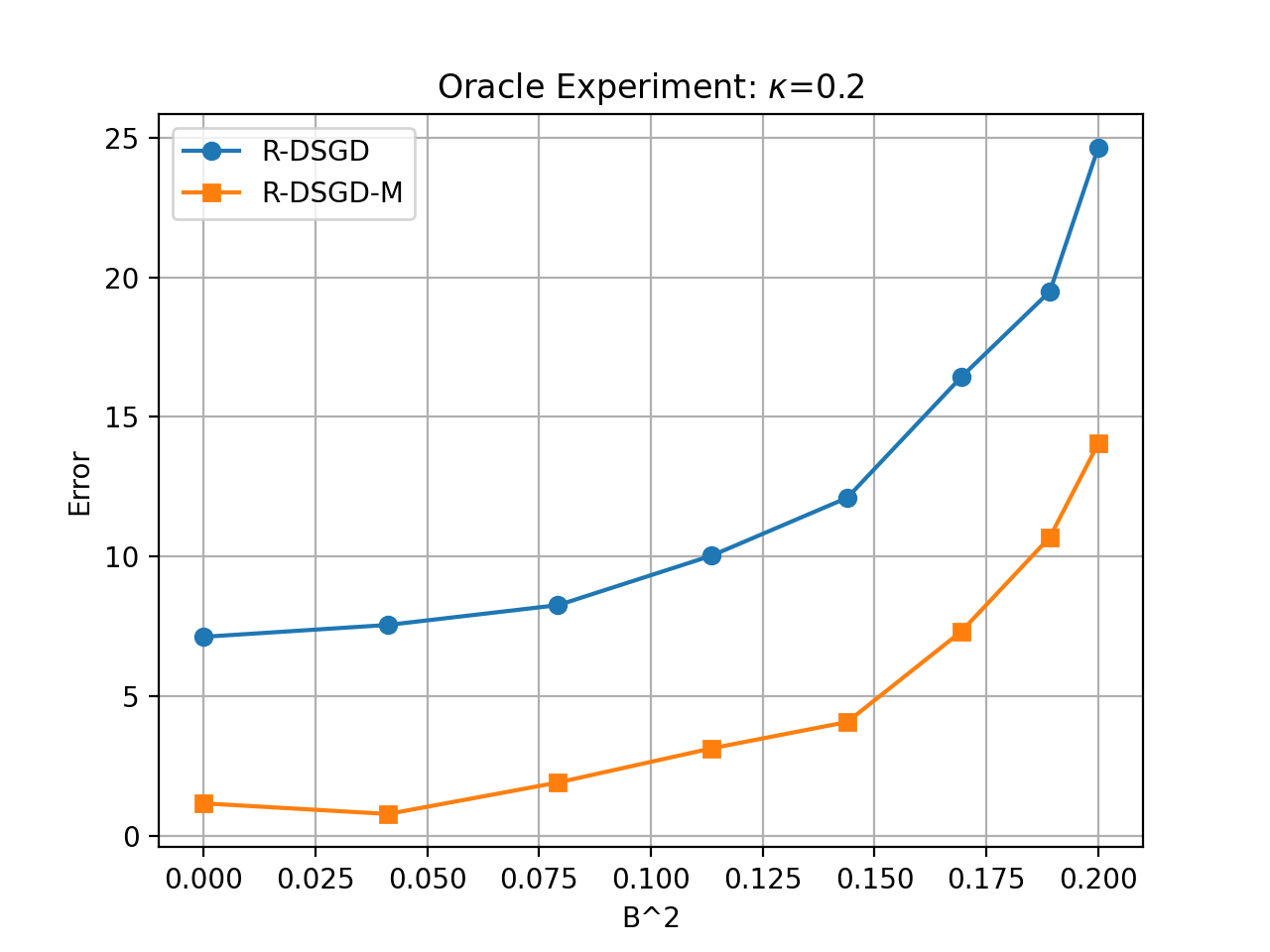}
    \end{subfigure}
    \caption{The performance of the oracle aggregation rules when increasing $B^2$ under three choices of the parameter $\kappa$: $\kappa=0.05$ (left), $\kappa=0.1$ (middle), and $\kappa=0.2$ (right).}
    \label{synthetic experiment}
    \end{figure*}
    In this section, we establish that the Byzantine error upper bounds derived in the previous section for R-DSGD (\ref{robustness lower bound for nonconvex}) and R-DSGD-M (\ref{robustness lower bound for nonconvex for momentum}) are tight. The key strategy is to construct an adversarial aggregation rule that introduces a controlled error—bounded by the $(b,\kappa)$-robustness limit—at every iteration. As the iterations progress, this error accumulates to induce a systematic drift, effectively forcing the algorithm to optimize a skewed objective and resulting in a non-vanishing Byzantine error. Detailed proofs are available in Appendix~\ref{sec: Proof of the Lower Bounds}.

\subsection{The Impact of Data Heterogeneity}
    We first address the tightness of the Byzantine error bound $\mo\left(\frac{\kappa G^2}{1-\kappa B^2}\right)$ arising from data heterogeneity. Since deterministic gradients represent a special case of stochastic gradients (where the variance is zero), it suffices to analyze the lower bound for the deterministic R-DGD-(M) algorithm.
    \begin{theorem}\label{thm for lower bound for heterogeneity}
        Given constants $G,B,\kappa\geq 0$ such that $\kappa B^2<1$, there exists an optimization problem where (i) the objective function is smooth and satisfies the PL condition with a common parameter $\mu>0$, and (ii) the honest local objective functions satisfy the $(G,B)$-bounded dissimilarity, along with a specific $(b,\kappa)$-aggregation rule, such that: if the number of the iteration $T$ is sufficiently large, the last iterate $x^{T-1}$ generated by R-DGD-(M) with any choice of stepsizes and momentum parameters considered in Section \ref{sec:alg:convergence} satisfies
        \begin{align*}
            \norm{\nabla f_{\mh}(x^{T-1})}^2&\geq\mo\left(\frac{\kappa G^2}{1-\kappa B^2}\right),\\
            f_{\mh}(x^{T-1})-f_{\mh}(x^*)&\geq\mo\left(\frac{\kappa G^2}{\mu(1-\kappa B^2)}\right).
        \end{align*}
    \end{theorem}
    
    \textbf{Comparison to previous lower bounds.} \citet{shi2025optimal} have established a lower bound of $\mo\left(\kappa G^2\right)$ for the $(G,0)$-bounded dissimilarity but left the general $(G,B)$-bounded dissimilarity case unexplored. In the $(G,B)$ setting,
\citet{allouah2023robust} have derived a minimum error of $\mo\left(\frac{b}{n-(2+B^2)b}G^2\right)$  for any distributed algorithm with $b$ Byzantine workers.  Theorem \ref{thm for lower bound for heterogeneity} offers a strictly stronger characterization:  it confirms the tightness of the error term $\mo\left(\frac{\kappa G^2}{1-\kappa B^2}\right)$ for any value of $\kappa$. This is tighter than that of \citet{allouah2023robust} whenever $\kappa\gg \frac{b}{n-2b}$, highlighting the necessity of our $\kappa$-dependent analysis.
    
\subsection{The Impact of Stochastic Noise}
    Next, we demonstrate that the Byzantine error $\mo\left(\frac{\kappa\sigma^2}{1-\kappa B^2}\right)$ stemming from stochastic noise is theoretically unavoidable for R-DSGD. To establish this lower bound, it suffices to consider the case where $G=0$.
    \begin{theorem}\label{thm for lower bound for randomness}
       Given constants $B,\kappa\geq 0$ such that $\kappa B^2<1$, there exists an optimization problem where (i) the objective function is smooth and satisfies PL condition with a common parameter $\mu>0$, and (ii) the honest local objective functions satisfy the $(0,B)$-bounded dissimilarity, along with a specific $(b,\kappa)$-aggregation rule, such that: if the number of the iteration $T$ is sufficiently large, the last iterate $x^{T-1}$ generated by D-RSGD 
       with the  stepsizes in Theorem \ref{convergence for nonconvex R-DSGD} or Theorem \ref{convergence for PL R-DSGD} satisfies
        \begin{align*}
            \ep\left[\norm{\nabla f_{\mh}(x^{T-1})}^2\right]&\geq \mo\left(\frac{\kappa\sigma^2}{1-\kappa B^2}\right), \\
            \ep\left[f_{\mh}(x^{T-1})\right]-f_{\mh}(x^*)&\geq\mo\left(\frac{\kappa \sigma^2}{\mu(1-\kappa B^2)}\right).
        \end{align*}
    \end{theorem}
    
    \textbf{Comparison to previous lower bounds.} \citet{karimireddy2021learning} have shown that any permutation invariant method (including R-DSGD) is subject to an error of $\mo\left(\frac{b}{n}\sigma^2\right).$ Theorem \ref{thm for lower bound for randomness} generalizes this limitation to the setting of the $(0,B)$-bounded dissimilarity.  Our lower bound not only validates the tightness of the noise-induced error term $\mo\left(\frac{\kappa\sigma^2}{1-\kappa B^2}\right)$ in the convergence analysis in Section \ref{sec:alg:convergence} but also reveals a critical insight: when $\kappa \gg \mo(b/n)$, the fundamental limit is dictated by the robustness parameter $\kappa$ rather than simply the fraction $b/n$ of Byzantine users. This provides a more precise characterization of robustness in stochastic settings.

    \begin{figure*}[!t]
    \centering
    \begin{subfigure}[b]{0.49\textwidth}
        \centering
        \includegraphics[width=1.05\textwidth]{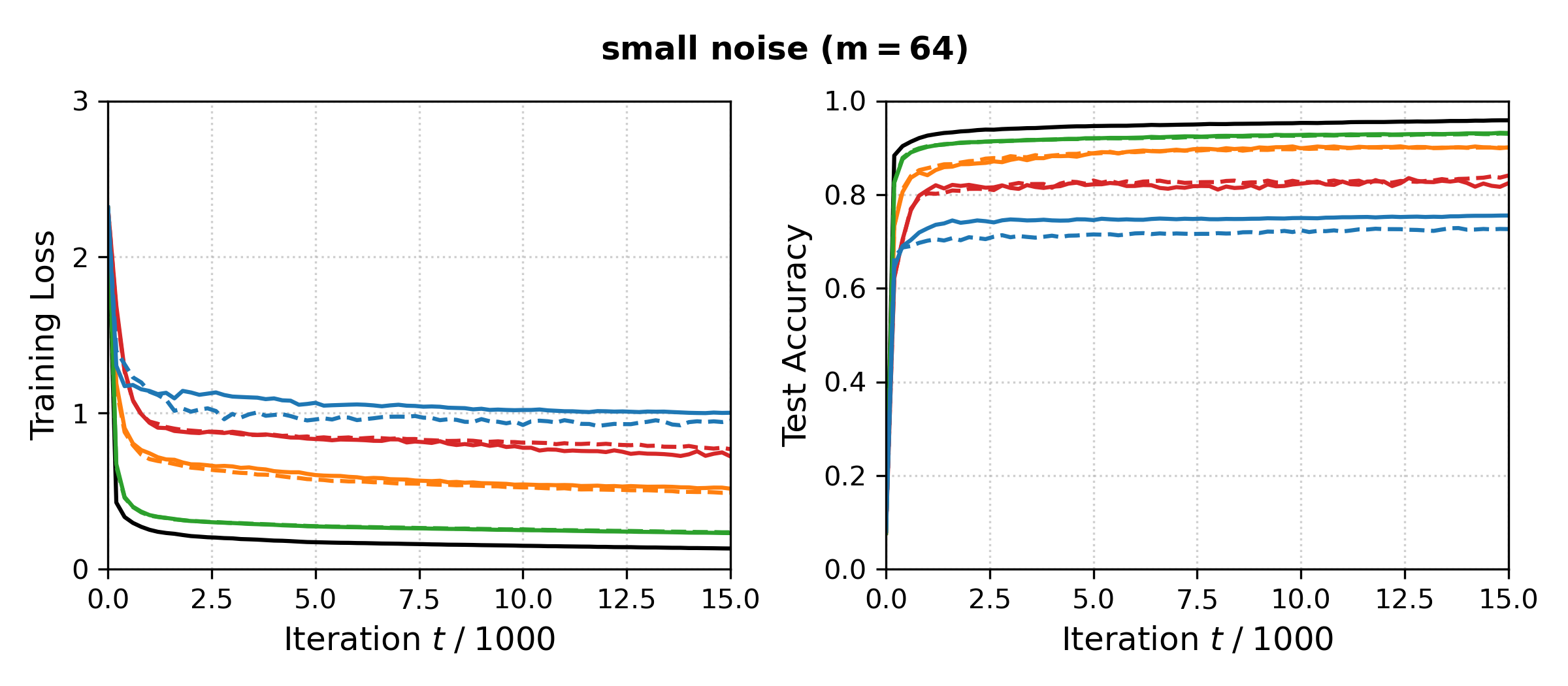}
    \end{subfigure}
    \hfill 
    \begin{subfigure}[b]{0.49\textwidth}
        \centering
        \includegraphics[width=1.05\textwidth]{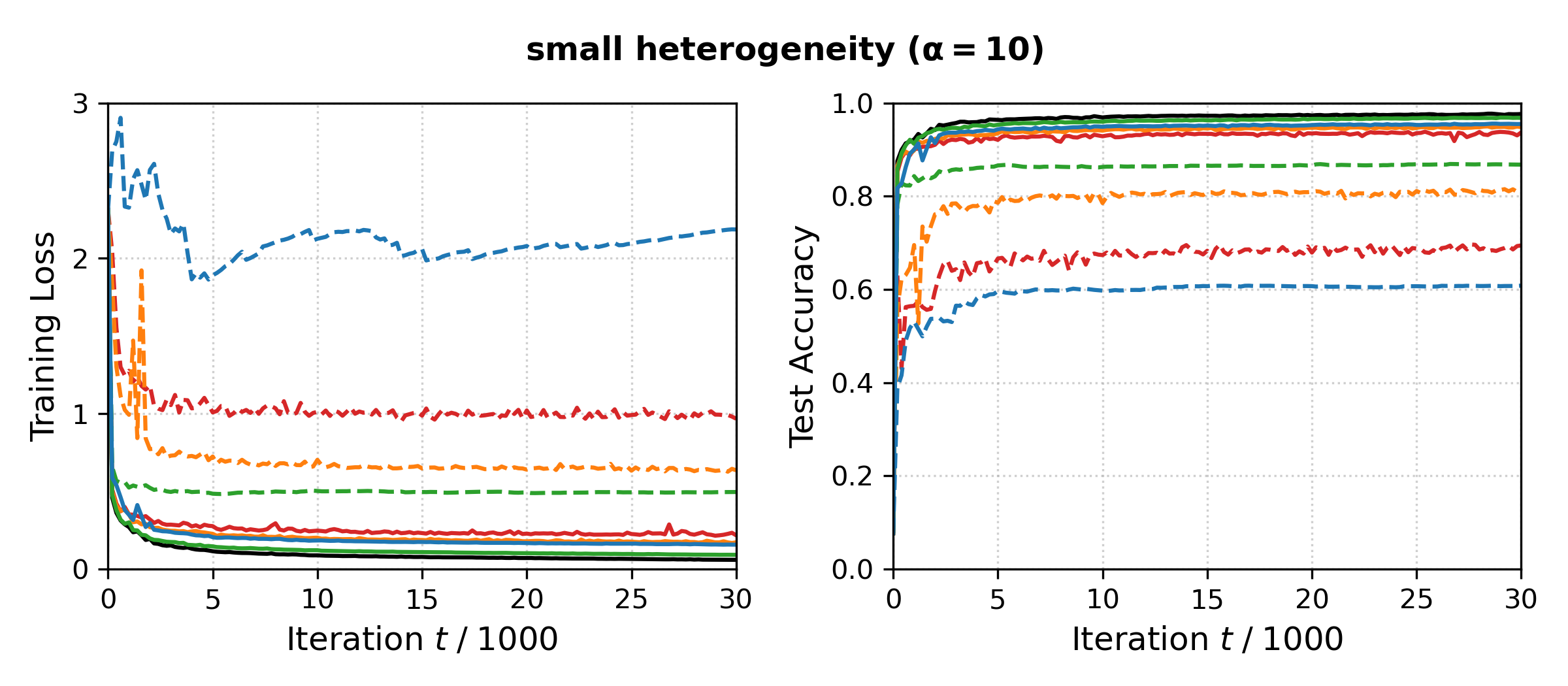}
    \end{subfigure}

    \begin{subfigure}[b]{0.6\textwidth}
        \centering
        \includegraphics[width=\textwidth]{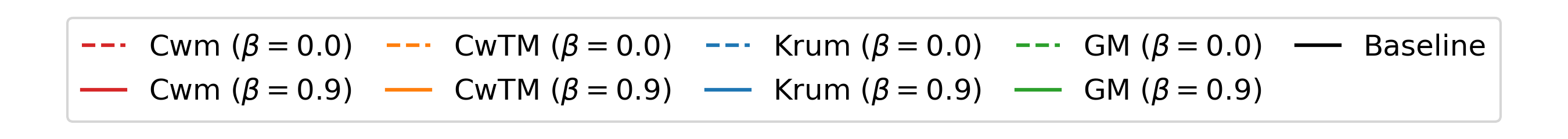}
    \end{subfigure}
    \caption{Performance comparison of R-DSGD (dashed lines) and R-DSGD-M (solid lines) on training MLP in the MNIST dataset using various robust aggregation rules in the absence of Byzantine attacks. Left: Settings with large batch size ($m=64$) and high data heterogeneity ($\alpha=0.1$). Right: Settings with small batch size ($m=1$) and low data heterogeneity ($\alpha=10$).}
    \label{no attack performance}
    
    \centering
    \begin{subfigure}[b]{0.49\textwidth}
        \centering
        \includegraphics[width=1.05\textwidth]{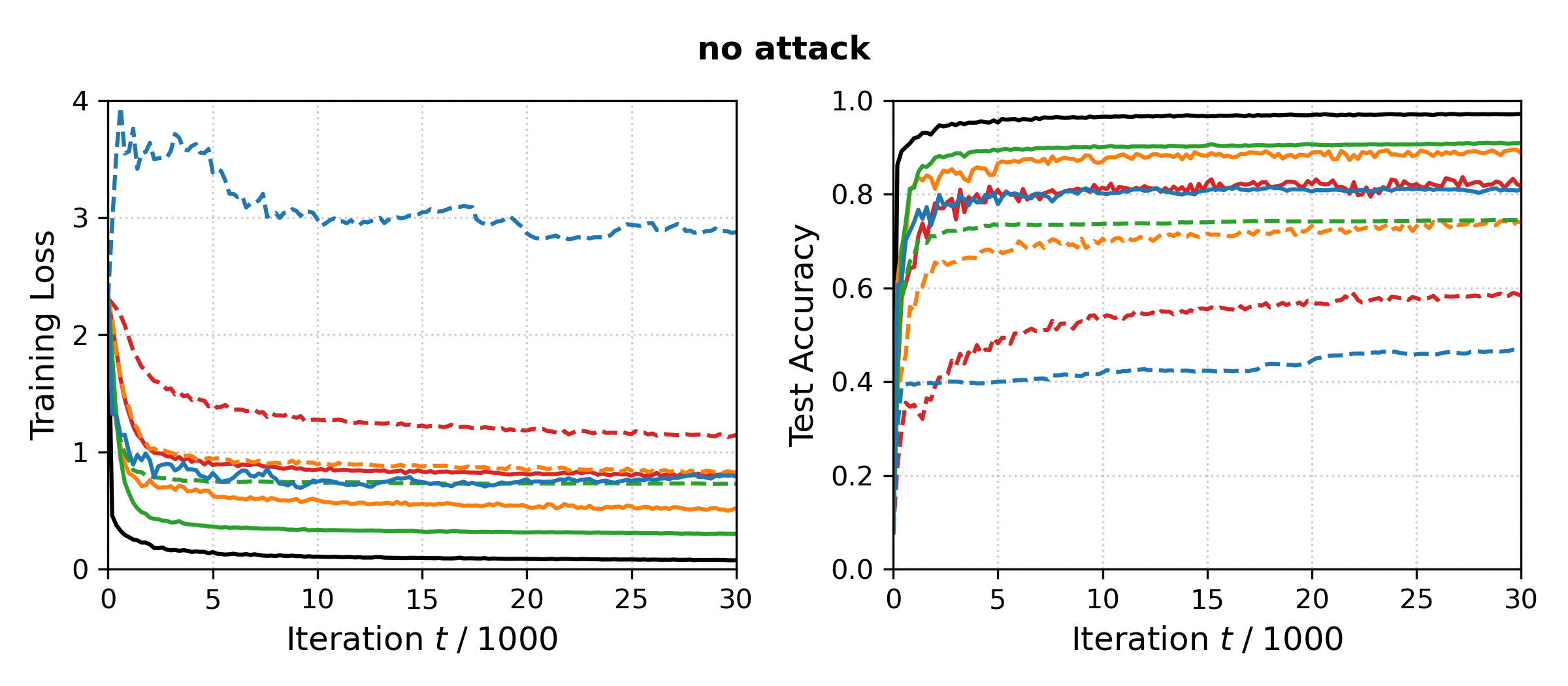}
    \end{subfigure}
    \hfill %
    \begin{subfigure}[b]{0.49\textwidth}
        \centering
        \includegraphics[width=1.05\textwidth]{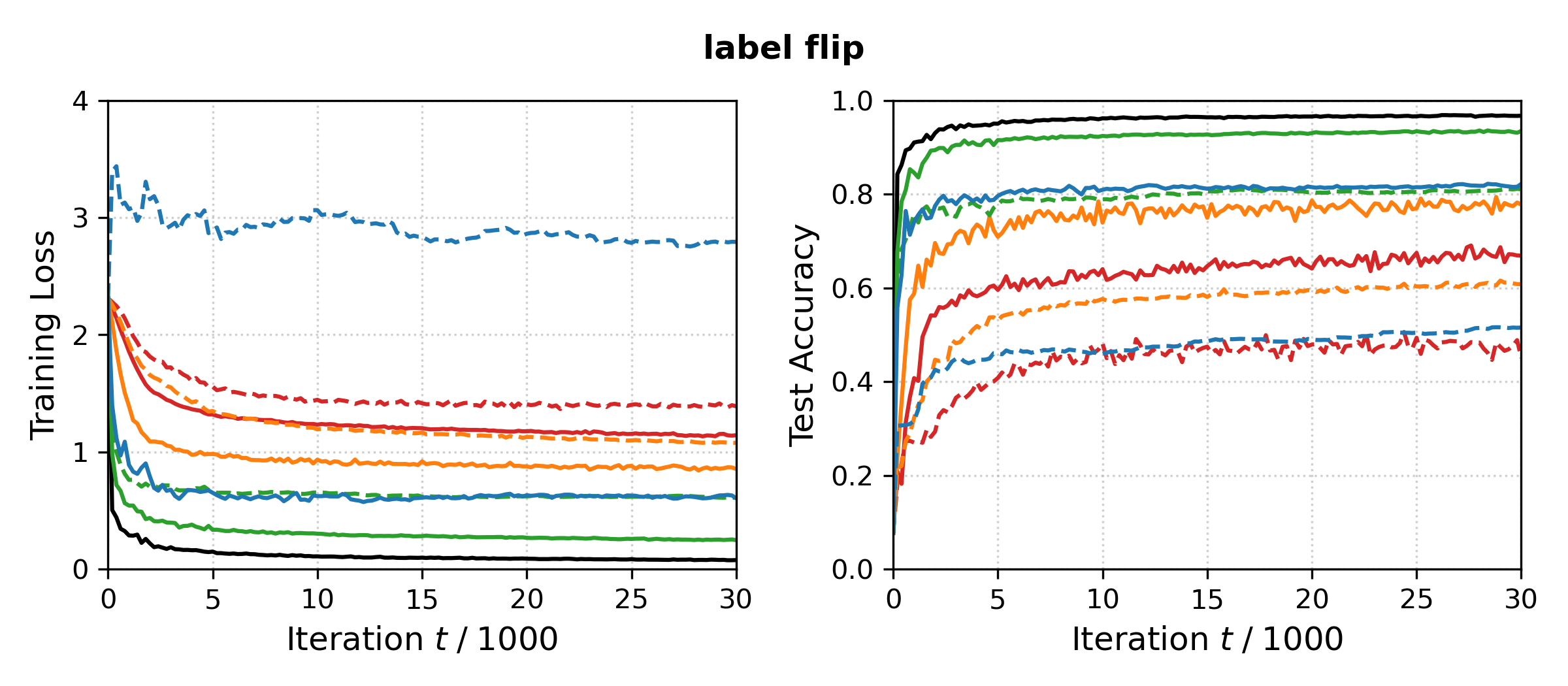}
    \end{subfigure}

    \begin{subfigure}[b]{0.49\textwidth}
        \centering
        \includegraphics[width=1.05\textwidth]{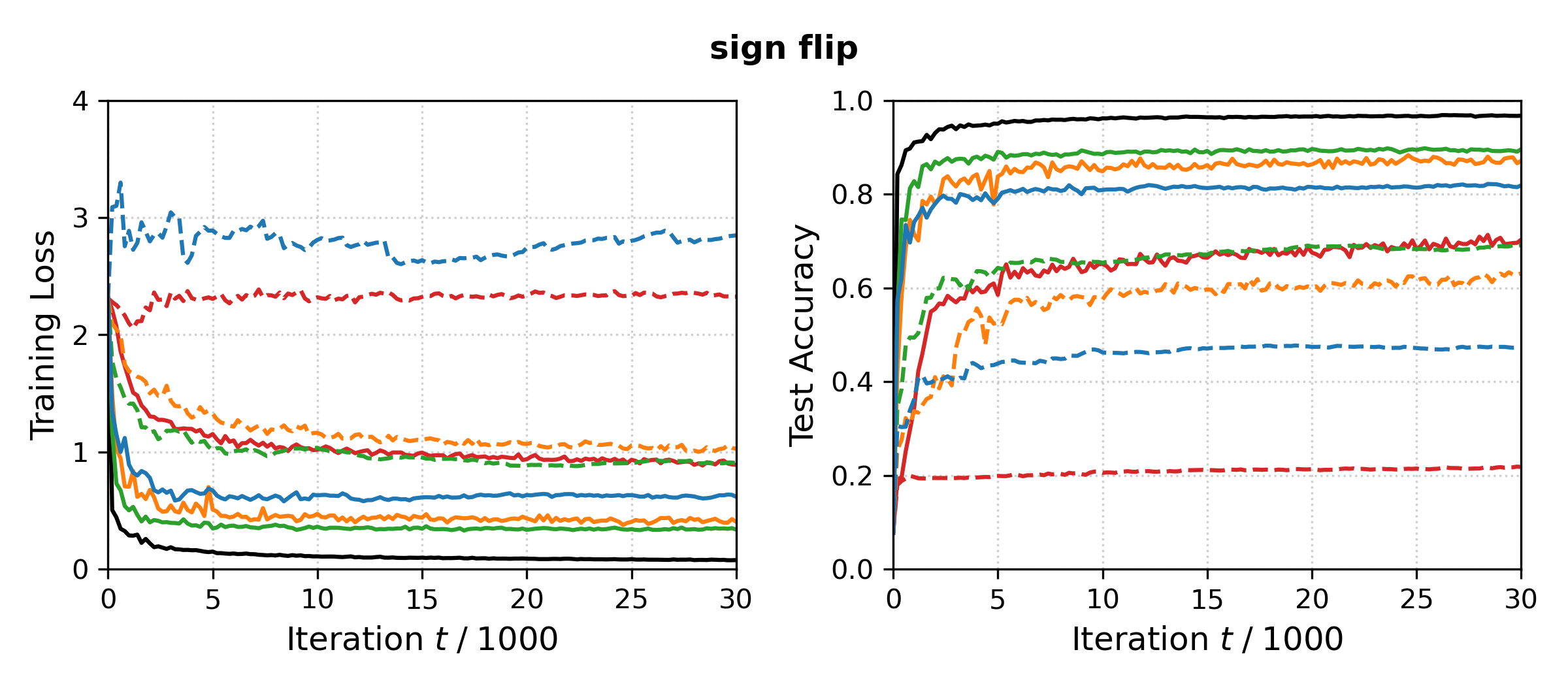}
    \end{subfigure}
    \hfill
    \begin{subfigure}[b]{0.49\textwidth}
        \centering
        \includegraphics[width=1.05\textwidth]{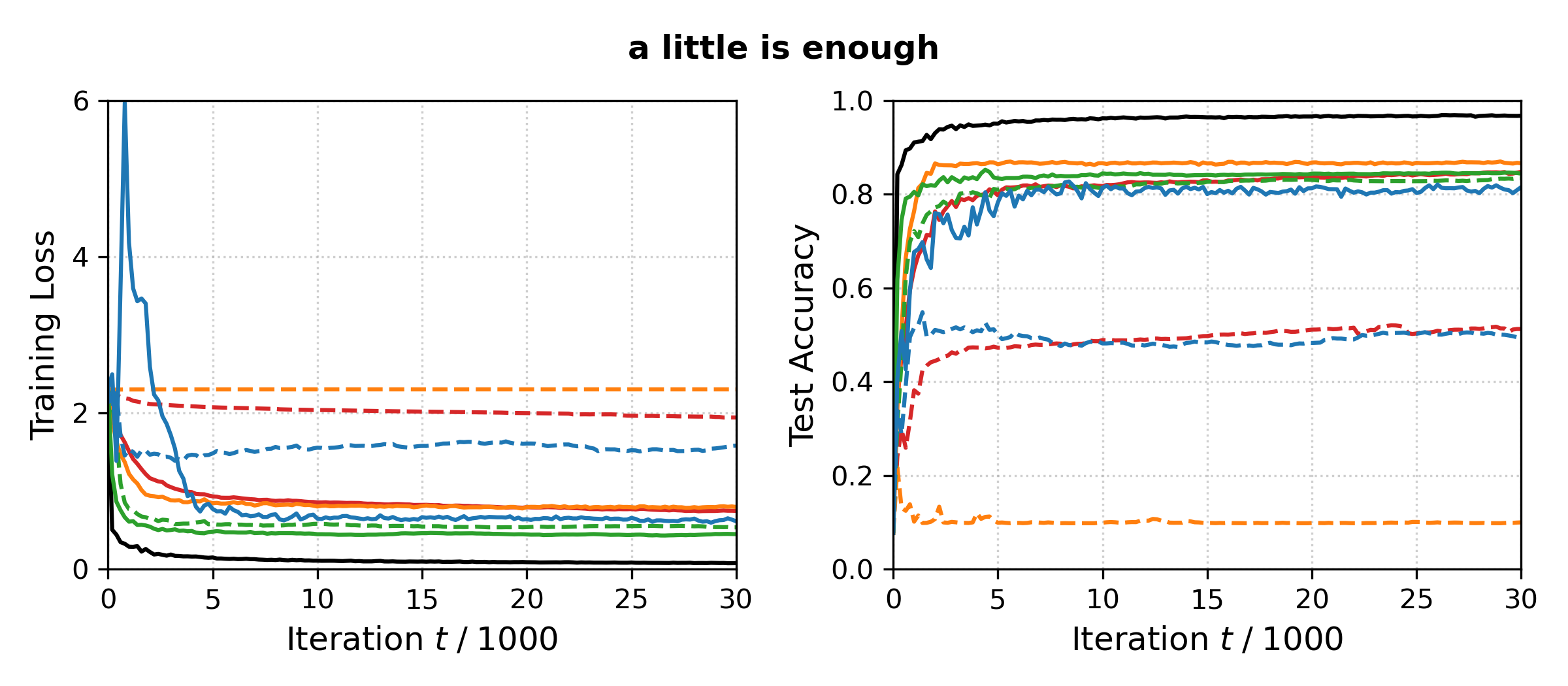}
    \end{subfigure}
    \begin{subfigure}[b]{0.6\textwidth}
        \centering
        \includegraphics[width=\textwidth]{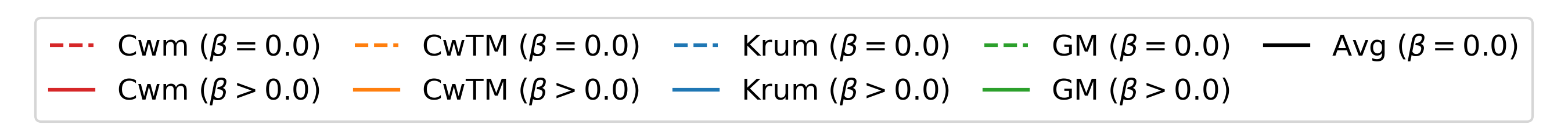}
    \end{subfigure}
    \caption{Performance comparison of R-DSGD (dashed lines) and R-DSGD-M (solid lines) on training MLP in the MNIST dataset with various robust aggregation rules under four scenarios: no attack (top-left), label flip attack (top-right), sign flip attack (bottom-left), and ``a little is enough" attack (bottom-right). All experiments use a batch size of $m=1$ and a Dirichlet distribution parameter of $\alpha=0.1$, simulating an environment with large stochastic noise and high data heterogeneity.}
    \label{attack performance}
    
    \end{figure*}
    
\section{Experiments}
\begin{figure*}[!t]
    \centering
    \begin{subfigure}[b]{0.49\textwidth}
        \centering
        \includegraphics[width=\textwidth]{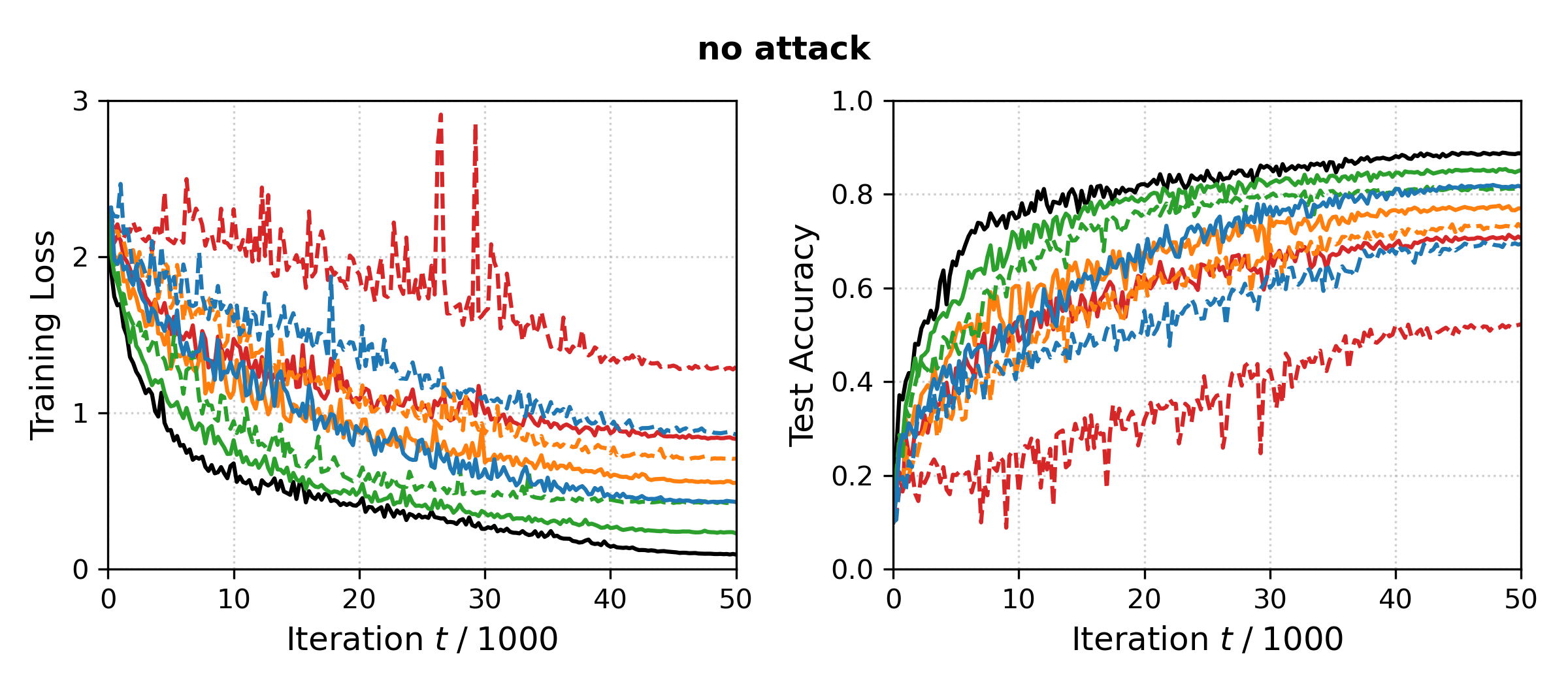}
    \end{subfigure}
    \hfill %
    \begin{subfigure}[b]{0.49\textwidth}
        \centering
        \includegraphics[width=\textwidth]{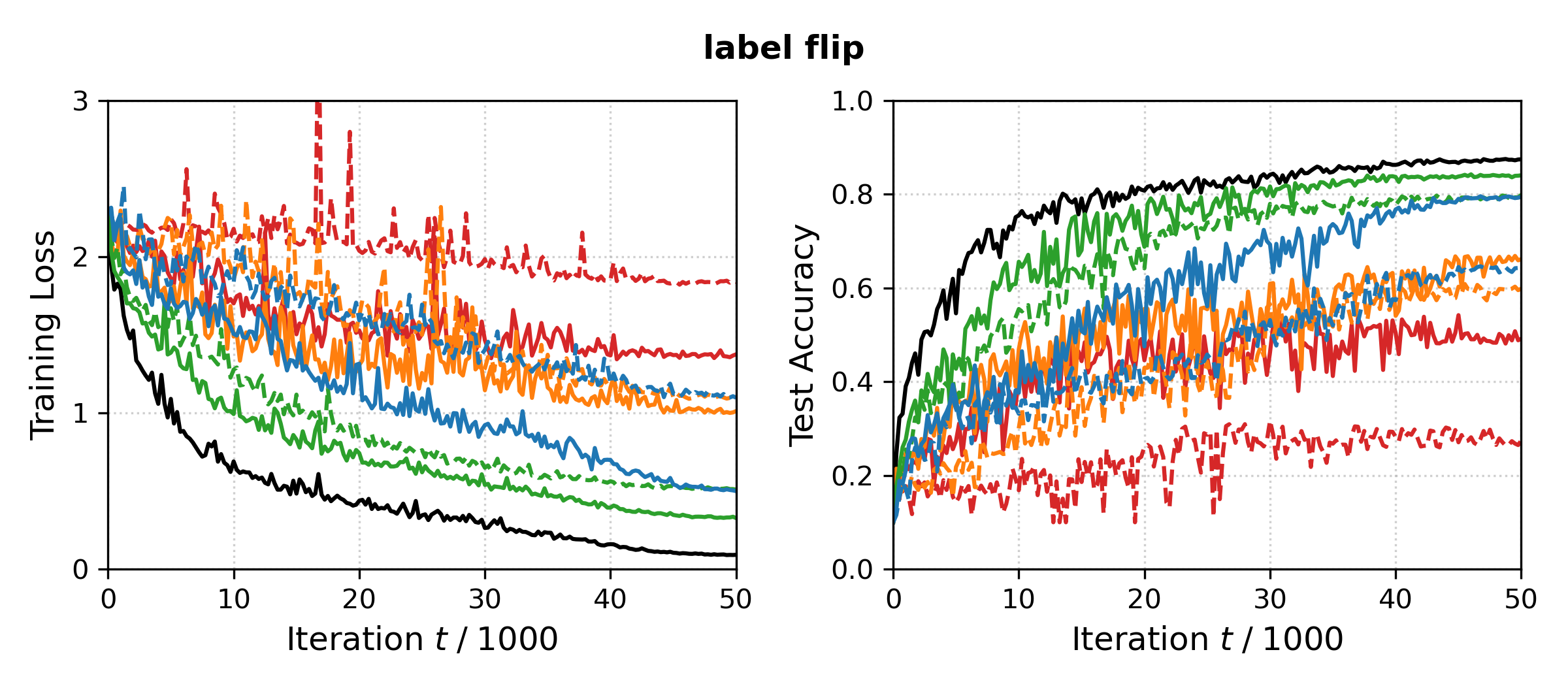}
    \end{subfigure}

    \begin{subfigure}[b]{0.49\textwidth}
        \centering
        \includegraphics[width=\textwidth]{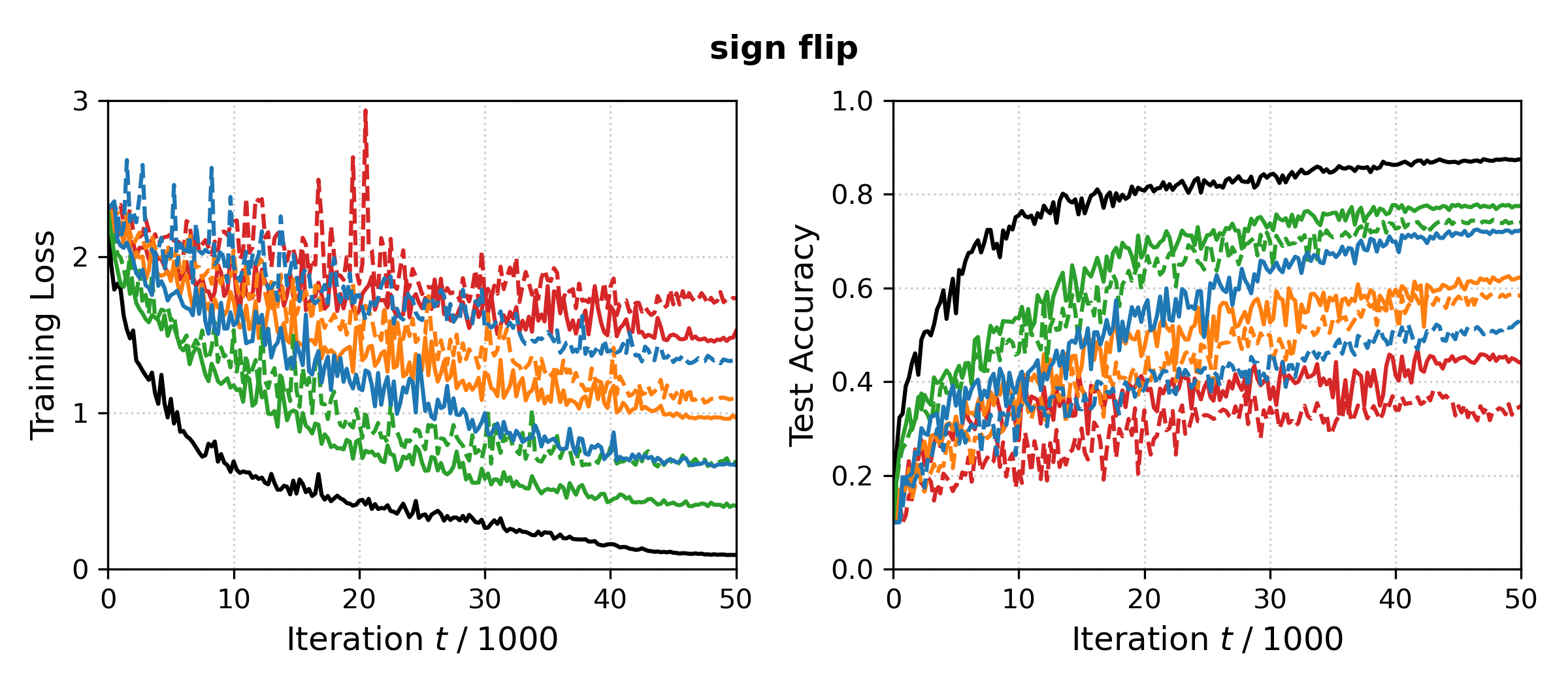}
    \end{subfigure}
    \hfill
    \begin{subfigure}[b]{0.49\textwidth}
        \centering
        \includegraphics[width=\textwidth]{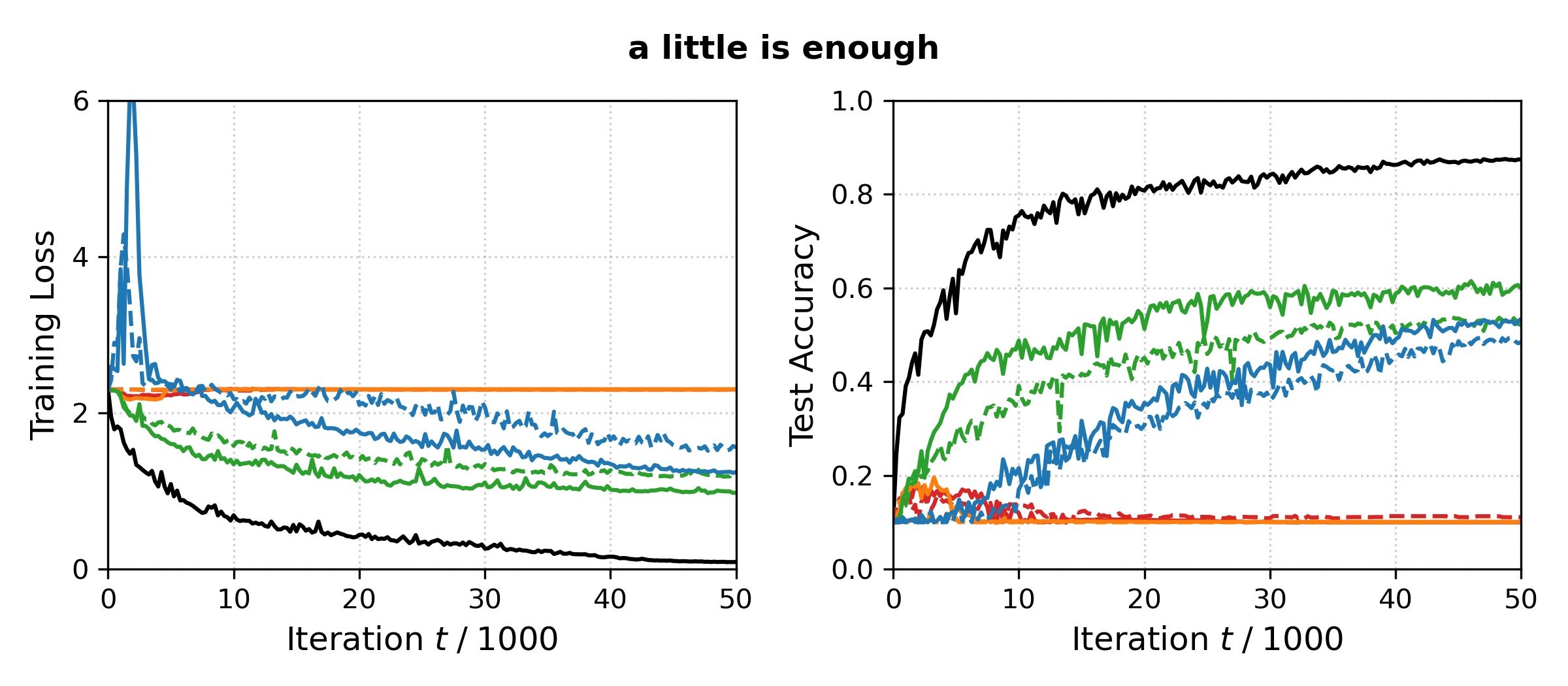}
    \end{subfigure}
    \begin{subfigure}[b]{0.6\textwidth}
        \centering
        \includegraphics[width=\textwidth]{figs/legends_cifar.png}
    \end{subfigure}
    \caption{Performance comparison of R-DSGD (dashed lines) and R-DSGD-M (solid lines) on training ResNet-20 in the CIFAR-10 dataset with various robust aggregation rules under four scenarios: no attack (top-left), label flip attack (top-right), sign flip attack (bottom-left), and ``a little is enough" attack (bottom-right). All experiments use a batch size of $m=4$ and a Dirichlet distribution parameter of $\alpha=0.5$, simulating an environment with large stochastic noise and high data heterogeneity.}
    \label{attack performance cifar}
    \end{figure*}
    In this section, we carry out synthetic experiments to validate the error bounds. We further evaluate the empirical performance of R-DSGD and R-DSGD-M  on the benchmark datasets and verify their consistency with our theoretical analyses. Due to space constraints, details of the experiments are provided in Appendix~\ref{Details and additional experiments}.
    
    \subsection{Synthetic Experiments}
    we design a controlled synthetic experiment to explicitly illustrate the impact of the heterogeneity parameter $B^2$ and the robustness parameter $\kappa$. We consider a one-dimensional optimization problem with $n$ workers, all of which are honest (i.e., $b=0$). The local objective functions $\{f_i\}_{i=1}^n$ are designed to exactly satisfy the $(G,B)$-bounded dissimilarity condition: $\frac{1}{n}\sum_{i=1}^n|\nabla f_i(x)-\nabla f(x)|^2\leq G^2+B^2|\nabla f(x)|^2$, and an aggregation rule $\ma$ is constructed such that $|\ma(x_1,x_2,\ldots,x_n)-\bar{x}|^2=\frac{\kappa}{n}\sum_{i=1}^n|x_i-\bar{x}|^2$ for any given $\kappa>0$. These constructions allow direction investigation of the impact of $\kappa$ and $B^2$ on the final error. 

    For fixed $\kappa$ and $G$, we vary $B^2$ to study its effect on the final error of R-DSGD and R-DSGD-M under the above aggregation rule. We also vary $\kappa$ to examine its impact. The results are presented in Figures~\ref{synthetic experiment}. For fixed $\kappa$, the final error of both R-DSGD and R-DSGD-M increases as $B^2$ grows, which is consistent with our theoretical bounds. Moreover, larger $\kappa$ further degrades performance and reduces the admissible range of $B^2$, providing empirical support for the necessity of the assumption $\kappa B^2<\mo(1)$.

    \subsection{MNIST Experiments}
    We evaluate R-DSGD and R-DSGD-M by training a Multilayer Perceptron (MLP) on the MNIST dataset \cite{mnist}. We employ four widely used aggregation rules: (Multi)-Krum, CwTM, CwM, and GM. 

To systematically analyze the effects of data heterogeneity and stochastic noise, we vary the data distribution and batch sizes as follows. (i) \textbf{Data heterogeneity:} training samples are distributed among $17$ workers following a Dirichlet distribution \cite{zhao2018federated} parameterized by $\alpha$. We use $\alpha=0.1$ to simulate high data heterogeneity (where workers hold distinct data subsets) and $\alpha=10$ for low data heterogeneity (where workers hold similar data distributions). (ii) \textbf{Stochastic noise:} We control the noise level via the batch size $m$. A small batch size of $m=1$ induces high stochastic noise, while a larger batch size of $m=64$ results in relatively low noise. By combining these settings, we can isolate the individual and joint impacts of heterogeneity and noise on convergence.

Our theoretical analysis holds for both benign and adversarial environments. Accordingly, we conduct experiments in two contexts: (1) \textbf{Benign (attack-free):} We test three scenarios: (i) high heterogeneity with low noise, (ii) low heterogeneity with high noise, and (iii) high heterogeneity with high noise. (2) \textbf{Adversarial:} We introduce $b=4$ Byzantine workers (out of $17$) under high heterogeneity and high noise. We evaluate resilience against three  typical attacks: label flip \cite{li2019rsa}, sign flip \cite{zhu2021byzantine}, and ``a little is enough" \cite{baruch2019little}.
As a baseline, we employ D-SGD-M averaging only over the honest workers (an ``oracle" baseline), consistent with the definition of the honest objective function in \eqref{the problem}.

    \textbf{The impact of data heterogeneity.}
     To isolate the effect of data heterogeneity, we set $\alpha = 0.1$ and $m=64$ in an attack-free environment (Figure \ref{no attack performance}, left panel). The results show that the performance differences between R-DSGD and R-DSGD-M in terms of training loss, top-1 test accuracy, and convergence rate—are negligible. However, both algorithms exhibit a noticeable performance gap compared to the baseline. This confirms the existence of a Byzantine error driven primarily by data heterogeneity. These observations are consistent with the theoretical analysis for R-DSGD-(M) in Section \ref{sec:alg:convergence} and the lower bound for the heterogeneity-induced error in Theorem \ref{thm for lower bound for heterogeneity}. 

\textbf{The impact of stochastic noise.} 
To isolate the effect of stochastic noise, we set $\alpha=10$ and $m=1$ in an attack-free environment (Figure \ref{no attack performance}, right panel). Here, R-DSGD performs much worse compared to the baseline. In contrast, R-DSGD-M achieves a comparable performance to the baseline. This aligns with our theoretical findings for R-DSGD-(M) in Section \ref{sec:alg:convergence}, which suggest that local momentum mitigates the Byzantine error caused by stochastic noise. Furthermore, the gap between R-DSGD and the baseline provides empirical validation of the noise-induced Byzantine error characterized in Theorem \ref{thm for lower bound for randomness}.
    
    \textbf{The joint impact of stochastic noise and data heterogeneity.} 
    Finally, we examine the combined effects of heterogeneity and noise by setting $\alpha = 0.1$ and $m=1$ (Figure \ref{attack performance}). The top-left panel shows the attack-free setting, while the remaining panels show performance under Byzantine attacks. Across all scenarios, incorporating local momentum consistently reduces the error in both training loss and test accuracy, demonstrating its efficacy in suppressing stochastic noise. However, even with local momentum, a persistent error gap remains compared to the baseline. This result is consistent with the unified conclusion of Section \ref{sec:alg:convergence}: while local momentum can dampen the effects of stochastic noise, it cannot eliminate the fundamental Byzantine error arising from data heterogeneity.

    \subsection{CIFAR-10 Experiments} 
    
    To validate our theoretical findings on a more challenging task, we conduct experiments on training a ResNet-20 model over the CIFAR-10 dataset. Similar to the MNIST setup, the training data are distributed across $n=17$ workers according to a Dirichlet distribution with parameter $\alpha$, and the stochastic noise level is controlled via the mini-batch size $m$.

    To examine the joint effect of data heterogeneity and stochastic noise, we set $\alpha=0.5$ and $m=4$. We evaluate the performance of R-DSGD and R-DSGD-M with four aggregation rules: Krum, CwTM, CwM, and GM, under three typical attacks: label flip, sign flip, and ``a little is enough". The results are shown in Figure~\ref{attack performance cifar}. The top-left panel corresponds to the attack-free setting, while the remaining panels present performance under Byzantine attacks. Across all scenarios, R-DSGD-M outperforms R-DSGD, while a persistent gap remains between R-DSGD-M and the baseline. This result is consistent with the unified conclusion of Section 3.

\section{Conclusion}
In this paper, we have conducted a comprehensive convergence analysis of R-DSGD and R-DSGD-M with $(b,\kappa)$-robust aggregation rules under the $(G,B)$-bounded dissimilarity. We demonstrate that, although stochasticity and data heterogeneity impose a fundamental upper bound on the error (i.e., the Byzantine error floor), local momentum serves as a provable mechanism to reduce the specific impact of stochastic variance. Furthermore, we provide a lower bound analysis confirming that the derived upper bounds on the Byzantine error are tight. These theoretical findings align well with our experimental observations. Future work may focus on relaxing the bounded stochastic noise assumption. Furthermore, it would be valuable to investigate whether our lower bound analysis is valid for a broader spectrum of Byzantine-robust methods and whether the bounded dissimilarity assumption can be further relaxed.

\bibliographystyle{icml2026}
\bibliography{main}

@article{lamport1982byzantine,
  title={The {B}yzantine Generals Problem},
  author={Lamport, Leslie and Shosta, Robert and Pease, Marshall},
  journal={ACM Transactions on Programming Languages and Systems},
  volume={4},
  number={3},
  pages={382--401},
  year={1982}
}

@InProceedings{brendan2017communication,
  title = 	 {Communication-Efficient Learning of Deep Networks from Decentralized Data},
  author = 	 {McMahan, Brendan and Moore, Eider and Ramage, Daniel and Hampson, Seth and Arcas, Blaise Aguera y},
  booktitle = 	 {Proceedings of the 20th International Conference on Artificial Intelligence and Statistics},
  pages = 	 {1273--1282},
  year = 	 {2017},
  volume = 	 {54}
}

@article{guerraoui2024byzantine,
  title={{B}yzantine machine learning: A primer},
  author={Guerraoui, Rachid and Gupta, Nirupam and Pinot, Rafael},
  journal={ACM Computing Surveys},
  volume={56},
  number={7},
  pages={1--39},
  year={2024},
  publisher={ACM New York, NY}
}

@article{wang2021field,
  title={A field guide to federated optimization},
  author={Wang, Jianyu and Charles, Zachary and Xu, Zheng and Joshi, Gauri and McMahan, H Brendan and Al-Shedivat, Maruan and Andrew, Galen and Avestimehr, Salman and Daly, Katharine and Data, Deepesh and others},
  journal={arXiv preprint arXiv:2107.06917},
  year={2021}
}

@inproceedings{allouah2023robust,
 author = {Allouah, Youssef and Guerraoui, Rachid and Gupta, Nirupam and Pinot, Rafael and Rizk, Geovani},
 booktitle = {Advances in Neural Information Processing Systems},
 pages = {45744--45776},
 title = {Robust Distributed Learning: Tight Error Bounds and Breakdown Point under Data Heterogeneity},
 volume = {36},
 year = {2023}
}

@article{shi2025optimal,
  author  = {Qiankun Shi and Jie Peng and Kun Yuan and Xiao Wang and Qing Ling},
  title   = {Optimal Complexity in {B}yzantine-Robust Distributed Stochastic Optimization with Data Heterogeneity},
  journal = {Journal of Machine Learning Research},
  year    = {2025},
  volume  = {26},
  number  = {268},
  pages   = {1--58},
}

@article{gupta2025reconciling,
  title={Reconciling Communication Compression and {B}yzantine-Robustness in Distributed Learning},
  author={Gupta, Diksha and Honsell, Antonio and Xu, Chuan and Gupta, Nirupam and Neglia, Giovanni},
  journal={arXiv preprint arXiv:2508.17129},
  year={2025}
}

@inproceedings{blanchard2017machine,
 author = {Blanchard, Peva and El Mhamdi, El Mahdi and Guerraoui, Rachid and Stainer, Julien},
 booktitle = {Advances in Neural Information Processing Systems},
 title = {Machine Learning with Adversaries: {B}yzantine Tolerant Gradient Descent},
 volume = {30},
 year = {2017}
}

@article{chen2017distributed,
author = {Chen, Yudong and Su, Lili and Xu, Jiaming},
title = {Distributed Statistical Machine Learning in Adversarial Settings: {B}yzantine Gradient Descent},
journal={Proceedings of the ACM on Measurement and Analysis of Computing Systems},
year = {2017},
volume = {1},
number = {2}
}

@InProceedings{yin2018byzantine,
  title = 	 {{B}yzantine-Robust Distributed Learning: Towards Optimal Statistical Rates},
  author =       {Yin, Dong and Chen, Yudong and Kannan, Ramchandran and Bartlett, Peter},
  booktitle = 	 {Proceedings of the 35th International Conference on Machine Learning},
  pages = 	 {5650--5659},
  year = 	 {2018},
  volume = 	 {80},
}

@article{xie2018generalized,
  title={Generalized {B}yzantine-tolerant {SGD}},
  author={Xie, Cong and Koyejo, Oluwasanmi and Gupta, Indranil},
  journal={arXiv preprint arXiv:1802.10116},
  year={2018}
}

@article{pillutla2022robust,
  title={Robust aggregation for federated learning},
  author={Pillutla, Krishna and Kakade, Sham M and Harchaoui, Zaid},
  journal={IEEE Transactions on Signal Processing},
  volume={70},
  pages={1142--1154},
  year={2022},
  publisher={IEEE}
}

@article{wu2020federated,
  title={Federated variance-reduced stochastic gradient descent with robustness to {B}yzantine attacks},
  author={Wu, Zhaoxian and Ling, Qing and Chen, Tianyi and Giannakis, Georgios B},
  journal={IEEE Transactions on Signal Processing},
  volume={68},
  pages={4583--4596},
  year={2020},
  publisher={IEEE}
}

@InProceedings{karimireddy2021learning,
  title = 	 {Learning from History for {B}yzantine Robust Optimization},
  author =       {Karimireddy, Sai Praneeth and He, Lie and Jaggi, Martin},
  booktitle = 	 {Proceedings of the 38th International Conference on Machine Learning},
  pages = 	 {5311--5319},
  year = 	 {2021},
  volume = 	 {139},
}

@inproceedings{li2019rsa,
  title={{RSA}: {B}yzantine-robust stochastic aggregation methods for distributed learning from heterogeneous datasets},
  author={Li, Liping and Xu, Wei and Chen, Tianyi and Giannakis, Georgios B and Ling, Qing},
  booktitle={Proceedings of the AAAI Conference on Artificial Intelligence},
  volume={33},
  pages={1544--1551},
  year={2019}
}

@article{kairouz2021advances,
  title={Advances and open problems in federated learning},
  author={Kairouz, Peter and McMahan, H Brendan and Avent, Brendan and Bellet, Aur{\'e}lien and Bennis, Mehdi and Bhagoji, Arjun Nitin and Bonawitz, Kallista and Charles, Zachary and Cormode, Graham and Cummings, Rachel and others},
  journal={Foundations and Trends in Machine Learning},
  volume={14},
  number={1--2},
  pages={1--210},
  year={2021},
  publisher={Now Publishers, Inc.}
}

@article{bottou2018optimization,
  title={Optimization methods for large-scale machine learning},
  author={Bottou, L{\'e}on and Curtis, Frank E and Nocedal, Jorge},
  journal={SIAM Review},
  volume={60},
  number={2},
  pages={223--311},
  year={2018},
  publisher={SIAM}
}

@book{lan2020first,
  title={First-order and Stochastic Optimization Methods for Machine Learning},
  author={Lan, Guanghui},
  year={2020},
  publisher={Springer}
}

@inproceedings{woodworth2020minibatch,
 author = {Woodworth, Blake E and Patel, Kumar Kshitij and Srebro, Nati},
 booktitle = {Advances in Neural Information Processing Systems},
 pages = {6281--6292},
 title = {Minibatch vs Local {SGD} for Heterogeneous Distributed Learning},
 volume = {33},
 year = {2020}
}

@InProceedings{karimireddy2020scaffold,
  title = 	 {{SCAFFOLD}: Stochastic Controlled Averaging for Federated Learning},
  author =       {Karimireddy, Sai Praneeth and Kale, Satyen and Mohri, Mehryar and Reddi, Sashank and Stich, Sebastian and Suresh, Ananda Theertha},
  booktitle = 	 {Proceedings of the 37th International Conference on Machine Learning},
  pages = 	 {5132--5143},
  year = 	 {2020},
  volume = 	 {119},
}

@article{zhang2021fedpd,
  title={{FedPD}: A federated learning framework with adaptivity to non-{IID} data},
  author={Zhang, Xinwei and Hong, Mingyi and Dhople, Sairaj and Yin, Wotao and Liu, Yang},
  journal={IEEE Transactions on Signal Processing},
  volume={69},
  pages={6055--6070},
  year={2021},
  publisher={IEEE}
}

@ARTICLE{wang22wireless,
  author={Wang, Yanmeng and Xu, Yanqing and Shi, Qingjiang and Chang, Tsung-Hui},
  journal={IEEE Journal on Selected Areas in Communications}, 
  title={Quantized Federated Learning Under Transmission Delay and Outage Constraints}, 
  year={2022},
  volume={40},
  number={1},
  pages={323-341},
}

@book{guerraoui2024robust,
  title={Robust Machine Learning: Distributed Methods for Safe AI},
  author={Guerraoui, Rachid and Gupta, Nirupam and Pinot, Rafael},
  year={2024},
  publisher={Springer}
}

@inproceedings{zakerinia2024communication,
  title = 	 {Communication-{e}fficient {f}ederated {l}earning With {d}ata and {c}lient {h}eterogeneity},
  author =       {Zakerinia, Hossein and Talaei, Shayan and Nadiradze, Giorgi and Alistarh, Dan},
  booktitle = 	 {Proceedings of The 27th International Conference on Artificial Intelligence and Statistics},
  pages = 	 {3448--3456},
  year = 	 {2024},
  volume = 	 {238},
}

@inproceedings{gorbunov2023variance,
  title={Variance Reduction is an Antidote to {B}yzantines: Better Rates, Weaker Assumptions and Communication Compression as a Cherry on the Top},
  author={Gorbunov, Eduard and Horv{\'a}th, Samuel and Richt{\'a}rik, Peter and Gidel, Gauthier},
  booktitle={International Conference on Learning Representations},
  year = {2023}
}

@inproceedings{cheng2024momentum,
  title={Momentum Benefits Non-{IID} Federated Learning Simply and Provably},
  author={Cheng, Ziheng and Huang, Xinmeng and Wu, Pengfei and Yuan, Kun},
  booktitle={International Conference on Learning Representations},
  year = {2024}
}

@inproceedings{karimireddybyzantine,
  title={{B}yzantine-Robust Learning on Heterogeneous Datasets via Bucketing},
  author={Karimireddy, Sai Praneeth and He, Lie and Jaggi, Martin},
  booktitle={International Conference on Learning Representations},
  year = {2022}
}

@inproceedings{allouah2025adaptive,
  title={Adaptive Gradient Clipping for Robust Federated Learning},
  author={Allouah, Youssef and Guerraoui, Rachid and Gupta, Nirupam and Jellouli, Ahmed and Rizk, Geovani and Stephan, John},
  booktitle={International Conference on Learning Representations},
  year = {2025}
}

@inproceedings{yang2025tension,
  title={On the Tension between {B}yzantine Robustness and No-Attack Accuracy in Distributed Learning},
  author={Yang, Yi-Rui and Shi, Chang-Wei and Li, Wu-Jun},
  booktitle={Proceedings of the 42nd International Conference on Machine Learning},
  year = {2025},
  volume = {267},
  pages = {71051-71072}
}

@article{nguyen2019new,
  title={New convergence aspects of stochastic gradient algorithms},
  author={Nguyen, Lam M and Nguyen, Phuong Ha and Richt{\'a}rik, Peter and Scheinberg, Katya and Tak{\'a}{\v{c}}, Martin and van Dijk, Marten},
  journal={Journal of Machine Learning Research},
  volume={20},
  number={176},
  pages={1--49},
  year={2019}
}

@book{nesterov2018lectures,
  title={Lectures on Convex Optimization},
  author={Nesterov, Yurii},
  volume={137},
  year={2018},
  publisher={Springer}
}

@InProceedings{gorbunov2020unified,
  title = 	 {A Unified Theory of {SGD}: Variance Reduction, Sampling, Quantization and Coordinate Descent},
  author =       {Gorbunov, Eduard and Hanzely, Filip and Richt{\'a}rik, Peter},
  booktitle = 	 {Proceedings of the 23rd International Conference on Artificial Intelligence and Statistics},
  pages = 	 {680--690},
  year = 	 {2020},
  volume = 	 {108},
}

@article{konevcny2016federated,
  title={Federated optimization: Distributed machine learning for on-device intelligence},
  author={Kone{\v{c}}n{\`y}, Jakub and McMahan, H Brendan and Ramage, Daniel and Richt{\'a}rik, Peter},
  journal={arXiv preprint arXiv:1610.02527},
  year={2016}
}

@inproceedings{alistarh2018byzantine,
 author = {Alistarh, Dan and Allen-Zhu, Zeyuan and Li, Jerry},
 booktitle = {Advances in Neural Information Processing Systems},
 pages = {},
 title = {{B}yzantine Stochastic Gradient Descent},
 volume = {31},
 year = {2018}
}

@article{nemirovski2009robust,
  title={Robust stochastic approximation approach to stochastic programming},
  author={Nemirovski, Arkadi and Juditsky, Anatoli and Lan, Guanghui and Shapiro, Alexander},
  journal={SIAM Journal on Optimization},
  volume={19},
  number={4},
  pages={1574--1609},
  year={2009},
  publisher={SIAM}
}

@inproceedings{li2020federated,
 author = {Li, Tian and Sahu, Anit Kumar and Zaheer, Manzil and Sanjabi, Maziar and Talwalkar, Ameet and Smith, Virginia},
 booktitle = {Proceedings of Machine Learning and Systems},
 pages = {429--450},
 title = {Federated Optimization in Heterogeneous Networks},
 volume = {2},
 year = {2020}
}

@InProceedings{noble2022differentially,
  title = 	 {Differentially Private Federated Learning on Heterogeneous Data },
  author =       {Noble, Maxence and Bellet, Aur\'elien and Dieuleveut, Aymeric},
  booktitle = 	 {Proceedings of The 25th International Conference on Artificial Intelligence and Statistics},
  pages = 	 {10110--10145},
  year = 	 {2022},
  volume = 	 {151},
}

@InProceedings{data2021byzantine,
  title = 	 {{B}yzantine-Resilient High-Dimensional {SGD} with Local Iterations on Heterogeneous Data},
  author =       {Data, Deepesh and Diggavi, Suhas},
  booktitle = 	 {Proceedings of the 38th International Conference on Machine Learning},
  pages = 	 {2478--2488},
  year = 	 {2021},
  volume = 	 {139},
}

@article{zhao2018federated,
  title={Federated learning with non-{IID} data},
  author={Zhao, Yue and Li, Meng and Lai, Liangzhen and Suda, Naveen and Civin, Damon and Chandra, Vikas},
  journal={arXiv preprint arXiv:1806.00582},
  year={2018}
}

@ARTICLE{mnist,
  author={Lecun, Y. and Bottou, L. and Bengio, Y. and Haffner, P.},
  journal={Proceedings of the IEEE}, 
  title={Gradient-based learning applied to document recognition}, 
  year={1998},
  volume={86},
  number={11},
  pages={2278-2324},
}

@INPROCEEDINGS{cifar,
  author={He, Kaiming and Zhang, Xiangyu and Ren, Shaoqing and Sun, Jian},
  booktitle={2016 IEEE Conference on Computer Vision and Pattern Recognition }, 
  title={Deep Residual Learning for Image Recognition}, 
  year={2016},
  volume={},
  number={},
  pages={770-778},
  }

@inproceedings{baruch2019little,
 author = {Baruch, Gilad and Baruch, Moran and Goldberg, Yoav},
 booktitle = {Advances in Neural Information Processing Systems},
 title = {A Little Is Enough: Circumventing Defenses For Distributed Learning},
 volume = {32},
 year = {2019}
}

@inproceedings{
zhu2021byzantine,
title={{B}yzantine-Resilient Non-Convex Stochastic Gradient Descent},
author={Zeyuan Allen-Zhu and Faeze Ebrahimianghazani and Jerry Li and Dan Alistarh},
booktitle={International Conference on Learning Representations},
year={2021},
}

@inproceedings{loshchilov2017sgdr,
  title={SGDR: Stochastic Gradient Descent with Warm Restarts},
  author={Loshchilov, Ilya and Hutter, Frank},
  booktitle={International Conference on Learning Representations},
  year={2017}
}

@InProceedings{allouah2023fixing,
  title = 	 {Fixing by Mixing: A Recipe for Optimal {B}yzantine {ML} under Heterogeneity},
  author =       {Allouah, Youssef and Farhadkhani, Sadegh and Guerraoui, Rachid and Gupta, Nirupam and Pinot, Rafael and Stephan, John},
  booktitle = 	 {Proceedings of The 26th International Conference on Artificial Intelligence and Statistics},
  pages = 	 {1232--1300},
  year = 	 {2023},
  volume = 	 {206}
}

@InProceedings{farhadkhani2022byzantine,
  title = 	 {{B}yzantine Machine Learning Made Easy By Resilient Averaging of Momentums},
  author =       {Farhadkhani, Sadegh and Guerraoui, Rachid and Gupta, Nirupam and Pinot, Rafael and Stephan, John},
  booktitle = 	 {Proceedings of the 39th International Conference on Machine Learning},
  pages = 	 {6246--6283},
  year = 	 {2022},
  volume = 	 {162},
}

@InProceedings{xie2019zeno,
  title = 	 {Zeno: Distributed Stochastic Gradient Descent with Suspicion-based Fault-tolerance},
  author =       {Xie, Cong and Koyejo, Sanmi and Gupta, Indranil},
  booktitle = 	 {Proceedings of the 36th International Conference on Machine Learning},
  pages = 	 {6893--6901},
  year = 	 {2019},
  volume = 	 {97},
  series = 	 {Proceedings of Machine Learning Research},
}

@inproceedings{Molodtsov2026bant, 
title={Bant: Byzantine Antidote via Trial Function and Trust Scores},
volume={40}, 
booktitle ={Proceedings of the AAAI Conference on Artificial Intelligence}, 
author={Molodtsov, Gleb and Medyakov, Daniil and Skorik, Sergey and Khachaturov, Nikolas and Tigranyan, Shahane and Aletov, Vladimir and Avetisyan, Aram and Takáč, Martin and Beznosikov, Aleksandr}, 
year={2026}, 
pages={24431-24440}}

\newpage
\appendix
\onecolumn
\section{Proof of the Convergence Results for R-DSGD}\label{Proof of the Convergence Result for R-DSGD}
    In this section, we establish the convergence of R-DSGD. To this end, we first derive several technical bounds on the stochastic gradients and the aggregation error defined in~\eqref{eq:notations:A}. These results allow us to obtain a descent lemma (Lemma~\ref{descent lemma for R-DSGD}) for the objective function, which in turn leads to the convergence results of R-DSGD in Appendix~\ref{sec:Proof of Convergence of R-DSGD}.
    \subsection{Technical Lemmas and Notations}
    We first introduce some notations. For any real value $a\in\rr$, denote $\lfloor a\rfloor$ as the largest integer that is smaller than $a$: $\lfloor a\rfloor\in\mathbb{Z},\; a-1<\lfloor a\rfloor\leq a$. We denote the history from rounds $0$ to $t-1$ by $\mathcal{P}_t$. Specifically, 
    \[\mathcal{P}_t=\left\{x^0,x^1,\ldots,x^{t-1}; m_i^0,m_i^1,\ldots,m_i^{t-1}; i=1,2,\ldots,n\right\}.\]
    We denote by $\ep_t[\cdot]$ the conditional expectation $\ep[\cdot|\mathcal{P}_t]$ and $\ep[\cdot]$ the total expectation, respectively. We also introduce the following notational conventions:
    \begin{equation}\label{eq:notations:A}
        \begin{aligned}
            \bar{g}_{\mh}^t:&=\frac{1}{h}\sum_{i\in\mh}g_i^t,\\
            e^t&:=\ma(g_1^t,g_2^t,\ldots,g_n^t)-\bar{g}_{\mh}^t.
        \end{aligned}
    \end{equation}
    
    Now we develop two technical lemmas. 
    \begin{lemma}\label{lem:average sg}
        Suppose that Assumptions \ref{heterogeneity} and \ref{sfo} hold. Then, for $\bar{g}_{\mh}^t$ defined in~\eqref{eq:notations:A}, we have 
        \begin{equation}
            \frac{1}{h}\sum_{i\in\mh}\ep_t\left[\norm{g_i^t-\bar{g}_{\mh}^t}^2\right]\leq \sigma^2+G^2+B^2\norm{\nabla f_{\mh}(x^{t-1})}^2.
        \end{equation}
    \end{lemma}
    \begin{proof}
        Since $\bar{g}_{\mh}^t$ is the average of all honest stochastic updates, it can be characterized as the unique solution of 
        \[\bar{g}_{\mh}^t=\argmin_{y\in\rr^d}\frac{1}{h}\sum_{i\in\mh}\norm{g_i^t-y}^2.\]
        Due to the optimality of $\bar{g}_{\mh}^t$, we have 
        \begin{align*}
            \frac{1}{h}\sum_{i\in\mh}\ep_t\left[\norm{g_i^t-\bar{g}_{\mh}^t}^2\right]&\leq \frac1h\sum_{i\in\mh}\ep_t\left[\norm{g_i^t-\nabla f_{\mh}(x^{t-1})}^2\right]\\
            &= \frac{1}{h}\ep_t\left[\sum_{i\in\mh}\norm{g_i^t-\nabla f_i(x^{t-1})+\nabla f_i(x^{t-1})-\nabla f_{\mh}(x^{t-1})}^2\right]\\
            & \mathop{\leq}^{\text{(i)}} \frac{1}{h }\sum_{i\in\mh}\left[\sigma^2+\norm{\nabla f_i(x^{t-1})-\nabla f_{\mh}(x^{t-1})}^2\right]\\
            & \mathop{\leq}^{\text{(ii)}} \sigma^2+G^2+B^2\norm{\nabla f_{\mh}(x^{t-1})}^2,
        \end{align*}
        where $(i)$ follows from Assumption \ref{sfo} that the stochastic gradients are unbiased with bounded variances, and $(b)$ uses the $(G,B)$-bounded dissimilarity assumption. The proof is thus completed.
    \end{proof}
    \begin{lemma}\label{lem:et:gH}
        Suppose that Assumptions \ref{smoothness}, \ref{heterogeneity}-\ref{aggregation} hold. Then, for $\bar{g}_{\mh}^t$ and $e^t$ defined in~\eqref{eq:notations:A}, we have 
        \begin{align}
            \ep_t\left[\norm{\bar{g}_{\mh}^t}^2\right]&\leq \frac{2\sigma^2}{h}+2\norm{\nabla f_{\mh}(x^{t-1})}^2, \label{bound for noisy gradient}\\
            \ep_t\left[\norm{e^t}^2\right]&\leq \kappa\sigma^2+\kappa \left(G^2+B^2\norm{\nabla f_{\mh}(x^{t-1})}^2\right).\label{bound for error}
        \end{align}
    \end{lemma}
    \begin{proof}
        By Assumption \ref{sfo} that the stochastic gradients between workers are independent, $\ep_t[g_i^t]=\nabla f_i(x^{t-1})$, and the variance is bounded, we have
        \begin{align*}
            \ep_t\left[\norm{\bar{g}_{\mh}^t-\nabla f_{\mh}(x^{t-1})}^2\right]&=\ep_t\left[\norm{\frac{1}{h }\sum_{i\in\mh}(g_i^t-\nabla f_i(x^{t-1}))}^2\right]=\frac{1}{h ^2}\sum_{i\in\mh}\ep\left[\norm{g_i^t-\nabla f_i(x^{t-1})}^2\right] \leq \frac{\sigma^2}{h }.
        \end{align*}
        By applying AM-GM inequality and the inequality above, we obtain (\ref{bound for noisy gradient}) as
        \begin{align*}\ep_t\left[\norm{\bar{g}_{\mh}^t}^2\right]&=\ep_t\left[\norm{\frac{1}{h }\sum_{i\in\mh}[g_i^t-\nabla f_i(x^{t-1})]+\frac{1}{h }\sum_{i\in\mh}\nabla f_i(x^{t-1})}^2\right]\\
            &\leq 2\ep_t\left[\norm{\bar{g}_{\mh}^t-\nabla f_{\mh}(x^{t-1})}^2\right]+2\norm{\nabla f_{\mh}(x^{t-1})}^2\\
            &\leq \frac{2\sigma^2}{h }+2\norm{\nabla f_{\mh}(x^{t-1})}^2.
        \end{align*}
        
        Using Assumption~\ref{aggregation} that the aggregation rule $\ma$ is of $(b,\kappa)$-robustness and applying Lemma~\ref{lem:average sg}, we obtain
        \begin{align*}      \ep_t\left[\norm{e^t}^2\right]&=\ep_t\left[\norm{\ma(g_1^t,g_2^t,\ldots,g_n^t)-\bar{g}_{\mh}^t}^2\right]\\
        &\leq\frac{\kappa}{h }\sum_{i\in\mh}\ep_t\left[\norm{g_i^t-\bar{g}_\mh^t}^2\right]\\
        &\leq \kappa\sigma^2+\kappa(G^2+B^2\norm{\nabla f_{\mh}(x^{t-1})}^2).
        \end{align*}
         The proof of Lemma~\ref{lem:et:gH} is complete.
    \end{proof}

    The following descent lemma characterizes the one-step decrease of the objective function $f_{\mh}$.
    \begin{lemma}\label{descent lemma for R-DSGD}
        Suppose that Assumptions \ref{smoothness}, \ref{heterogeneity}-\ref{aggregation} hold. Then we have
        \begin{align*}
            \ep[f_{\mh}(x^{t})]&\leq\ep[f_{\mh}(x^{t-1})]+\gamma_{t}\left(-\frac12+2\gamma_{t} L+\kappa B^2\left(\frac{1}{2}+\gamma_{t} L \right)\right) \ep\left[\norm{\nabla f_{\mh}(x^{t-1})}^2\right]\\
            &\quad +\frac{2\gamma_{t} ^2L\sigma^2}{h }+\gamma_{t}\left(\frac{1}{2}+\gamma_{t} L \right) \left(\kappa\sigma^2+\kappa G^2\right).
        \end{align*}
    \end{lemma}
    \begin{proof}
        According to the $L$-smoothness and the iteration rule $x^t-x^{t-1}=-\gamma_{t}\ma(g_1^t,g_2^t,\ldots,g_n^t)=-\gamma_{t}[e^t+\bar{g}_{\mh}^t]$, we have
        \begin{align*}
            &\ep_t[f_\mh(x^{t})]\\
            &\leq f_{\mh}(x^{t-1})+\innerproduct{\nabla f_{\mh}(x^{t-1})}{\ep_t[x^{t}-x^{t-1}]}+\frac{L}{2}\ep_t\left[\norm{x^{t}-x^{t-1}}^2\right]\\
            &=f_{\mh}(x^{t-1})+\innerproduct{\nabla f_{\mh}(x^{t-1})}{-\gamma_{t} \ep_t[e^t]-\gamma_{t} \ep_t[\bar{g}^t_{\mh}]}+\frac{\gamma_{t} ^2L}{2}\ep_t\left[\norm{e^t+\bar{g}_{\mh}^t}^2\right]\\
            &\mathop{\leq}^{\text{(i)}} f_{\mh}(x^{t-1})-\gamma_{t}\innerproduct{\nabla f_{\mh}(x^{t-1})}{\ep_t[e^t]}-\gamma_{t} \norm{\nabla f_{\mh}(x^{t-1})}^2+\gamma_{t}^2L\ep_t\left[\norm{e^t}^2\right]+\gamma_{t} ^2L\ep_t\left[\norm{\bar{g}_{\mh}^t}^2\right]\\
            &\mathop{\leq}^{\text{(ii)}}f_{\mh}(x^{t-1})+\frac{\gamma_{t} }{2}\norm{\nabla f_{\mh}(x^{t-1})}^2+\frac{\gamma_{t} }{2}\ep_t\left[\norm{e^t}^2\right]-\gamma_{t} \norm{\nabla f_{\mh}(x^{t-1})}^2+\gamma_{t} ^2L\ep_t\left[\norm{e^t}^2\right]+\gamma_{t} ^2L\ep_t\left[\norm{\bar{g}_{\mh}^t}^2\right]\\
            &=f_{\mh}(x^{t-1})-\frac{\gamma_{t}}{2}\norm{\nabla f_{\mh}(x^{t-1})}^2+\left(\frac{\gamma_{t} }{2}+\gamma_{t} ^2L\right)\ep_t\left[\norm{e^t}^2\right]+\gamma_{t} ^2L\ep_t\left[\norm{\bar{g}_{\mh}^t}^2\right],
        \end{align*}
        where (i) follows from that $\norm{a+b}^2\leq 2\norm{a}^2+2\norm{b}^2$ and $\ep_t[\bar{g}_{\mh}^t]=\nabla f_{\mh}(x^{t-1})$, and (ii) follows from that $\innerproduct{a}{b}\leq \frac{\norm{b}^2}{2}+\frac{\norm{a}^2}{2}$. 

        Now using Lemma~\ref{lem:et:gH} for the last two terms $\ep\left[\norm{e^t}^2\right]$ $\ep \left[\norm{\bar{g}_{\mh}^t}^2\right]$ and taking the total expectation , we further obtain 
        \begin{align*}
            \ep[f_{\mh}(x^{t})]&\leq\ep\left[f_{\mh}(x^{t-1})\right]-\frac{\gamma_{t} }{2}\ep\left[\norm{\nabla f_{\mh}(x^{t-1})}^2\right]+\left(\frac{\gamma_{t}}{2}+\gamma_{t}^2 L\right)\left(\kappa\sigma^2+\kappa G^2+\kappa B^2\ep\left[\norm{\nabla f_{\mh}(x^{t-1})}^2\right]\right)\\
            &\quad + \gamma_{t}^2 L\left(\frac{2\sigma^2}{h}+2\ep\left[\norm{\nabla f_{\mh}(x^{t-1})}^2\right]\right)\\
            &=\ep[f_{\mh}(x^{t-1})]+\left(-\frac{\gamma_{t}}{2} +2\gamma_{t} ^2L+\kappa B^2\left(\frac{\gamma_{t} }{2}+\gamma_{t} ^2L\right)\right)\ep\left[\norm{\nabla f_{\mh}(x^{t-1})}^2\right]\\
            &\quad +\frac{2\gamma_{t} ^2L\sigma^2}{h }+\left(\frac{\gamma_{t} }{2}+\gamma_{t} ^2L\right)\left(\kappa\sigma^2+\kappa G^2\right),
        \end{align*}
        as claimed. 
    \end{proof}
    The next lemma is the key tool for proving results for global objective functions satisfying the PL condition in Assumption \ref{pl}, and is slightly modified from Lemma 3.1 in the work of \citet{guerraoui2024robust}. 
    \begin{lemma}\label{lemma for PL}
        Let $\alpha_1,\alpha_2,\alpha_3$ be real values such that $\alpha_1>0$, $\alpha_2\geq 0$ and $\alpha_3\geq 0$, and let $T>2$. Consider two positive real-valued sequences $\left\lbrace \gamma_t\right\rbrace_{t=1}^T$ and $\left\lbrace U_t\right\rbrace_{t=1}^T$ such that for all $t\in [T]$, 
        \[\gamma_t =\begin{cases}
            \gamma_0, & \mathrm{if}\;t<\left\lfloor T/2\right\rfloor,\\
            \frac{2}{\alpha_1\left(s_0+t-\left\lfloor T/2\right\rfloor\right)}, &\mathrm{if}\;t\geq \left\lfloor T/2\right\rfloor,
        \end{cases}\; U_{t+1}\leq (1-\alpha_1\gamma_t)U_t+\alpha_2\gamma_t ^2+\alpha_3\gamma_t ,\]
        where $\gamma_0$ and $s_0$ are positive real values such that $s_0>2$ and $\gamma_0=\frac{2}{\alpha_1s_0}<\frac{1}{\alpha_1}$. Then 
        \begin{align*}
            U_{T}\leq(s_0-1)^2U_1\frac{4\exp\left(-\alpha_1\gamma_0(\left\lfloor T/2\right\rfloor-1)\right)}{T^2}+\frac{4(\alpha_2\gamma_0+\alpha_3)(s_0-1)^2}{\alpha_1 T^2}+\frac{8\alpha_2}{\alpha_1^2 T}+\frac{2\alpha_3}{\alpha_1}.
        \end{align*}
    \end{lemma}
    \begin{proof}
        For any $t\in\{1,2,\ldots,\left\lfloor T/2\right\rfloor-1\}$, since $\gamma_t=\gamma_0$, we have 
        \[U_{t+1}\leq (1-\alpha\gamma_0)U_t+\alpha_2\gamma_0^2+\alpha_3\gamma_0.\]
        By unrolling this recursion, we obtain 
        \begin{align}
            U_{\left\lfloor T/2\right\rfloor}&\leq (1-\alpha_1\gamma_0)^{\left\lfloor T/2\right\rfloor-1}U_1+(\alpha_2\gamma_0^2+\alpha_3\gamma_0)\sum_{t=1}^{\left\lfloor T/2\right\rfloor-1}(1-\alpha_1\gamma_0)^{t-1}\notag\\
            &\leq (1-\alpha_1\gamma_0)^{\left\lfloor T/2\right\rfloor-1}U_1+\frac{\alpha_2\gamma_0^2+\alpha_3\gamma_0}{1-(1-\alpha_1\gamma_0)}\notag\\
            &=(1-\alpha_1\gamma_0)^{\left\lfloor T/2\right\rfloor-1}U_1+\frac{\alpha_2\gamma_0+\alpha_3}{\alpha_1}\notag\\
            &\overset{\text{(i)}}{\leq} U_1\exp\left(-\alpha_1\gamma_0(\left\lfloor T/2\right\rfloor-1)\right)+\frac{\alpha_2\gamma_0+\alpha_3}{\alpha_1},\label{intermediate for pl lemma}
        \end{align}
        where in $(i)$ we use $(1-a)^b\leq \exp(-ab)$ for $a<1$, $b>0$. 

        For any $t\in\{\left\lfloor T/2\right\rfloor,\left\lfloor T/2\right\rfloor+1, \ldots, T\}$, we have $\gamma_t=\frac{2}{\alpha_1(s_0+t-\left\lfloor T/2\right\rfloor)}\leq\gamma_0$. Let $s=s_0-\left\lfloor T/2\right\rfloor$, and hence $\gamma_t=\frac{2}{\alpha_1(s+t)}$. Now the recursion of $U^t$ is
        \begin{align*}
            U_{t+1}&\leq (1-\alpha_1\gamma_t)U_t+\alpha_2\gamma_t^2+\alpha_3\gamma_t\\
            &=\left(1-\frac{2}{s+t}\right)U_t+\frac{4\alpha_2}{\alpha_1^2(s+t)^2}+\frac{2\alpha_3}{\alpha_1(s+t)}.
        \end{align*} 
        By multiplying each side with $(s+t)^2$ and noting that $(s+t)^2-2(s+t)\leq (s+t-1)^2$, we obtain   
        \begin{align*}
            (s+t)^2U_{t+1}&\leq \left((s+t)^2-2(s+t)\right)U_t+\frac{4\alpha_2}{\alpha_1^2}+\frac{2\alpha_3}{\alpha_1}(s+t)\\
            &\leq (s+t-1)^2U_t+\frac{4\alpha_2}{\alpha_1^2}+\frac{2\alpha_3}{\alpha_1}(s+t).
        \end{align*}
        Unrolling this recursion, we obtain
        \begin{align*}
            (s+T-1)^2U_{T}&\leq \left(s+\left\lfloor T/2\right\rfloor-1\right)^2U_{\left\lfloor T/2\right\rfloor}+\frac{4\alpha_2}{\alpha_1^2}(T-\left\lfloor T/2\right\rfloor)+\frac{2\alpha_3}{\alpha_1}\sum_{t=\left\lfloor T/2\right\rfloor}^{T-1}(s+t)\\
            &\leq \left(s+\left\lfloor T/2\right\rfloor-1\right)^2U_{\left\lfloor T/2\right\rfloor}+\frac{4\alpha_2}{\alpha_1^2}(T-\left\lfloor T/2\right\rfloor)+\frac{2\alpha_3}{\alpha_1}(s+T-1)(T-\left\lfloor T/2\right\rfloor).
        \end{align*}
        Dividing both sides by $(s+T-1)^2$, and recalling that $s=s_0-\left\lfloor T/2\right\rfloor$, we have 
        \[U_{T}\leq\frac{(s_0-1)^2}{(s_0+T-1-\left\lfloor T/2\right\rfloor)^2}U_{\left\lfloor T/2\right\rfloor}+\frac{4\alpha_2(T-\left\lfloor T/2\right\rfloor)}{\alpha_1^2(s_0+T-1-\left\lfloor T/2\right\rfloor)^2}+\frac{2\alpha_3}{\alpha_1}\frac{T-\left\lfloor T/2\right\rfloor}{s_0+T-1-\left\lfloor T/2\right\rfloor}.\]
        Noting that $\left\lfloor T/2\right\rfloor\leq T/2$ and $s_0>2$, we have $s_0+T-1-\lfloor T/2\rfloor\geq T-\lfloor T/2\rfloor\geq T/2$, and thus
        \[U_{T}\leq\frac{4(s_0-1)^2}{T^2}U_{\left\lfloor T/2\right\rfloor}+\frac{8\alpha_2}{\alpha_1^2T}+\frac{2\alpha_3}{\alpha_1}.\]

        The result follows by
        substituting $U_{\left\lfloor T/2\right\rfloor}$ in (\ref{intermediate for pl lemma}) into above.
    \end{proof}
    \subsection{Proof of Convergence of R-DSGD}\label{sec:Proof of Convergence of R-DSGD}
    Now we are ready to prove Theorem \ref{convergence for nonconvex R-DSGD}.
    \begin{proof}[Proof of Theorem \ref{convergence for nonconvex R-DSGD}]
        Recall that we use constant stepsizes $\gamma_t\equiv \gamma$. Noting that $\gamma L\leq\delta$, we have 
        \begin{align*}
            -\frac12+2\gamma L +\kappa B^2\left(\frac{1}{2}+\gamma L\right)\leq -\frac12+2\delta+\kappa B^2\left(\frac{1}{2}+\delta\right)=:-C_1.
        \end{align*}
        Since $\kappa B^2<\frac{1-4\delta}{1+2\delta}$, we have $C_1>0$. Hence,
        we obtain from Lemma \ref{descent lemma for R-DSGD} and $\gamma L\leq\delta$ that
        \[
            C_1\gamma\ep\left(\norm{\nabla f_{\mh}(x^{t})}^2\right)\leq \ep\left[f_{\mh}(x^{t})\right]-\ep\left[f_{\mh}(x^{t+1})\right]+\frac{2\gamma^2L\sigma^2}{h }+\left(\frac{1}{2}+\delta\right)\gamma\left[\kappa\sigma^2+\kappa G^2\right].
        \]
        Telescoping the above inequality from $t=0$ to $T-1$ and then dividing both sides by $C_1\gamma T$ gives
        \begin{align*}
            \frac1T\sum_{t=0}^{T-1}\ep\left[\norm{\nabla f_{\mh}(x^t)}^2\right]&\leq \frac{1}{T}\sum_{t=0}^{T-1}\frac{\ep\left[f_{\mh}(x^{t})-f_{\mh}(x^{t+1})\right]}{C_1\gamma T}+\frac{2\gamma L\sigma^2}{C_1h}+\frac{\frac{1}{2}+\delta}{C_1}\left(\kappa\sigma^2+\kappa G^2\right)\\
            &\leq \frac{f_{\mh}(x^0)-f_{\mh}^*}{C_1\gamma T}+\frac{2\gamma L\sigma^2}{C_1h}+C_2\left(\kappa\sigma^2+\kappa G^2\right),
        \end{align*}
        as desired.
    \end{proof}

    Due to space constraints, we present the formal statement of Theorem \ref{convergence for PL R-DSGD} below, followed by its proof. 
    \begin{theorem}\label{convergence for PL R-DSGD specific}
        Suppose that Assumptions \ref{smoothness}-\ref{aggregation} hold. For any $\delta<1/4$,
        denote 
        \begin{align*}
            \alpha_1&=2\mu\left(\frac12-2\delta-\kappa B^2\left(\frac{1}{2}+\delta\right)\right),\\
            \alpha_2&=L\left(\frac{2\sigma^2}{h }+\kappa\sigma^2+\kappa G^2\right),\\
            \alpha_3&=\frac12\left(\kappa\sigma^2+\kappa G^2\right).
        \end{align*}
        For any given  $T>0$, suppose that the stepsizes $\{\gamma_t\}$ are chosen as
        \[\gamma_t =\begin{cases}
            \gamma_0, & \mathrm{if}\;t< \left\lfloor T/2\right\rfloor,\\
            \frac{2}{\alpha_1\left(s_0+t-\left\lfloor T/2\right\rfloor\right)}, &\mathrm{if}\;t\geq \left\lfloor T/2\right\rfloor,
        \end{cases}\]
        with $s_0>\max\{\frac{2L}{\delta\alpha_1},2\}$ and $\gamma_0=\frac{2}{\alpha_1s_0}<\min\{\frac{1}{\alpha_1},\frac{\delta}{L}\}$, and assume that $\kappa B^2<\frac{1-4\delta}{1+2\delta}$.
        Then the iterates $\{x^{t}\}$ generated by R-DSGD satisfy
        \begin{align*}
            \ep[f_{\mh}(x^{T-1})]-f_{\mh}(x^*)&\leq (s_0-1)^2\left(f_{\mh}(x^0)-f_{\mh}^*\right)\frac{4\exp\left(-\alpha_1\gamma_0(\left\lfloor T/2\right\rfloor-1)\right)}{T^2}\\
            &+\frac{4(\alpha_2\gamma_0+\alpha_3)(s_0-1)^2}{\alpha_1T^2}+\frac{8\alpha_2}{\alpha_1^2T}+\frac{2\alpha_3}{\alpha_1}.
        \end{align*}
    \end{theorem}
    \begin{proof}[Proof of Theorem \ref{convergence for PL R-DSGD specific}]
        From the choices of stepsizes, we obtain $\gamma_t\leq\gamma_0\leq \frac{\delta}{L}$ and hence $\gamma_t L\leq \delta$ for any $t\geq 1$.
        It follows from  Lemma \ref{descent lemma for R-DSGD} and $\gamma_t L\leq \delta$ that
        \begin{align*}
            \ep[f_{\mh}(x^{t})]-f_{\mh}^*&\leq\ep[f_{\mh}(x^{t-1})]-f_{\mh}^*+\left(-\frac12+2\gamma_t L +\kappa B^2\left(\frac{1}{2}+\gamma_t L \right)\right)\gamma_t \ep\left[\norm{\nabla f_{\mh}(x^{t-1})}^2\right]\\
            &\quad +\frac{2\gamma_t^2L\sigma^2}{h }+\left(\frac{1}{2}+\gamma_t L \right)\gamma_t \left(\kappa\sigma^2+\kappa G^2\right)\\
            &\leq \ep[f_{\mh}(x^{t-1})]-f_{\mh}^*+\left(-\frac12+2\delta +\kappa B^2\left(\frac{1}{2}+\delta \right)\right)\gamma_t\ep\left[\norm{\nabla f_{\mh}(x^{t-1})}^2\right]\\
            &\quad +\frac{2\gamma_t ^2 L\sigma^2}{h }+\left(\frac{1}{2}+\gamma_t L\right)\gamma_t \left(\kappa\sigma^2+\kappa G^2\right).
        \end{align*}
        By the assumption $\kappa B^2<\frac{1-4\delta}{1+2\delta}$ and the PL condition, we have 
        \[\left(-\frac12+2\delta +\kappa B^2\left(\frac{1}{2}+\delta \right)\right)\ep\left[\norm{\nabla f_{\mh}(x^{t-1})}^2\right]\leq 2\mu\left(-\frac12+2\delta +\kappa B^2\left(\frac{1}{2}+\delta \right)\right)\left(f_{\mh}(x^{t-1})-f_{\mh}^*\right).\]
        Denote $U_t=\ep[f_{\mh}(x^{t-1})]-f^*$, and hence
        \begin{align*}
            U_{t+1}&\leq U_t+2\mu\gamma_t \left(-\frac12+2\delta+\kappa B^2\left(\frac{1}{2}+\delta\right)\right)U_t+\frac{2\gamma_t ^2 L\sigma^2}{h }+\left(\frac{1}{2}+\gamma_t L \right)\gamma_t \left(\kappa\sigma^2+\kappa G^2\right)\\
            &=\left(1-\alpha_1\gamma_t \right)U_t+\gamma_t^2L\left(\frac{2\sigma^2}{h}+\kappa\sigma^2+\kappa G^2\right)+\frac{\gamma_t}{2}\left(\kappa\sigma^2+\kappa G^2\right)\\
            &=(1-\alpha_1\gamma_t)U_t+\alpha_2\gamma_t^2+\alpha_3\gamma_t.
        \end{align*}
        The result follows by invoking Lemma \ref{lemma for PL}.
    \end{proof}
    \section{Proof of Convergence Results of R-DSGD-M} \label{Proof for Convergence Result of R-DSGD with Momentum}
    In this section, we prove the convergence of R-DSGD-M by adopting a Lyapunov analysis, which is a standard technique for analyzing momentum-based algorithms. To this end, we first establish the recursive properties for the individual components of the Lyapunov function in Appendix~\ref{sec:Technical Lemmas and Notations}. We then derive the recursion for the Lyapunov function in Appendix~\ref{sec:The Lyapunov Function}. Leveraging this recursive property, we are able to establish the convergence results in Appendix~\ref{sec:proof of R-DSGD-M smooth} and Appendix~\ref{sec:proof of R-DSGD-M PL}.
    \subsection{Technical Lemmas and Notations}\label{sec:Technical Lemmas and Notations}
    
    For convenience, we first introduce our notational conventions:
    \begin{equation}\label{eq:notations}
        \begin{aligned}
            m_i^t &= \beta_t m_i^{t-1}+(1-\beta_t)g_i^t;\\
            \bar{m}_{\mh}^t&:=\frac{1}{h}\sum_{i\in\mh} m_i^t;\\
            \bar{g}_{\mh}^t&:=\frac{1}{h}\sum_{i\in\mh} g_i^t;\\
            \Gamma_{\mh}^t&:=\frac{1}{h }\sum_{i\in\mh}\norm{m_i^t-\bar{m}_{\mh}^t}^2;\\
            \zeta^t &:=\ma(m_1^t,m_2^t,\ldots,m_n^t)-\bar{m}_{\mh}^t;\\
            \delta^t&:=\bar{m}_{\mh}^t-\nabla f_{\mh}(x^{t-1}).
        \end{aligned}
    \end{equation}
    As in Appendix \ref{Proof of the Convergence Result for R-DSGD}, we denote $\ep_t[\cdot]$ as the conditional expectation given the history of last $t-1$ iterates.
    
    In the next lemma, we investigate the recursive property of $\Gamma_{\mh}^t$ .
    \begin{lemma}\label{lem:moment}
        Under Assumptions \ref{heterogeneity} and \ref{sfo}, for all $t>1$, we have
        \begin{equation}
            \ep\left[\Gamma_{\mh}^t\right]\leq\beta_t\ep\left[\Gamma_{\mh}^{t-1}\right]+(1-\beta_t)\left(G^2+B^2\ep\left[\norm{\nabla f_{\mh}(x^{t-1})}^2\right]\right)+(1-\beta_t)^2\sigma^2.
        \end{equation}
    \end{lemma}
    \begin{proof}
        By the recursion rule of local momentum, we have 
        \[m_i^t-\bar{m}_{\mh}^t=\beta_t(m_i^{t-1}-\bar{m}_{\mh}^{t-1})+(1-\beta_t)(g_i^t-\bar{g}_{\mh}^{t}).\]
        Taking the squared norm on both sides and expanding, we obtain 
        \begin{align*}
            \norm{m_i^t-\bar{m}_{\mh}^t}^2 &=\beta_t^2\norm{m_i^{t-1}-\bar{m}_{\mh}^{t-1}}^2+(1-\beta_t)^2\norm{g_i^t-\bar{g}_{\mh}^{t}}^2+2\beta_t(1-\beta_t)\innerproduct{m_i^{t-1}-\bar{m}_{\mh}^{t-1}}{g_i^t-\bar{g}_{\mh}^{t}}.
        \end{align*}
        Taking the conditional expectation $\ep_t[\cdot]$ on both sides and noting that $\ep_t[g_i^t]=\nabla f_i(x^{t-1})$ by Assumption~\ref{sfo}, we have 
        \begin{align*}
            \ep_t\left[\norm{m_i^t-\bar{m}_{\mh}^t}^2\right] &=\beta_t^2\ep_t\left[\norm{m_i^{t-1}-\bar{m}_{\mh}^{t-1}}^2\right]+(1-\beta_t)^2\ep_t\left[\norm{g_i^t-\bar{g}_{\mh}^{t}}^2\right]\\
            &\quad+2\beta_t(1-\beta_t)\innerproduct{m_i^{t-1}-\bar{m}_{\mh}^{t-1}}{\nabla f_i(x^{t-1})-\nabla f_{\mh}(x^{t-1})}.
        \end{align*}
        Applying the inequality $2\innerproduct{x}{y}\leq\norm{x}^2+\norm{y}^2$, we obtain  
        \begin{align*}
            \ep_t\left[\norm{m_i^t-\bar{m}_{\mh}^t}^2\right] &\leq (\beta_t^2+\beta_t(1-\beta_t))\ep_t\left[\norm{m_i^{t-1}-\bar{m}_{\mh}^{t-1}}^2\right]+(1-\beta_t)^2\ep_t\left[\norm{g_i^t-\bar{g}_{\mh}^{t}}^2\right]\\
            &\quad+\beta_t(1-\beta_t)\norm{\nabla f_i(x^{t-1})-\nabla f_{\mh}(x^{t-1})}^2\\
            &=\beta_t\ep_t\left[\norm{m_i^{t-1}-\bar{m}_{\mh}^{t-1}}^2\right]+(1-\beta_t)^2\ep_t\left[\norm{g_i^t-\bar{g}_{\mh}^{t}}^2\right]\\
            &\quad+\beta_t(1-\beta_t)\norm{\nabla f_i(x^{t-1})-\nabla f_{\mh}(x^{t-1})}^2.
        \end{align*}
        Taking the average over $i\in\mh$ in the inequality above derives
        \begin{align*}
            \frac1h\sum_{i\in\mh}\ep_t\left[\norm{m_i^t-\bar{m}_{\mh}^t}^2\right]&\leq \frac{\beta_t}{h}\sum_{i\in\mh}\ep_t\left[\norm{m_i^{t-1}-\bar{m}_{\mh}^{t-1}}^2\right]+\frac{(1-\beta_t)^2}{h}\sum_{i\in\mh}\ep_t\left[\norm{g_i^t-\bar{g}_{\mh}^t}^2\right]\\
            &\quad+\frac{\beta_t(1-\beta_t)}{h}\sum_{i\in\mh}\norm{\nabla f_i(x^{t-1})-\nabla f_{\mh}(x^{t-1})}^2.
        \end{align*}
        In the inequality above, by using Assumption~\ref{heterogeneity} for the term $\frac1h\sum_{i\in\mh}\norm{\nabla f_i(x^{t-1})-\nabla f_{\mh}(x^{t-1})}^2$ and Lemma~\ref{lem:average sg} for the term $\frac1h\sum_{i\in\mh}\ep_t\left[\norm{g_i^t-\bar{g}_{\mh}^t}^2\right]$, we further obtain 
        \begin{align*}
            \frac1h\sum_{i\in\mh}\ep_t\left[\norm{m_i^t-\bar{m}_{\mh}^t}^2\right]&\leq \frac{\beta_t}{h}\sum_{i\in\mh}\ep_t\left[\norm{m_i^{t-1}-\bar{m}_{\mh}^{t-1}}^2\right]+(1-\beta_t)^2\left(\sigma^2+G^2+B^2\norm{\nabla f(x^{t-1})}^2\right)\\
            &\quad+\beta_t(1-\beta_t)\left(G^2+B^2\norm{\nabla f_{\mh}(x^{t-1})}^2\right)\\
            &= \frac{\beta_t}{h}\sum_{i\in\mh}\ep_t\left[\norm{m_i^{t-1}-\bar{m}_{\mh}^{t-1}}^2\right]+(1-\beta_t)^2\sigma^2\\
            &\quad+(1-\beta_t)\left(G^2+B^2\norm{\nabla f_{\mh}(x^{t-1})}^2\right).
        \end{align*}
        By taking the total expectation on both sides and recalling the definition of $\Gamma_{\mh}^t$, we finally obtain 
        \[\ep\left[\Gamma_{\mh}^t\right]\leq \beta_t\ep\left[\Gamma_{\mh}^{t-1}\right]+(1-\beta_t)\left(G^2+B^2\ep\left[\norm{\nabla f_{\mh}(x^{t-1})}^2\right]\right)+(1-\beta_t)^2\sigma^2.\]
        The proof is completed.
    \end{proof}
    The relationship between $\zeta^t$ and $\Gamma_{\mh}^t$ is provided in the next lemma.
    \begin{lemma}\label{lem:zeta}
        Suppose that Assumption \ref{aggregation} holds. Then, we have
        \begin{equation}
     \norm{\zeta^t}^2\leq\kappa\Gamma_{\mh}^t.
        \end{equation}
    \end{lemma}
    \begin{proof}
        The result holds directly from the $(b,\kappa)$-robustness of $\ma$ and the definitions in ~\eqref{eq:notations}:
        \[\norm{\zeta^t}^2=\norm{\ma(m_1^t,m_2^t,\ldots,m_n^t)-\bar{m}_{\mh}^t}^2\leq\frac{\kappa}{h}\sum_{i\in\mh}\norm{m_i^t-\bar{m}_{\mh}^t}^2=\kappa \Gamma_{\mh}^t.\]
    \end{proof}
    Next, we introduce two technical lemmas from the work of \citet{allouah2023fixing} for the recursive properties of $\norm{\delta^t}$ and $f_{\mh}(x^t)$. For completeness, we also provide the proofs here. 
    \begin{lemma}[Lemma 8 in \cite{allouah2023fixing}]\label{lem:delta}
        Suppose that Assumptions \ref{smoothness} and \ref{sfo} hold. With notations in (\ref{eq:notations}), we have for any $t\geq 1$ that 
        \begin{equation}
            \begin{aligned}
                \ep\left[\norm{\delta^{t+1}}^2\right]&\leq (1+\gamma_{t} L)(1+4\gamma_{t} L)\beta_{t+1}^2\ep\left[\norm{\delta^{t}}^2\right]+4\gamma_{t} L(1+\gamma_{t} L)\beta_{t+1}^2\ep\left[\norm{\nabla f_{\mh}(x^{t-1})}^2\right]\\
                &\quad+\frac{(1-\beta_{t+1})^2}{h}\sigma^2+2\gamma_{t} L(1+\gamma_{t} L)\beta_{t+1}^2\ep\left[\norm{\zeta^{t}}^2\right].
            \end{aligned}
        \end{equation}
    \end{lemma}
    \begin{proof}
        Recall from the notations in (\ref{eq:notations}) that 
        \[\delta^{t+1}=\bar{m}_{\mh}^{t+1}-\nabla f_{\mh}(x^{t})=\beta_{t+1}\bar{m}_{\mh}^{t}+(1-\beta_{t+1})\bar{g}_{\mh}^{t+1}-\nabla f_{\mh}(x^{t}).\]
        By adding and subtracting $\beta_{t+1}\nabla f_{\mh}(x^{t-1})$ and $\beta_{t+1}\nabla f_{\mh}(x^{t})$, we have 
        \begin{align*}
            \delta^{t+1}&=\beta_{t+1}\bar{m}_\mh^{t}-\beta_{t+1}\nabla f_{\mh}(x^{t-1})+(1-\beta_{t+1})\bar{g}_{\mh}^{t+1}-\nabla f_{\mh}(x^{t})+\beta_{t+1}\nabla f_{\mh}(x^{t})+\beta_{t+1}\nabla f_{\mh}(x^{t-1})-\beta_{t+1}\nabla f_{\mh}(x^{t})\\
            &=\beta_{t+1}\left(\bar{m}_{\mh}^{t}-\nabla f_{\mh}(x^{t-1})\right)+(1-\beta_{t+1})\bar{g}_{\mh}^{t+1}-(1-\beta_{t+1})\nabla f_{\mh}(x^{t})+\beta_{t+1}(\nabla f_{\mh}(x^{t-1})-\nabla f_{\mh}(x^{t}))\\
            &=\beta_{t+1}\delta^{t}+(1-\beta_{t+1})\left(\bar{g}_{\mh}^{t+1}-\nabla f_{\mh}(x^{t})\right)+\beta_{t+1}\left(\nabla f_{\mh}(x^{t-1})-\nabla f_{\mh}(x^{t})\right).
        \end{align*}
        Recall that $\delta^{t}$, $x^{t}$, $x^{t-1}$ are deterministic given the history before round $t+1$, hence taking the conditional expectation $\ep_{t+1}[\cdot]$ on both sides yields
        \begin{align*}
            \ep_{t+1}\left[\norm{\delta^{t+1}}^2\right]&=\beta_{t+1}^2\norm{\delta^{t}}^2+(1-\beta_{t+1})^2\ep_{t+1}\left[\norm{\bar{g}_{\mh}^{t+1}-\nabla f_{\mh}(x^{t})}^2\right]+\beta_{t+1}^2\norm{\nabla f_{\mh}(x^{t-1})-\nabla f_{\mh}(x^{t})}^2\\
            &\quad+2\beta_{t+1}(1-\beta_{t+1})\innerproduct{\delta^{t}}{\ep_{t+1}[\bar{g}_{\mh}^{t+1}]-\nabla f_{\mh}(x^{t})}+2\beta_{t+1}^2\innerproduct{\delta^{t}}{\nabla f_{\mh}(x^{t-1})-\nabla f_{\mh}(x^{t})}\\
            &\quad +2\beta_{t+1}(1-\beta_{t+1})\innerproduct{\ep_{t+1}[\bar{g}_{\mh}^{t+1}]-\nabla f_{\mh}(x^{t})}{\nabla f_{\mh}(x^{t})-\nabla f_{\mh}(x^{t-1})}\\
            &\overset{\text{(i)}}{=}\beta_{t+1}^2\norm{\delta^{t}}^2+(1-\beta_{t+1})^2\ep_{t+1}\left[\norm{\bar{g}_{\mh}^{t+1}-\nabla f_{\mh}(x^{t})}^2\right]+\beta_{t+1}^2\norm{\nabla f_{\mh}(x^{t-1})-\nabla f_{\mh}(x^{t})}^2\\ 
            &\quad + 2\beta_{t+1}^2\innerproduct{\delta^{t}}{\nabla f_{\mh}(x^{t-1})-\nabla f_{\mh}(x^{t})}\\
            &\overset{\text{(ii)}}{\leq} \beta_{t+1}^2\norm{\delta^{t}}^2 + (1-\beta_{t+1})^2\frac{\sigma^2}{h}+\beta_{t+1}^2\norm{\nabla f_{\mh}(x^{t-1})-\nabla f_{\mh}(x^{t})}^2\\
            &\quad +2\beta_{t+1}^2\innerproduct{\delta^{t}}{\nabla f_{\mh}(x^{t-1})-\nabla f_{\mh}(x^{t})},
        \end{align*}
        where in (i) we use $\ep_{t+1}[\bar{g}_{\mh}^{t+1}]=\nabla f_{\mh}(x^{t})$ and in (ii) we calculate
        \begin{align*}
            \ep_{t+1}\left[\norm{\bar{g}_{\mh}^{t+1}-\nabla f_{\mh}(x^{t})}^2\right]&=\ep_{t+1}\left[\frac{1}{h^2}\norm{\sum_{i\in\mh}[g_i^{t+1}-\nabla f_i(x^{t})]}^2\right]=\frac{1}{h^2}\sum_{i\in\mh}\norm{g_i^{t+1}-\nabla f_i(x^{t})}^2\leq\frac{\sigma^2}{h}.
        \end{align*}
        Now by the Cauchy-Schwartz inequality, we have $\innerproduct{\delta^{t}}{\nabla f_{\mh}(x^{t-1})-\nabla f_{\mh}(x^{t})}\leq\norm{\delta^{t}}\norm{\nabla f_{\mh}(x^{t-1})-\nabla f_{\mh}(x^{t})}$. Since $f_{\mh}$ is $L$-smooth, we have 
        \[\norm{\nabla f_{\mh}(x^{t-1})-\nabla f_{\mh}(x^{t})}\leq L\norm{x^{t-1}-x^{t}}= \gamma_{t} L\norm{\ma(m_1^{t},m_2^t,\ldots,m_n^{t})}.\]
        Using the above inequality, we obtain 
        \begin{align*}
            \ep_{t+1}\left[\norm{\delta^{t+1}}^2\right]&\leq \beta_{t+1}^2\norm{\delta^{t}}^2+(1-\beta_{t+1})^2\frac{\sigma^2}{h}+\gamma_{t}^2L^2\beta_{t+1}^2\norm{\ma(m_1^{t},m_2^t,\ldots,m_n^{t})}^2\\
            &\quad +2\gamma_{t} L\beta_{t+1}^2\norm{\delta^{t}}\norm{\ma(m_1^{t},m_2^t,\ldots,m_n^{t})}\\
            &\overset{\text{(i)}}{\leq} (1+\gamma_{t} L)\beta_{t+1}^2\norm{\delta^{t}}^2+(1-\beta_{t+1})^2\frac{\sigma^2}{h}+\gamma_{t} L(1+\gamma_{t} L)\beta_{t+1}^2\norm{\ma(m_1^{t},m_2^t,\ldots,m_n^{t})}^2,
        \end{align*}
        where we use $2ab\leq a^2+b^2$ in (i). Now using the definitions of $\zeta^{t}$ and $\delta^{t}$ in (\ref{eq:notations}), we have from $\norm{a+b}^2\leq 2\norm{a}^2+2\norm{b}^2$ that
        \begin{equation}\label{eq: intermediate for borrowed lemma}
            \begin{aligned}
                \norm{\ma(m_1^{t},m_2^t,\ldots,m_n^{t})}^2&=\norm{\zeta^{t}+\bar{m}_{\mh}^{t}}^2\leq 2\norm{\zeta^{t}}^2+2\norm{\bar{m}_{\mh}^{t}}^2\\
                &=2\norm{\zeta^{t}}^2+2\norm{\nabla f_{\mh}(x^{t-1})+\delta^{t}}^2\\
                &\leq 2\norm{\zeta^{t}}^2+4\norm{\nabla f_{\mh}(x^{t-1})}^2+4\norm{\delta^{t}}^2.
            \end{aligned}
        \end{equation}
        Plugging this inequality into the estimate of $\ep_{t+1}\left[\norm{\delta^{t+1}}^2\right]$, we finally obtain
        \begin{align*}
            \ep_{t+1}\left[\norm{\delta^{t+1}}^2\right]&\leq (1+\gamma_{t} L)\beta_{t+1}^2\norm{\delta^{t}}^2+(1-\beta_{t+1})^2\frac{\sigma^2}{h}\\
            &\quad +2\gamma_{t} L(1+\gamma_{t} L)\beta_{t+1}^2\left(\norm{\zeta^{t}}^2+2\norm{\nabla f_{\mh}(x^{t-1})}^2+2\norm{\delta^{t}}^2\right)\\
            &\leq (1+\gamma_{t} L)(1+4\gamma_{t} L)\beta_{t+1}^2\norm{\delta^{t}}^2+(1-\beta_{t+1})^2\frac{\sigma^2}{h}\\
            &\quad+2\gamma_{t} L(1+\gamma_{t} L)\beta_{t+1}^2\norm{\zeta^{t}}^2+4\gamma_{t} L(1+\gamma_{t} L)\beta_{t+1}^2\norm{\nabla f_{\mh}(x^{t-1})}^2,
        \end{align*}
        and the proof is completed by taking the total expectation.
    \end{proof}
    \begin{lemma}[Lemma 9 in \cite{allouah2023fixing}]\label{lem:objective}
        Under Assumption \ref{smoothness}, we have 
        \begin{equation}
            \begin{aligned}
                \ep\left[2f_{\mh}(x^t)-2f_{\mh}(x^{t-1})\right]&\leq-\gamma_t(1-4\gamma_t L)\ep\left[\norm{\nabla f_{\mh}(x^{t-1})}^2\right]\\
                &\quad +2\gamma_t(1+2\gamma_t L)\ep\left[\norm{\delta^t}^2\right]+2\gamma_t(1+\gamma_t L)\ep\left[\norm{\zeta^t}^2\right].
            \end{aligned}
        \end{equation}
    \end{lemma}
    \begin{proof}
        By the $L$-smoothness of $f_{\mh}$ and the recursion rule $x^{t}=x^{t-1}-\gamma_t\ma(m_1^t,m_2^t,\ldots,m_n^t)$, we have
        \begin{equation}\label{eq:inter:allouach}
            \begin{aligned}
                &2f_{\mh}(x^t)-2f_{\mh}(x^{t-1})\\
                &\leq 2\innerproduct{\nabla f_{\mh}(x^{t-1})}{x^t-x^{t-1}}+L\norm{x^t-x^{t-1}}^2 \\
                &=-2\gamma_t\innerproduct{\ma(m_1^t,m_2^t,\ldots,m_n^t)}{\nabla f_{\mh}(x^{t-1})}+\gamma_t^2L\norm{\ma(m_1^t,m_2^t,\ldots,m_n^t)}^2\\
                &\mathop{=}^{\text{(i)}}-2\gamma_t\innerproduct{\bar{m}_{\mh}^t}{\nabla f_{\mh}(x^{t-1})}-2\gamma_t\innerproduct{\zeta^t}{\nabla f_{\mh}(x^{t-1})}+\gamma_t^2L\norm{\zeta^t+\bar{m}_{\mh}^t}^2 \\
                &\mathop{=}^{\text{(ii)}} -2\gamma_t\norm{\nabla f_{\mh}(x^{t-1})}^2-2\gamma_t\innerproduct{\delta^t}{\nabla f_{\mh}(x^{t-1})}-2\gamma_t\innerproduct{\zeta^t}{\nabla f_{\mh}(x^{t-1})}+\gamma_t^2L\norm{\zeta^t+\bar{m}_{\mh}^t}^2,
            \end{aligned}
        \end{equation}
        where in (i), (ii) we recall the notations of $\zeta^t$ and $\delta^t$ in (\ref{eq:notations}). 
        By the Cauchy-Schwartz inequality, the inequality $2ab\leq 2a^2+\frac12b^2$, and (\ref{eq: intermediate for borrowed lemma}), we have 
        \begin{align*}
            2\abs{\innerproduct{\delta^t}{\nabla f_{\mh}(x^{t-1})}}&\leq 2\norm{\delta^t}\norm{\nabla f_{\mh}(x^{t-1})}\leq 2\norm{\delta^t}^2+\frac12\norm{\nabla f_{\mh}(x^{t-1})}^2,\\
            2\abs{\innerproduct{\zeta^t}{\nabla f_{\mh}(x^{t-1})}}&\leq 2\norm{\zeta^t}\norm{\nabla f_{\mh}(x^{t-1})}\leq 2\norm{\zeta^t}^2+\frac12\norm{\nabla f_{\mh}(x^{t-1})}^2,\\
            \norm{\zeta^t+\bar{m}_{\mh}^t}&\leq 2\norm{\zeta^{t}}^2+4\norm{\nabla f_{\mh}(x^{t-1})}^2+4\norm{\delta^{t}}^2.
        \end{align*}
        Substituting these three inequalities into~\eqref{eq:inter:allouach}, we finally obtain 
        \begin{align*}
            2f_{\mh}(x^t)-2f_{\mh}(x^{t-1})&\leq -2\gamma_t\norm{\nabla f_{\mh}(x^{t-1})}^2+\gamma_t\left(2\norm{\delta^t}^2+\frac12\norm{\nabla f_{\mh}(x^{t-1})}^2\right)\\
            &\quad +\gamma_t\left(2\norm{\zeta^t}^2+\frac12\norm{\nabla f_{\mh}(x^{t-1})}^2\right)+\gamma_t^2L\left(2\norm{\zeta^{t}}^2+4\norm{\nabla f_{\mh}(x^{t-1})}^2+4\norm{\delta^{t}}^2\right)\\
            &=(-\gamma_t+4\gamma_t^2L)\norm{\nabla f_{\mh}(x^{t-1})}^2+(2\gamma_t+4\gamma_t^2L)\norm{\delta^t}^2+(2\gamma_t+2\gamma_t^2L)\norm{\zeta^t}^2.
        \end{align*}
        The result follows from taking the total expectation.
    \end{proof}
    \subsection{The Lyapunov Function}\label{sec:The Lyapunov Function}
    \textbf{Motivation and technical novelties.} 
    We now construct the Lyapunov function for analyzing the convergence of R-DSGD-M. To accommodate potentially time-varying momentum parameters, we explicitly isolate the momentum variance term $\Gamma_{\mh}^t$ and incorporate it into the Lyapunov function, thereby enabling a more tractable analysis. However, under the $(G,B)$-bounded dissimilarity assumption, squared gradient-norm terms appear in all recursive bounds (see Lemmas \ref{lem:moment}, \ref{lem:delta}, \ref{lem:objective}). To ensure that these gradient terms are consistently evaluated at the same iterate $x^{t-1}$, we must align the iteration indices of the  components of the Lyapunov functions. This motivates the definition of the Lyapunov function $V^t$ as 
    \begin{equation}\label{eq:lyapunov}
        V^t:=2\ep\left[f_{\mh}(x^{t-1})-f_{\mh}^*\right]+c_1\ep\left[\norm{\delta^t}^2\right]+c_2\ep\left[\Gamma_{\mh}^{t-1}\right],
    \end{equation}
    where $c_1$ and $c_2$ are constants to be determined later. 

    This alignment, however, introduces a technical challenge. The recursion of $\norm{\delta^t}^2$ involves the momentum parameter $\beta_{t+1}$, whereas the recursion of $\Gamma_{\mh}^{t-1}$ depends on $\beta_t$. Since achieving  variance decay requires $1-\beta_t\approx \mo(\gamma_t)$, allowing time-varying stepsizes complicates the control of certain terms in the analysis. To circumvent this issue, we impose a mild condition that prevents the stepsizes from decreasing too rapidly, which ensures the stability of the resulting bounds. We formally investigate the recursion of this Lyapunov function in the following lemma.
    \begin{lemma}\label{lem:lyapunov}
        Suppose that Assumptions \ref{smoothness}, \ref{heterogeneity}-\ref{aggregation} hold.  Let the stepsizes $\{\gamma_t\}$ satisfy $\gamma_t L\leq 1/36$, and let the momentum parameters $\{\beta_t\}$ be set as $\beta_t=1-36\gamma_tL$. Furthermore, assume that the stepsizes are nonincreasing, i.e., $\gamma_1\geq\gamma_2\geq\cdots$, and satisfy the condition $\gamma_{t+1}\geq\frac{2}{3}\gamma_t$ for all $t\geq 1$. Then the Lyapunov function $V^t$ defined in~\eqref{eq:lyapunov} with $c_1=1/(8L)$ and $c_2=\kappa/(2L)$ satisfies
        \begin{equation}
            \begin{aligned}
                V^{t+1}&\leq \gamma_t\left(-\frac{3}{8}+21\kappa B^2\right)\ep\left[\norm{\nabla f_{\mh}(x^{t-1})}^2\right]+2\ep\left[f_{\mh}(x^{t-1})-f_{\mh}^*\right]+(1-\gamma_tL)c_1\ep\left[\norm{\delta^t}^2\right]\\
                &\quad+(1-\gamma_tL)c_2\ep\left[\Gamma_{\mh}^{t-1}\right]+\gamma_t^2\left(\frac{162L}{h}+756\kappa L\right)\sigma^2+21\gamma_t\kappa G^2.
            \end{aligned}
        \end{equation}
    \end{lemma}
    \begin{proof}
       Recall the definition of the Lyapunov function $V^{t+1}$ in~\eqref{eq:lyapunov}: 
        \[V^{t+1}=2\ep\left[f_{\mh}(x^{t})-f_{\mh}^*\right]+c_1\ep\left[\norm{\delta^{t+1}}^2\right]+c_2\ep\left[\Gamma_{\mh}^{t}\right].\]
        Using the recursion rules for $f(x^{t})$ in Lemma \ref{lem:objective} and $\norm{\delta^{t+1}}$ in Lemma \ref{lem:delta} gives
        \begin{align*}
            V^{t+1}&\leq 2\ep\left[f_{\mh}(x^{t-1})-f_{\mh}^*\right]-\gamma_t(1-4\gamma_tL)\ep\left[\norm{\nabla f_{\mh}(x^{t-1})}^2\right]\\
            &\quad+2\gamma_t(1+2\gamma_tL)\ep\left[\norm{\delta^t}^2\right]+2\gamma_t(1+\gamma_tL)\ep\left[\norm{\zeta^t}^2\right]\\
            &\quad+c_1\biggl((1+\gamma_{t} L)(1+4\gamma_{t} L)\beta_{t+1}^2\ep\left[\norm{\delta^t}^2\right]+4\gamma_tL(1+\gamma_tL)\beta_{t+1}^2\ep\left[\norm{\nabla f_{\mh}(x^{t-1})}^2\right]\\
            &\quad+\frac{(1-\beta_{t+1})^2}{h}\sigma^2+2\gamma_tL(1+\gamma_tL)\beta_{t+1}^2\ep\left[\norm{\zeta^t}^2\right]\biggr)\\
            &\quad+c_2\ep\left[\Gamma_{\mh}^{t}\right].
        \end{align*}
        Plugging the inequality $\norm{\zeta^t}^2\leq \kappa\Gamma_{\mh}^t$ in Lemma \ref{lem:zeta} into this above, we further obtain
        \begin{align*}
            V^{t+1}&\leq 2\ep\left[f_{\mh}(x^{t-1})-f_{\mh}^*\right]-\gamma_t(1-4\gamma_tL)\ep\left[\norm{\nabla f_{\mh}(x^{t-1})}^2\right]+2\gamma_t(1+2\gamma_tL)\ep\left[\norm{\delta^t}^2\right]\\
            &\quad+c_1\biggl((1+\gamma_{t} L)(1+4\gamma_{t} L)\beta_{t+1}^2\ep\left[\norm{\delta^t}^2\right]+4\gamma_tL(1+\gamma_tL)\beta_{t+1}^2\ep\left[\norm{\nabla f_{\mh}(x^{t-1})}^2\right]+\frac{(1-\beta_{t+1})^2}{h}\sigma^2\biggr)\\
            &\quad+\left(1+\frac{2\kappa\gamma_t(1+\gamma_tL)}{c_2}+\frac{2c_1\kappa\gamma_tL(1+\gamma_tL)\beta_{t+1}^2}{c_2}\right)c_2\ep\left[\Gamma_{\mh}^t\right].
        \end{align*}
        We first evaluate the coefficient of $c_2\ep\left[\Gamma_{\mh}^t\right]$ as follows. By plugging $c_1=1/(8L)$ and $c_2=\kappa/(2L)$, we have
        \begin{align*}
            1+\frac{2\kappa\gamma_t(1+\gamma_tL)}{c_2}+\frac{2c_1\kappa\gamma_tL(1+\gamma_tL)\beta_{t+1}^2}{c_2}&=1+4\gamma_t L(1+\gamma_tL)+\frac{\gamma_tL(1+\gamma_tL)\beta_{t+1}^2}{2}\\
            &\leq 1+\frac{9}{2}\gamma_tL(1+\gamma_tL)\\
            &\leq \frac{7}{6},
        \end{align*}
        where we use $\beta_{t+1}\leq 1$ and $\gamma_t L\leq 1/36$ in the above inequalities. Plugging this upper bound into the above and using the recursion rule of $\Gamma_{\mh}^t$ in Lemma \ref{lem:moment}, we obtain
        \begin{align*}
            V^{t+1}&\leq 2\ep\left[f_{\mh}(x^{t-1})-f_{\mh}^*\right]-\gamma_t(1-4\gamma_tL)\ep\left[\norm{\nabla f_{\mh}(x^{t-1})}^2\right]+2\gamma_t(1+2\gamma_tL)\ep\left[\norm{\delta^t}^2\right]\\
            &\quad+c_1\biggl((1+\gamma_{t} L)(1+4\gamma_{t} L)\beta_{t+1}^2\ep\left[\norm{\delta^t}^2\right]+4\gamma_tL(1+\gamma_tL)\beta_{t+1}^2\ep\left[\norm{\nabla f_{\mh}(x^{t-1})}^2\right]+\frac{(1-\beta_{t+1})^2}{h}\sigma^2\biggr)\\
            &\quad+\left(1+\frac{9}{2}\gamma_tL(1+\gamma_tL)\right)c_2\left(\beta_t\ep\left[\Gamma_\mh^{t-1}\right]+(1-\beta_t)^2\sigma^2+(1-\beta_t)\left(G^2+B^2\ep\left[\norm{\nabla f_{\mh}(x^{t-1})}^2\right]\right)\right).
        \end{align*}
        Rearranging the terms $\ep\left[\norm{\nabla f_{\mh}(x^{t-1})}^2\right]$, $\ep\left[\norm{\delta^t}^2\right]$, and $\ep\left[\Gamma_{\mh}^{t-1}\right]$ and using the inequality $1+\frac{9}{2}\gamma_tL(1+\gamma_tL)\leq 7/6$ gives
        \begin{align}
            V^{t+1}\leq &\left(-\gamma_t(1-4\gamma_tL)+4c_1\gamma_tL(1+\gamma_tL)\beta_{t+1}^2+\frac{7}{6}c_2(1-\beta_t)B^2\right)\ep\left[\norm{\nabla f_{\mh}(x^{t-1})}^2\right]\notag\\
            &\quad+\left(2\gamma_t(1+2\gamma_tL)+c_1(1+\gamma_{t} L)(1+4\gamma_{t} L)\beta_{t+1}^2\right)\ep\left[\norm{\delta^t}^2\right]+\left(1+\frac{9}{2}\gamma_tL(1+\gamma_tL)\right)\beta_t c_2\ep\left[\Gamma_{\mh}^{t-1}\right]\notag\\
            &\quad+2\ep\left[f_{\mh}(x^{t-1})-f_{\mh}^*\right]+\left(\frac{c_1(1-\beta_{t+1})^2}{h}+\frac{7}{6}c_2(1-\beta_t)^2\right)\sigma^2+\frac{7}{6}c_2(1-\beta_t)G^2\notag\\
            &=A_1\gamma_t\ep\left[\norm{\nabla f_{\mh}(x^{t-1})}^2\right]+A_2c_1\ep\left[\norm{\delta^t}^2\right]+A_3c_2\ep\left[\Gamma_{\mh}^{t-1}\right]\notag\\
            &\quad+2\ep\left[f_{\mh}(x^{t-1})-f_{\mh}^*\right]+A_4\sigma^2+A_5G^2,\label{eq:inter:lyapunov}
        \end{align}
        where 
        \begin{align*}
            A_1&:=-1+4\gamma_tL+4c_1L(1+\gamma_tL)\beta_{t+1}^2+\frac{7c_2(1-\beta_t)B^2}{6\gamma_t},\\
            A_2&=\frac{2\gamma_t(1+2\gamma_tL)}{c_1}+(1+\gamma_{t} L)(1+4\gamma_{t} L)\beta_{t+1}^2,\\
            A_3&:=\beta_t+\frac{9}{2}\gamma_tL(1+\gamma_tL)\beta_t,\\
            A_4&:=\frac{c_1(1-\beta_{t+1})^2}{h}+\frac{7}{6}c_2(1-\beta_t)^2,\\
            A_5&:=\frac{7}{6}c_2(1-\beta_t).
        \end{align*}
        Now we provides upper bounds for $A_1,A_2,\ldots,A_5$. 

        \textbf{Upper bound for $A_1$.} Combining $\beta_t=1-36\gamma_tL$, $\beta_{t+1}\leq 1$, $c_1=1/(8L)$, $c_2=\kappa/(2L)$, and $\gamma_tL\leq 1/36$, we obtain
        \begin{align*}
            A_1&=-1+4\gamma_tL+4c_1L(1+\gamma_tL)\beta_{t+1}^2+\frac{7c_2(1-\beta_t)B^2}{6\gamma_t}\\
            &=-1+4\gamma_tL+\frac{1}{2}(1+\gamma_tL)\beta_{t+1}^2+21\kappa B^2\\
            &\leq -1+4\gamma_tL+\frac12(1+\gamma_tL)+21\kappa B^2\\
            &=-\frac12+\frac92\gamma_tL+21\kappa B^2\\
            &\leq -\frac{3}{8}+21\kappa B^2.
        \end{align*}
        \textbf{Upper bound for $A_2$.} Noting that $\beta_{t+1}^2\leq\beta_{t+1}=1-36\gamma_{t+1}L$, and $c_1=1/(8L)$, we obtain
        \begin{align*}
            A_2&=\frac{2\gamma_t(1+2\gamma_tL)}{c_1}+(1+\gamma_{t} L)(1+4\gamma_{t} L)\beta_{t+1}^2\leq \frac{2\gamma_t(1+2\gamma_tL)}{c_1}+(1+\gamma_{t} L)(1+4\gamma_{t} L)\beta_{t+1}\\
            &=16\gamma_tL(1+2\gamma_tL)+(1+\gamma_tL)(1+4\gamma_tL)(1-36\gamma_{t+1}L).
        \end{align*}
        By our assumption on $\{\gamma_t\}$ that $\gamma_{t+1}\geq\frac{2}{3}\gamma_t$, we further derive that
        \begin{align*}
            A_2&\leq 16\gamma_tL(1+2\gamma_tL)+(1+\gamma_tL)(1+4\gamma_tL)(1-24\gamma_{t+1}L)\\
            &=1-3\gamma_tL-84(\gamma_tL)^2-96(\gamma_tL)^3\\
            &\leq 1-\gamma_tL.
        \end{align*}
        \textbf{Upper bound for $A_3$.} By using $\beta_t=1-36\gamma_tL\leq 1$ and $\gamma_tL\leq 1/36$, we have 
        \begin{align*}
            A_3&=\beta_t+\frac{9}{2}\gamma_tL(1+\gamma_tL)\beta_t\\
            &\leq 1-36\gamma_tL+\frac{9}{2}\gamma_tL(1+\gamma_tL)\\
            &\leq 1-\gamma_tL.
        \end{align*}
        \textbf{Upper bound for $A_4$.} Plugging $c_1=1/(8L)$, $c_2=\kappa/(2L)$, and $\beta_t=1-36\gamma_tL$ and noting that $\gamma_{t+1}\leq \gamma_t$, we obtain
        \begin{align*}
            A_4&=\frac{c_1(1-\beta_{t+1})^2}{h}+\frac{7}{6}c_2(1-\beta_t)^2\\
            &=\frac{162\gamma_{t+1}^2L}{h}+756\kappa\gamma_t^2L\\
            &\leq \gamma_t^2\left(\frac{162L}{h}+756\kappa L\right).
        \end{align*}
        \textbf{Upper bound for $A_5$.} Using $c_2=\kappa/(2L)$ and $\beta_t=1-36\gamma_tL$, we have 
        \[A_5=\frac{7}{6}c_2(1-\beta_t)=21\gamma_t\kappa .\]

        Now we plug all these upper bounds into~\eqref{eq:inter:lyapunov} and achieve that 
        \begin{align*}
            V^{t+1}&\leq \gamma_t\left(-\frac{3}{8}+21\kappa B^2\right)\ep\left[\norm{\nabla f_{\mh}(x^{t-1})}^2\right]+2\ep\left[f_{\mh}(x^{t-1})-f_{\mh}^*\right]+(1-\gamma_tL)c_1\ep\left[\norm{\delta^t}^2\right]\\
            &\quad+(1-\gamma_tL)c_2\ep\left[\Gamma_{\mh}^{t-1}\right]+\gamma_t^2\left(\frac{162L}{h}+756\kappa L\right)\sigma^2+21\gamma_t\kappa G^2,
        \end{align*}
        as desired.
    \end{proof}
    \textbf{On the choices of parameters.} We note that the main difficulty caused by the mismatch between $\beta_{t+1}$ and $\beta_t$ arises when bounding the term $A_2$. This difficulty is mitigated if a constant stepsize is employed, allowing for larger momentum parameters, for example, $\beta=1-24\gamma L$ or even $\beta=\sqrt{1-24\gamma L}$. Under such a regime, the coefficient in the upper bound for $A_1$ improves from $21\kappa B^2$ to $14\kappa B^2$, and the condition for convergence relaxes to $\kappa B^2 < 1/28$. Similarly, there is flexibility in selecting other parameters, such as $c_1>1/(4L)$ and $c_2>\kappa/L$.  While we fix specific parameters (such as $c_1$ and $c_2$) to streamline the proof of Theorem \ref{convergence for R-DSGD with momentum}, our framework remains robust to these variations. These parameter adjustments affect only the constant factors, leaving the overall convergence rate and error order unchanged.
    
    The next technical lemma provides an upper bound for the initial Lyapunov function $V^1$.
    \begin{lemma}\label{lem:initial lyapunov}
        Let the Lyapunov function $V^t$ be defined as in \eqref{eq:lyapunov}, and suppose that Assumption \ref{smoothness} holds. Then for any $t\geq 1$, we have
        \[V^1-V^t\leq\frac{9}{4}\left(f_{\mh}(x^0)-f_{\mh}^*\right).\]
    \end{lemma}
    \begin{proof}
        It follows from the definition that 
        \[V^t=2\ep\left[f_{\mh}(x^{t-1})-f_{\mh}^*\right]+c_1\ep\left[\norm{\delta^{t}}^2\right]+c_2\ep\left[\Gamma_{\mh}^{t-1}\right]\geq 0,\]
        hence $V^1-V^t\leq V^1$. Further, since $m_i^t=0$ for all $i\in\mh$, we have $\Gamma_\mh^0=0$. Thus, recalling that $m_i^1=\beta_1m_i^0+(1-\beta_1)\nabla f_i(x^0)=(1-\beta_1)\nabla f_i(x^0)$ and $c_1=1/(8L)$, we obtain
        \begin{align*}
            V^1&= 2\left(f_{\mh}(x^0)-f_{\mh}^*\right)+c_1\ep\left[\norm{\delta^t}^2\right]\\
            &=2\left(f_{\mh}(x^0)-f_{\mh}^*\right)+c_1\ep\left[\norm{\frac1h\sum_{i\in\mh}m_i^1-\nabla f_{\mh}(x^0)}^2\right]\\
            &=2\left(f_{\mh}(x^0)-f_{\mh}^*\right)+c_1\ep\left[\norm{(1-\beta_1)\nabla f_{\mh}(x^0)-\nabla f_{\mh}(x^0)}^2\right]\\
            &=2\left(f_{\mh}(x^0)-f_{\mh}^*\right)+\frac{\beta_1^2}{8L}\norm{\nabla f_{\mh}(x^0)}^2.
        \end{align*}
        Due to the smoothness of $f_{\mh}$, we have $\norm{\nabla f_{\mh}(x^0)}^2\leq 2L\left(f_{\mh}(x^0)-f_{\mh}^*\right)$. Noting that $\beta_1\leq 1$, we obtain
        \[V_1-V^t\leq V_1\leq 2\left(f_{\mh}(x^0)-f_{\mh}^*\right)+\frac{1}{4}\left(f_{\mh}(x^0)-f_{\mh}^*\right)=\frac{9}{4}\left(f_{\mh}(x^0)-f_{\mh}^*\right).\]
        The proof is complete.
    \end{proof}
    \subsection{Proof of Theorem \ref{convergence for R-DSGD with momentum}}\label{sec:proof of R-DSGD-M smooth}
    With the recursive property of the Lypunov function in Lemma~\ref{lem:lyapunov}, we are ready to prove the convergence for R-DSGD-M. We provide a formal description of Theorem \ref{convergence for R-DSGD with momentum} as follows.
    \begin{theorem}
        Suppose that Assumptions \ref{smoothness}, \ref{heterogeneity}-\ref{aggregation} hold, and assume that $\kappa B^2<1/56$. Take $\gamma_t\equiv \gamma$ with $0<\gamma L\leq 1/36$ and $\beta=1-36\gamma L$. Then the iterates $\{x^t\}$ generated by R-DSGD-M satisfy
        \[
            \frac1T\sum_{t=0}^{T-1}\ep\left[\norm{\nabla f_{\mh}(x^t)}^2\right]\leq \frac{9(f(x^0)-f_{\mh}^*)}{4(\frac{3}{8}-21\kappa B^2)\gamma T}+\frac{\frac{162L}{h}+756\kappa L}{\frac{3}{8}-21\kappa B^2}\gamma\sigma^2+\frac{21\kappa G^2}{\frac{3}{8}-21\kappa B^2},
        \]
        where $f_{\mh}^*=\inf_{x\in\rr^d}f_{\mh}(x)$.
    \end{theorem}
    \begin{proof}
        Since we consider the constant stepsize $\gamma_t\equiv\gamma$, the requirements of $\{\gamma_t\}$ in Lemma \ref{lem:lyapunov} are satisfied. 
        Using the recursion rule of the Lyapunov function in Lemma \ref{lem:lyapunov}, we obtain 
        \begin{align*}
            V^{t+1}-V^t&\leq \gamma\left(-\frac{3}{8}+21\kappa B^2\right)\ep\left[\norm{\nabla f_{\mh}(x^{t-1})}^2\right]-\gamma Lc_1\ep\left[\norm{\delta^t}^2\right]\\
                &\quad-\gamma Lc_2\ep\left[\Gamma_{\mh}^{t-1}\right]+\gamma^2\left(\frac{162L}{h}+756\kappa L\right)\sigma^2+21\gamma\kappa G^2\\
                &\leq \gamma\left(-\frac{3}{8}+21\kappa B^2\right)\ep\left[\norm{\nabla f_{\mh}(x^{t-1})}^2\right]+\gamma^2\left(\frac{162L}{h}+756\kappa L\right)\sigma^2+21\gamma\kappa G^2,
        \end{align*}
        since $\norm{\delta^t}^2\geq 0$ and $\Gamma_{\mh}^{t-1}\geq 0$ by their definitions in~\eqref{eq:notations}. 
        Since $\kappa B^2<1/56$, we can arrange the above inequality as follows:
        \[\ep\left[\norm{\nabla f_{\mh}(x^{t-1})}^2\right]\leq \left(\frac{3}{8}-21\kappa B^2\right)^{-1}\left(\frac{V^{t}-V^{t+1}}{\gamma}+\gamma\left(\frac{162L}{h}+756\kappa L\right)\sigma^2+21\kappa G^2\right).\]
        Telescoping the above inequality from $t=1$ to $T$, we obtain 
        \[\frac{1}{T}\sum_{t=0}^{T-1}\ep\left[\norm{\nabla f_{\mh}(x^t)}^2\right]\leq \frac{V^1-V^{T+1}}{\left(\frac{3}{8}-21\kappa B^2\right)\gamma T}+\frac{\frac{162L}{h}+756\kappa L}{\frac{3}{8}-21\kappa B^2}\gamma\sigma^2+\frac{21\kappa G^2}{\frac{3}{8}-21\kappa B^2}.\]
        Plugging the upper bound for $V^1-V^{T+1}$ in Lemma \ref{lem:initial lyapunov} yields the desired result.
    \end{proof}
    
    \subsection{Proof of Theorem \ref{convergence for PL R-DSGD-M}}\label{sec:proof of R-DSGD-M PL}
    In this part, we provide the proof of the convergence result for R-DSGD-M under the PL condition. We provide a formal form of Theorem \ref{convergence for PL R-DSGD-M} as below.
    \begin{theorem}
        Suppose that Assumptions \ref{smoothness}-\ref{aggregation} hold. Denote 
        \begin{align*}
            \alpha_4&=\left(\frac{3}{8}-21\kappa B^2\right)\mu,\\
            \alpha_5&=\left(\frac{162L}{h}+756\kappa L\right)\sigma^2,\\
            \alpha_6&=21\kappa G^2.
        \end{align*}
        For any given $T>0$, suppose that the stepsizes $\{\gamma_t\}$ are chosen as
        \[\gamma_t =\begin{cases}
            \gamma_0, & \mathrm{if}\;t< \left\lfloor T/2\right\rfloor,\\ 
            \frac{2}{\alpha_4\left(s_0+t-\left\lfloor T/2\right\rfloor\right)}, &\mathrm{if}\;t\geq \left\lfloor T/2\right\rfloor,
        \end{cases}\]
        with $s_0>\max\{2,\frac{72L}{\alpha_4}\}$, $\gamma_0=\frac{2}{\alpha_4s_0}<\min\{\frac{1}{\alpha_4},\frac{1}{36L}\}$. Suppose that the momentum parameters $\{\beta_t\}$ are chosen as $\beta_t=1-36\gamma_tL$, and assume that $\kappa B^2<1/56$.
        Then the iterates $\{x^{t}\}$ generated by R-DSGD-M satisfy
        \begin{align*}
            \ep[f_{\mh}(x^{T-1})-f_{\mh}(x^*)]&\leq (s_0-1)^2\left(f_{\mh}(x^0)-f_{\mh}^*\right)\frac{2\exp\left(-\alpha_4\gamma_0(\left\lfloor T/2\right\rfloor-1)\right)}{T^2}\\
            &\quad+\frac{2(\alpha_5\gamma_0+\alpha_6)(s_0-1)^2}{\alpha_4T^2}+\frac{4\alpha_5}{\alpha_4^2T}+\frac{\alpha_6}{\alpha_4}.
        \end{align*}
    \end{theorem}
    \begin{proof}
        We first check that the requirements (i) $\gamma_1\geq\gamma_2\geq\cdots$ and (ii) $\gamma_{t+1}\geq\frac{2}{3}\gamma_t$ for stepsizes $\{\gamma_t\}$ in Lemma \ref{lem:lyapunov} are satisfied. Since $\gamma_0=\frac{2}{\alpha_4s_0}$, we have $\gamma_{\lfloor T/2\rfloor}=\gamma_0$. Now,
        if $t<\lfloor T/2\rfloor$, then $\gamma_{t+1}=\gamma_t=\gamma_0$; if $t\geq \lfloor T/2\rfloor$, then 
        \[1\geq \frac{\gamma_{t+1}}{\gamma_t}=\frac{s_0+t-\lfloor T/2\rfloor}{s_0+t+1-\lfloor T/2\rfloor}\geq\frac{s_0}{s_0+1}\geq\frac{2}{3},\]
        due to $s_0\geq 2$. Hence, we can apply the recursion rule of the Lyapunov function from Lemma \ref{lem:lyapunov}:
        \begin{align*}
            V^{t+1}&\leq \gamma_t\left(-\frac{3}{8}+21\kappa B^2\right)\ep\left[\norm{\nabla f_{\mh}(x^{t-1})}^2\right]+2\ep\left[f_{\mh}(x^{t-1})-f_{\mh}^*\right]+(1-\gamma_tL)c_1\ep\left[\norm{\delta^t}^2\right]\\
            &\quad+(1-\gamma_tL)c_2\ep\left[\Gamma_{\mh}^{t-1}\right]+\gamma_t^2\left(\frac{162L}{h}+756\kappa L\right)\sigma^2+21\gamma_t\kappa G^2.
        \end{align*}
        
        Plugging the PL condition $\norm{\nabla f_{\mh}(x^{t-1})}^2\geq 2\mu (f_{\mh}(x^{t-1})-f_\mh^*)$ into the recursion of $V^{t+1}$, we further obtain 
        \begin{equation}\label{eq:inter:pl proof}
            \begin{aligned}
                V^{t+1}&\leq \left(1-\gamma_t\mu\left(\frac{3}{8}-21\kappa B^2\right)\right)2\ep\left[f_{\mh}(x^{t-1})-f_{\mh}^*\right]+\left(1-\gamma_tL\right)\left(c_1\ep\left[\norm{\delta^t}^2\right]+c_2\ep\left[\Gamma^{t-1}\right]\right)\\
                &\quad+\gamma_t^2\left(\frac{162L}{h}+756\kappa L\right)\sigma^2+21\gamma_t\kappa G^2.
            \end{aligned}
        \end{equation}
        Since $f_{\mh}$ is $L$-smooth and satisfies the PL condition with parameter $\mu$, the following inequality holds for all $x\in\rr^d$:
        \[2\mu(f_{\mh}(x)-f_{\mh}^*)\leq \norm{\nabla f_{\mh}(x)}^2\leq 2L(f_{\mh}(x)-f_{\mh}^*).\]
        This implies that $\mu\leq L$. Furthermore, given the assumption that $\kappa B^2<1/56$, we obtain
        \begin{equation}
            1-\gamma_tL\leq 1-\gamma_t\mu\leq 1-\gamma_t\mu\left(\frac{3}{8}-21\kappa B^2\right)=1-\alpha_4\gamma_t.
        \end{equation}
        Further plugging this inequality into ~\eqref{eq:inter:pl proof} and recalling the definition of $V^t$ in ~\eqref{eq:lyapunov}, we obtain 
        \begin{align*}
            V^{t+1}&\leq \left(1-\alpha_4\gamma_t\right)\left(2\ep\left[f_{\mh}(x^{t-1})-f_{\mh}^*\right]+c_1\ep\left[\norm{\delta^t}^2\right]+c_2\ep\left[\Gamma_{\mh}^{t-1}\right]\right)\\
            &\quad+\gamma_t^2\left(\frac{162L}{h}+756\kappa L\right)\sigma^2+21\gamma_t\kappa G^2\\
            &=(1-\alpha_4\gamma_t)V^t+\alpha_5\gamma_t^2+\alpha_6\gamma_t.
        \end{align*}
        The result follows immediately by invoking Lemma \ref{lemma for PL}.
    \end{proof}
    
    \section{Proof of the Lower Bounds}\label{sec: Proof of the Lower Bounds}
    In this section, we prove that the upper bounds $\mo\left(\frac{\kappa G^2}{1-\kappa B^2}\right)$ and $\mo\left(\frac{\kappa\sigma^2}{1-\kappa B^2}\right)$ in our convergence analysis are tight. 
    Notably, it suffices to construct such problem instances and aggregation rules in a one-dimensional setting with $\abs{\mh}=n=2$ and $b=0$ to establish these lower bounds and to validate the tightness of the upper bounds, since the upper bounds are independent of $n$ and $b$. 
    
    The proof relies on the observation that the discrepancies among local updates can be exploited by an adversarial aggregation rule satisfying the robustness Assumption 2.5. In short, given honest updates $\{g_i\}_{i=1}^2$, the oracle aggregation rule first calculates the variance $V^2(g_1,g_2)=\frac{1}{2}\sum_{i=1}^2\norm{g_i-\bar{g}}^2$, and then outputs $\bar{g}-\sqrt{\kappa}V(g_1,g_2)$. Intuitively, this adversarial aggregation rule knows the honest workers, yet deliberately introduce an error controlled by $\kappa$ while satisfying the $(0,\kappa)$-robustness. With such an aggregation rule, the aggregated update no longer corresponds to the gradient of the original objective, but instead behaves as the gradient of a drifted one. Consequently, as the iteration proceeds, the iterates are gradually steered away from the true optimizer and converge to the minimizer of this drifted objective, which induces a Byzantine error.
    \subsection{Lower Bound for Heterogeneity}\label{sec:heterogeneity bound}
    We now establish the lower bound for the Byzantine error induced by data heterogeneity. 
    \begin{proof}[Proof of Theorem \ref{thm for lower bound for heterogeneity}]
    Consider the two honest workers in a one-dimensional setting with $\mh =\{1,2\}$ and $d=1$.
     We construct the following local objective functions for all $x\in\rr^1$ as
        \[f_1(x)=\frac{\mu+\delta}{2}x^2+\epsilon x,\; f_2(x)=\frac{\mu-\delta}{2}x^2-\epsilon x\]
        with $\mu>0$ and the parameters $ \delta, \epsilon$ to be chosen later. The global objective function $f_{\mh}$ is then
        \[f_\mh(x):=\frac{f_1(x)+f_2(x)}{2}=\frac{\mu}{2}x^2,\]
        which is smooth and strongly convex with a common parameter $\mu>0$. The heterogeneity between the local and global gradients is given by
        \[\frac{1}{2}\sum_{i=1}^2(\nabla f_i(x)-\nabla f_\mh(x))^2=\frac12\left((\delta x+\epsilon)^2+(-\delta x-\epsilon)^2\right)=(\delta x+\epsilon)^2.\]\par
        Given any $G,B>0$, we now choose $\delta,\epsilon$ such that the local functions $\{f_i:i=1,2\}$ satisfies the $(G,B)$-bounded dissimilarity condition, which requires
        \[(\delta x+\epsilon)^2=\frac{1}{2}\sum_{i=1}^2(\nabla f_i(x)-\nabla f_\mh(x))^2\leq G^2+B^2(\nabla f_\mh(x))^2=G^2+B^2\mu^2x^2\]
        for all $x \in \rr^1$. Rearranging terms, this is equivalent to that the quadratic inequality
        \[(\delta^2-B^2\mu^2)x^2+2\delta\epsilon x+\epsilon^2-G^2\leq 0\] holding for all $x \in \rr^1$. For a quadratic $ax^2 + bx +c \leq 0$ for all $x \in \rr^1$, it is sufficient that $a < 0$ and the discriminant $\Delta = b^2 - 4 a c \leq 0$. These conditions are simplified to 
        \begin{align*}
\delta^2<B^2\mu^2,\;\Delta=4\delta^2\epsilon^2-4(\delta^2-B^2\mu^2)(\epsilon^2-G^2)\leq 0,
        \end{align*}
        which is satisfied, for example, by setting $\delta=\sqrt{3}B\mu/2$ and $\epsilon=G/2$. 
        
        Now, for any given $\kappa>0$, we define the aggregation rule $\ma$.

        \textbf{Aggregation rule for R-DGD}. Given any $\kappa>0$, we define $\ma$ as 
        \[\ma(\nabla f_1(x),\nabla f_2(x))=\frac{\nabla f_1(x)+\nabla f_2(x)}{2}-\sqrt{\kappa}(\delta x+\epsilon)=(\mu-\sqrt{\kappa}\delta)x-\sqrt{\kappa}\epsilon.\]
        Then $\ma$ is $(b,\kappa)$-robust, since 
        \[\left(\ma(\nabla f_1(x),\nabla f_2(x))-\nabla f_{\mh}(x)\right)^2=\kappa(\delta x+\epsilon)^2=\kappa \cdot \frac12\sum_{i=1}^2\left(\nabla f_i(x)-\nabla f_{\mh}(x)\right)^2.\]
        However, performing R-DGD with the aggregation rule $\ma$ is equivalent to performing the vanilla gradient descent method to minimize the function 
        \begin{equation}\label{misleading function in heterogeneity}
            F(x)=\frac{\mu-\sqrt{\kappa}\delta}{2}x^2-\sqrt{\kappa}\epsilon x.
        \end{equation}

        \textbf{Aggregation rule for R-DGD-M}. For any iteration $t\geq 1$ and the past iterates $\{x^j\}_{j=1}^{t-1}$ (with $x^0$ as the initial iterate), the local momentum for each honest worker $i=1,2$ is 
        \[m_i^t=\beta_t m_i^{t-1}+(1-\beta_t)\nabla f_i(x^{t-1})=\cdots=\sum_{j=1}^t\left((1-\beta_j)\prod_{k=j+1}^{t}\beta_k\right)\nabla f_i(x^{j-1}),\]
        with $m_i^0=0$. Denote $\theta_j:=(1-\beta_j)\prod_{k=j+1}^{t}\beta_k$.
        The aggregated honest momentum is then
        \[m_{\mh}^t=\frac{m_1^t+m_2^t}{2}=\sum_{j=1}^t\theta_j\nabla f_{\mh}(x^{j-1}).\]
        Given any $\kappa>0$, we define the aggregation rule $\ma'$ for each $t\geq 1$ as 
        \begin{align*}
            \ma'(m_1^t,m_2^t)&=m_{\mh}^t-\sqrt{\kappa}\sum_{j=1}^t\theta_j(\delta x^{j-1}+\epsilon)\\
            &=\sum_{j=1}^t\theta_j\left(\nabla f_{\mh}(x^{t-j})-\sqrt{\kappa}(\delta x^{j-1}+\epsilon)\right).
        \end{align*}
        Since $\nabla f_i(x)-\nabla f_{\mh}(x)=\pm(\delta x+\epsilon)$, we have
        \begin{align*}
            \frac12\sum_{i=1}^2(m_i^t-m_{\mh}^t)^2&=\frac12\left(\left(\sum_{j=1}^t\theta_j(\delta x^{t-j}+\epsilon)\right)^2+\left(\sum_{j=1}^t\theta_j(-\delta x^{t-j}-\epsilon)\right)^2\right)\\
            &=\left(\sum_{j=1}^t\theta_j(\delta x^{t-j}+\epsilon)\right)^2\\
            &=\frac{1}{\kappa}(\ma'(m_1^t,m_2^t)-m_{\mh}^t)^2,
        \end{align*}
        which implies that $\ma'$ is $(b,\kappa)$-robust. However, performing R-DGD-M with this aggregation rule $\ma'$ is equivalent to performing gradient descent with momentum to minimize $F(x)$ in (\ref{misleading function in heterogeneity}). 
        
        Note that $\kappa B^2<1$, and that $F(x)$ is a quadratic function on $\rr$ with the quadratic coefficient 
        \[\frac{\mu-\sqrt{\kappa}\delta}{2}=\frac{\mu\left(1-\frac{\sqrt{3}}{2}\sqrt{\kappa}B\right)}{2}>0,\]
        and hence $F(x)$ is strongly convex.
        Consequently, by the classical convergence result of GD (with or without momentum) for strongly convex objectives, for sufficient large $T$, the iterate $x^{T-1}$ satisfies $\abs{x^{T-1}-x_F^*} \leq \abs{x_F^*}/4$, where 
        \[x_F^*=\frac{\sqrt{\kappa}\epsilon}{\mu-\sqrt{\kappa}\delta}\]
        is the minimizer of $F(x)$. Thus, 
        \[(x^{T-1})^2=(x^{T-1}-x_F^*+x_F^*)^2\geq (x^{T-1}-x_F^*)^2+(x_F^*)^2-2\abs{x^{T-1}-x_F^*}\abs{x_F^*}\geq \frac12 (x_F^*)^2\]
        and the gradient error at $x^{T-1}$ is lower bounded by
        \begin{align*}
            (\nabla f_{\mh}(x^{T-1}))^2=\mu^2(x^{T-1})^2 &\geq \frac{\mu^2}{2}(x_{F}^*)^2  =\frac{\mu^2\kappa\epsilon^2}{2(\mu-\sqrt{\kappa}\delta)^2}=\frac{1}{8}\frac{\kappa G^2}{\left(1-\sqrt{\frac{3\kappa}{4}}B\right)^2}\\
            &\geq \frac{1}{8}\frac{\kappa G^2}{1-\frac34\kappa B^2}=\mo\left(\frac{\kappa G^2}{1-\kappa B^2}\right),
        \end{align*}
        where the last inequality follows from $(1-a)^2\leq 1-a^2$ for $a\in[0,1]$.
        Finally, for the function value sub-optimality gap,  we have
        \[f_{\mh}(x^{T-1})=\frac{\mu}{2}(x_F^*)^2\geq \mo\left(\frac{\kappa G^2}{\mu(1-\kappa B^2)}\right).\]
        This verifies the tightness of the lower bound.
    \end{proof}
    \subsection{Lower Bound for Randomness}
    Before proving the lower bound $\mo\left(\frac{\sigma^2}{1-\kappa B^2}\right)$, we first introduce the following convergence results for SGD from the work of \citet{nguyen2019new} and \citet{gorbunov2020unified}.
    \begin{lemma}[\cite{nguyen2019new}, Theorems 9, 11]\label{advanced convergence result for SGD}
        Consider the following stochastic optimization problem:
        \[\min_{x\in\rr^d} F(x):=\ep_{\xi}\left[f(x;\xi)\right].\]
        Suppose that $F(x)$ is strongly convex with parameter $\mu > 0$, $f(x;\xi)$ is smooth with parameter $L$ and convex for every realization of $\xi$, and that $\mu \leq L$. Suppose that SGD 
        \[x^{t+1}=x^t-\gamma_t\nabla f(x;\xi),\]
       where the stepsize $\{\gamma_t\}$ satisfies $\gamma_t=1/(K+t)^q$ with constants $q, K>0$ and $\gamma_t\leq 1/(2L)$, is used to solve this problem. Then, we have 
        \[\lim_{t\to+\infty}\ep\left[\norm{x^t-x^*}^2\right]=0,\]
        where $x^*=\argmin_{x\in\rr^d}F(x)$.
    \end{lemma}
    \begin{lemma}[\cite{gorbunov2020unified}, Theorem 4.1]\label{advanced nonasymptotic convergence result for SGD}
        Consider the following stochastic optimization problem:
        \[\min_{x\in\rr^d} F(x):=\ep_{\xi}\left[f(x;\xi)\right].\]
        Suppose that $F(x)$ is $\mu$-strongly convex and $L$-smooth for some constants $L\geq\mu>0$. Suppose that SGD is used to solve this problem
        \[x^{t+1}=x^t-\gamma_t\nabla f(x;\xi),\]
        and there exist constant $M,M_V\geq 0$ such that for all $x\in\rr^d$, the stochastic gradient $\nabla f(x;\xi)$ satisfies
        \[\ep_{\xi}\left[\nabla f(x;\xi)\right]=\nabla F(x),\;\ep_{\xi}\left[\norm{\nabla f(x;\xi)}^2\right]\leq M+M_V\norm{\nabla F(x)}^2.\]
        Then if the stepsize is chosen as $\gamma_t=\gamma\leq\frac{1}{M_VL}$, then the iterates $\{x^t\}$ generated by SGD satisfy 
        \[\ep\left[\norm{x^t-x^*}^2\right]\leq (1-\gamma\mu)^t\norm{x^0-x^*}^2+\frac{\gamma}{\mu} M.\]
        Consequently, for any $\epsilon>0$, if we take $\gamma=\mo(1/T^{\alpha})$ with $0<\alpha<1$ and $T$ large enough, then we have \[\ep\left[\norm{x^T-x^*}^2\right]\leq\epsilon.\]
    \end{lemma}
    
    Using these convergence results, we can establish the lower bound for the Byzantine error induced by stochastic noise. The proof follows a similar yet more intricate line of reasoning in Appendix~\ref{sec:heterogeneity bound}. For a construction of local objective functions satisfying the $(0,B)$-bounded dissimilarity and a rule of generating stochastic gradients satisfying Assumption~\ref{sfo}, we design an adversarial aggregation rule such that the robustness assumption is satisfied for every realization of the stochastic gradients. Moreover, the aggregated direction is an unbiased estimator of the gradient of a drifted objective function and satisfies the requirements in Lemmas~\ref{advanced convergence result for SGD} and \ref{advanced nonasymptotic convergence result for SGD}. Consequently, when the number of iterations is sufficiently large, the sequence of iterates converges in expectation to the minimizer of the drifted objective rather than that of the original one, thereby giving rise to a Byzantine error.
    \begin{proof}[Proof of Theorem \ref{thm for lower bound for randomness}]
    First, we consider two honest workers $\mh =\{1,2\}$ in a one-dimensional setting. Given constant $B\geq 0$, the local objective functions are defined as
        \[f_1(x)=\frac{(1+B)\mu}{2}x^2,\;f_2(x)=\frac{(1-B)\mu}{2}x^2.\]
       The global objective function $f_{\mh}$ is then
        \[f_{\mh}(x)=\frac{f_1(x)+f_2(x)}{2}=\frac{\mu}{2}x^2,\]
      which is smooth and strongly convex with a common parameter $\mu$. Note that 
        \[\frac12\sum_{i=1}^2(\nabla f_i(x)-\nabla f_{\mh}(x))^2=B^2\mu^2x^2=B^2(\nabla f_{\mh}(x))^2,\]
        hence the local functions $\{f_1,f_2\}$ satisfy the $(0,B)$-bounded dissimilarity assumption. 
        
        Next, we construct the unbiased stochastic gradients for each local objective function with bounded variance, satisfying Assumption \ref{sfo}. Given $\sigma>0$, let $\xi_1,\xi_2$ be independent Bernoulli random variables with $\mathbb{P}(\xi_i=0)=\mathbb{P}(\xi_i=1)=1/2$. For any $i\in\{1,2\}$, the stochastic gradients are defined as follows:
        \[
            g_i(x;\xi_i)=\begin{cases}
            \nabla f_i(x)+\sigma, &\text{if}\;\xi_i=0,\\
            \nabla f_i(x)-\sigma, &\text{if}\;\xi_i=1.\\
        \end{cases}
        \]
        These satisfy Assumption \ref{sfo} since for any $i\in\{1,2\}$, 
        \[\ep_{\xi_i}[g_i(x;\xi_i)]=\nabla f_i(x),\;\ep_{\xi_i}\left[(g_i(x)-\nabla f_i(x))^2\right]=\sigma^2. \] Let $\xi=\xi_1+\xi_2$. The average stochastic gradient is
        \[\bar{g}(x;\xi):=\frac{g_1(x;\xi_1)+g_2(x;\xi_2)}{2}=\mu x+\frac12\left[-\xi\sigma+(2-\xi)\sigma\right]=\mu x+(1-\xi)\sigma,\]
        with the expectation $\ep[\bar{g}(x;\xi)]=\mu x$.

        For any given $\kappa>0$, we define the output of a specific aggregation rule $\ma$ for each realization of $(\xi_1,\xi_2)$. Let
        \[\nu(x;\xi_1,\xi_2):=\frac12\sum_{i=1}^2(g_i(x;\xi_i)-\bar{g}(x;\xi))^2.\]
        We enumerate all the possible realizations:
        
        \textbf{Case I}: $\xi_1=\xi_2=0$. Then 
        \[g_1(x;\xi_1)=(1+B)\mu x+\sigma,\; g_2(x;\xi_2)=(1-B)\mu x+\sigma,\;\bar{g}_{\xi}(x)=\mu x+\sigma,\]
        and hence 
       \[\nu(x;\xi_1,\xi_2)=\frac{(B\mu x)^2+(-B\mu x)^2}{2}=(B\mu x)^2.\]
        \textbf{Case II}: $\xi_1=1,\;\xi_2=0$. Then 
        \[g_1(x;\xi_1)=(1+B)\mu x-\sigma,\;g_2(x;\xi_2)=(1-B)\mu x+\sigma,\;\bar{g}_{\xi}(x;\xi)=\mu x,\]
        and hence 
        \[\nu(x;\xi_1,\xi_2)=\frac{(B\mu x-\sigma)^2+(-B\mu x+\sigma)^2}{2}=(B\mu x-\sigma)^2.\]
        \textbf{Case III}: $\xi_1=0,\;\xi_2=1$. Then 
        \[g_1(x;\xi_1)=(1+B)\mu x+\sigma,\; g_2(x;\xi_2)=(1-B)\mu x-\sigma,\;\bar{g}(x;\xi)=\mu x,\]
        and hence 
        \[\nu(x;\xi_1,\xi_2)=\frac{(B\mu x+\sigma)^2+(-B\mu x-\sigma)^2}{2}=(B\mu x+\sigma)^2.\]
        \textbf{Case IV}: $\xi_1=\xi_2=1$. Then 
        \[g_1(x;\xi_1)=(1+B)\mu x-\sigma,\;g_2(x;\xi_2)=(1-B)\mu x-\sigma,\;\bar{g}(x;\xi)=\mu x-\sigma,\]
        and hence 
        \[\nu(x;\xi_1,\xi_2)=\frac{(B\mu x)^2+(-B\mu x)^2}{2}=(B\mu x)^2.\]
        Now we define the aggregation rule $\ma$ as 
        \[\ma(x;\xi_1,\xi_2)=\bar{g}(x;\xi)-\sqrt{\kappa}W(x;\xi_1,\xi_2),\]
        where 
        \begin{equation}\label{drift gradient for lower bound}
            W(x;\xi_1,\xi_2):=\begin{cases}
            B\mu x,&\text{if}\;\xi_1=\xi_2,\\
            -B\mu x+\sigma, &\text{if}\;\xi_1=1\;\mathrm{and}\;\xi_2=0,\\
            B\mu x+\sigma,  &\text{if}\;\xi_1=0\;\mathrm{and}\;\xi_2=1.
        \end{cases}
        \end{equation}
        Then for any realization of $\xi_1$ and $\xi_2$, we have 
        \[\left(\ma(x;\xi_1,\xi_2)-\bar{g}(x;\xi)\right)^2=\kappa\left(W(x;\xi_1,\xi_2)\right)^2=\kappa \nu(x;\xi_1,\xi_2),\]
       which implies that $\ma$ is a $(b,\kappa)$-robust aggregation rule satisfying Assumption \ref{aggregation}. 
        
        Now we define the function 
        \begin{align}\label{func:F:random}
        F(x)&=\frac{\mu}{2}\left(1-\frac{\sqrt{\kappa}}{2}B\right)x^2-\frac{\sqrt{\kappa}}{2}\sigma x.
        \end{align}
        Since $\kappa B^2<1$, $F(x)$ is strongly convex, and its global minimizer is
        \[x^*_F=\frac{\frac{\sqrt{\kappa}}{2}\sigma}{\mu\left(1-\frac{\sqrt{\kappa}}{2}B\right)}. \]
        Note that 
        \[\ep_{\xi_1,\xi_2}\left[W(x;\xi_1,\xi_2)\right]=\frac12(B\mu x)+\frac14(-B\mu x+\sigma)+\frac14(B\mu x+\sigma)=\frac12(B\mu x+\sigma).\]
        Then, $\ma(x;\xi_1,\xi_2)$ is an unbiased estimator of $\nabla F(x)$ since
        \begin{align*}  \ep_{\xi_1,\xi_2}\left[\ma(x;\xi_1,\xi_2)\right]&=\ep_{\xi_1,\xi_2}\left[\bar{g}(x;\xi)\right]-\sqrt{\kappa}\ep_{\xi_1,\xi_2}\left[W(x;\xi_1,\xi_2)\right]\\
        &=\mu x-\frac{\sqrt{\kappa}}{2}(B\mu x+\sigma)=\mu\left(1-\frac{\sqrt{\kappa}}{2}B\right)x-\frac{\sqrt{\kappa}}{2}\sigma\\
        &=\nabla F(x).
        \end{align*}
Therefore, treating $\ma(x;\xi_1,\xi_2)$ as a stochastic gradient of $F(x)$, running D-RSGD with this aggregation rule is equivalent to minimizing $F(x)$ with vanilla SGD. 
        
  Here, we verify that the assumptions in Lemma \ref{advanced convergence result for SGD} hold. Since $\kappa B^2<1$, we have $1-\frac{\sqrt{\kappa}}{2}B>0$ and hence the function $F(x)$ defined in \eqref{func:F:random} is smooth and strongly convex. Next, we construct the function $f(x;\xi_1,\xi_2)$: 
        \[f(x;\xi_1,\xi_2):=\frac{\mu}{2}x^2+(1-\xi)\sigma x-\sqrt{\kappa}\tilde{W}(x;\xi_1,\xi_2),\]
        where $\xi = \xi_1 + \xi_2$ and $\tilde{W}(x;\xi_1,\xi_2)$ is defined as:
        \[\tilde{W}(x;\xi_1,\xi_2)=\begin{cases}
            \frac12 B\mu x^2,&\text{if}\;\xi_1=\xi_2,\\
            -\frac12B\mu x^2+\sigma x, &\text{if}\;\xi_1=1\;\mathrm{and}\;\xi_2=0,\\
            \frac12 B\mu x^2+\sigma x, &\text{if}\;\xi_1=0\;\mathrm{and}\;\xi_2=1.
        \end{cases}\]
        Thus, for any realization of $(\xi_1,\xi_2)$, we have $\ma(x;\xi_1,\xi_2)=\nabla f(x;\xi_1,\xi_2)$. 
   By simple calculations, we obtain
        \begin{align*}
            \ep_{\xi_1,\xi_2}\left[\tilde{W}(x;\xi_1,\xi_2)\right]&=\frac12\left[\frac12 B\mu x^2\right]+\frac14\left[-\frac12B\mu x^2+\sigma x\right]+\frac14\left[\frac12 B\mu x^2+\sigma x\right]\\
            &=\frac14B\mu x^2+\frac12\sigma x;\\
            \ep_{\xi_1,\xi_2}\left[(1-\xi)\sigma x\right]&=0.
        \end{align*}
        Thus, $f(x;\xi_1,\xi_2)$ is the unbiased estimate of $F(x)$ for any realization of $(\xi_1,\xi_2)$, i.e., 
        \[\ep_{\xi_1,\xi_2}\left[f(x;\xi_1,\xi_2)\right]=\frac{\mu}{2}x^2-\sqrt{\kappa}\left( \frac14B\mu x^2+\frac12\sigma x\right)=F(x).\]
        Moreover, since $\kappa B^2<1$, for any realization of $(\xi_1,\xi_2)$, $f(x;\xi_1,\xi_2)$ is a quadratic  with positive leading coefficient, thus is smooth and strongly convex.  Consequently, the assumptions in Lemma \ref{advanced convergence result for SGD} hold. 
        
        We now verify that the assumptions in Lemma \ref{advanced nonasymptotic convergence result for SGD} hold. Since we have already confirmed the required assumptions of $F(x)$ in Lemma \ref{advanced convergence result for SGD}, 
        we focus here on the assumptions w.r.t. stochastic gradients, i.e., that there exist constants $M,M_V\geq 0$ such that 
        \[\mathrm{Var}(\ma(x;g_1,g_2))\leq M+M_V|\nabla F(x)|^2.\]
        Recalling that $\bar{g}(x,\xi)=\mu x+(1-\xi)\sigma$, we obtain
        \begin{align*}
            \mathrm{Var}(\bar{g}(x;\xi))=\ep_{\xi}\left[\left(\bar{g}(x;\xi)-\ep_{\xi}[\bar{g}(x;\xi)]\right)^2\right]=\ep_{\xi}\left[(1-\xi)^2\sigma^2\right]=\frac12\sigma^2.
        \end{align*}
        For $W(x;\xi_1,\xi_2)$ defined in~\eqref{drift gradient for lower bound}, by using the inequality $(a+b)^2\leq 2a^2+2b^2$, we have
        \begin{align*}
            \mathrm{Var}(W(x;\xi_1,\xi_2))&=\ep_{\xi_1,\xi_2}\left[\left(W(x;\xi_1,\xi_2)-\ep_{\xi_1,\xi_2}[W(x;\xi_1,\xi_2)]\right)^2\right]\\
            &=\frac12 \left(\frac12 B\mu x-\frac12\sigma\right)^2+\frac14\left(-\frac32 B\mu x+\frac12\sigma\right)^2+\frac14\left(\frac{1}{2}B\mu x+\frac12\sigma\right)^2\\
            &\leq \left(\frac12 B\mu x\right)^2+\left(\frac12\sigma\right)^2+\frac12\left(\left(\frac32 B\mu x\right)^2+\left(\frac12\sigma\right)^2\right)+\frac12\left(\left(\frac12 B\mu x\right)^2+\left(\frac12\sigma\right)^2\right)\\
            &=\frac32 B^2\mu^2x^2+\frac{1}{2}\sigma^2.
        \end{align*}
       Since $\mathrm{Var}(a+b)\leq 2\mathrm{Var}(a)+2\mathrm{Var}(b)$, we have
        \begin{align*}
            \mathrm{Var}(\ma(x;\xi_1,\xi_2))&\leq 2\mathrm{Var}(\bar{g}(x;\xi))+2\kappa\mathrm{Var}(W(x;\xi_1,\xi_2))\\
            &\leq 3\kappa B^2\mu^2x^2+(\kappa+1)\sigma^2\\
            &\leq \frac{3\kappa B^2}{\left(1-\frac{\sqrt{\kappa}}{2}B\right)^2}\left[2\left(\mu\left(1-\frac{\sqrt{\kappa}}{2}B\right)x-\frac{\sqrt{\kappa}}{2}\sigma\right)^2+2\left(\frac{\sqrt{\kappa}}{2}\sigma\right)^2\right]+(\kappa+1)\sigma^2,
        \end{align*}
        where we use the inequality $(a+b)^2\leq 2a^2+2b^2$ in the last inequality. 
        Consequently,
        \begin{align*}
            \ep\left[\abs{\ma(x;\xi_1,\xi_2)}^2\right]&\leq \mathrm{Var}(\ma(x;\xi_1,\xi_2))+\abs{\nabla F(x)}^2\\
            &\leq\left(\frac{6\kappa B^2}{\left(1-\frac{\sqrt{\kappa}}{2}B\right)^2}+1\right)\abs{\nabla F(x)}^2+\left(\frac{3\kappa^2B^2}{2\left(1-\frac{\sqrt{\kappa}}{2}B\right)^2}+(\kappa+1)\right)\sigma^2,
        \end{align*}
        and the assumption on the variance of stochastic gradients in Lemma \ref{advanced nonasymptotic convergence result for SGD} holds. 
        
        By the conclusions of Lemma \ref{advanced convergence result for SGD} and Lemma \ref{advanced nonasymptotic convergence result for SGD}, for sufficient large $T$  and stepsizes $\{\eta_t\}$ with $\gamma_t=\left(\frac{1}{(K+t)^a}\right)$ where $0<a\leq 1$ or $\gamma_t=\mo\left(1/T^{a}\right)$ where $0<a<1$, we have  $\ep\left[\abs{x^{T-1}-x_F^*}\right]\leq \frac14\abs{x_F^*}$. Then
        \[\ep[(x^{T-1})^2]=\ep\left[\left(x^{T-1}-x_F^*+x_F^*\right)^2\right]\geq \ep\left[(x_F^*)^2-2\abs{x_F^*}\abs{x^{T-1}-x_F^*}\right]\geq \frac12(x_F^*)^2.\] 
        Thus, we obtain
        \begin{align*}
            \ep[(\nabla f_{\mh}(x^{T-1}))^2]&=\mu^2\ep[(x^{T-1})^2]\geq\frac{\mu^2}{2}(x_F^*)^2\\
            &=\frac{\kappa\sigma^2}{8\left(1-\frac{\sqrt{\kappa}}{2}B\right)^2}\geq \frac{\kappa\sigma^2}{8\left(1-\frac{\kappa}{4}B^2\right)}=\mo\left(\frac{\kappa\sigma^2}{1-\kappa B^2}\right).
        \end{align*}
        Finally, for the function value sub-optimality gap, we have
        \[\ep\left[f_{\mh}(x^{T-1})-f_{\mh}^*\right]=\frac{\mu}{2}\ep[(x^{T-1})^2]\geq \frac{\mu}{4}(x_F^*)^2=\mo\left(\frac{\kappa\sigma^2}{\mu(1-\kappa B^2)}\right).\]
        The lower bounds are thus verified.
    \end{proof}
    
    \section{Experimental Details}\label{Details and additional experiments}
    \subsection{Detailed Experimental Setup for the Synthetic Experiment}
    \textbf{Model COnstruction.} In the synthetic experiment,
    we consider a one-dimensional optimization problem with $n$ workers, all of which are honest (i.e., $b=0$). The local objective functions are constructed as
    \[f_i(x)=\begin{cases}
        f_1(x) = \frac{1}{2}ax^2+cx,&\;\mathrm{if}\;i=1,2,\ldots,k,\\
        f_2(x) = \frac{1}{2}ax^2-cx,&\;\mathrm{if}\;i=k+1,k+2,\ldots,2k,\\
        f_3(x) = \frac{1}{2}(a+d)x^2,&\;\mathrm{if}\;i=2k+1,2k+2,\ldots,n,
    \end{cases}\]
    where constants $a$, $c$, $d$, and $k$ are parameters to be specified.
    
    Given any $G,B \ge 0$, we construct $\{f_i\}_{i=1}^n$ so that they exactly satisfy the $(G,B)$-bounded dissimilarity condition: 
    \[\frac1n\sum_{i=1}^n\abs{\nabla f_i(x)-\nabla f(x)}^2=G^2+B^2\abs{\nabla f(x)}^2,\;\forall x\in\rr,\]
    where $f=\frac1n\sum_{i=1}^nf_i$. To this end, given any $G,B\geq 0$, it suffices to set the parameters $a,c,d,k$ such that:
    \begin{equation}\label{eq:synthetic papameters}
        c=\sqrt{\frac{n}{2k}}G, B^2<\frac{2k}{n-2k},\;\text{and}\;\frac{\sqrt{2k(n-2k)}-(n-2k)B}{n}d=aB.
    \end{equation}
    
    \textbf{The constructed aggregation rule.} To isolate the effect of $\kappa$, we adopt an adversarial aggregation rule as discussed in Appendix C: Given the honest updates $g_1,g_2,\ldots,g_n$, the variance 
    $V^2(g_1,g_2,\ldots,g_n)=\frac1n\sum_{i=1}^n\abs{g_i-\bar{g}}^2$
    with $\bar{g}=\frac1n\sum_{i=1}^ng_i$ is calculated, and the aggregation rule outputs 
    \[\ma(g_1,g_2,\ldots,g_n)=\bar{g}\pm\sqrt{\kappa V^2(g_1,g_2,\ldots,g_n)},\]
     where the sign is chosen adversarially to push the iterates away from the exact optimum. This rule satisfies the $(0,\kappa)$-robustness condition.

    \textbf{Parameter settings.} We fix $a=G=1.0$, $n=20$, and $k=7$, and compute $c$ and $d$ according to \eqref{eq:synthetic papameters}. For fixed $\kappa$ and $G$, we vary $B^2$ to study its effect on the final error of R-DSGD and R-DSGD-M under the above aggregation rule. We also vary $\kappa$ to examine its impact. For each pair $(\kappa, B^2)$, we tune the stepsize using two schemes: (i) diminishing stepsizes:
    \[\gamma_t=\begin{cases}
        \gamma_0 & \;\text{if}\; 0<t<T/2,\\
        \frac{\gamma_0}{t+1-T/2} &\;\text{if}\; T/2\leq t\leq T,
    \end{cases}\]
    and (ii) constant stepsizes $\gamma_t=\gamma_0$.  The initial stepsize $\gamma_0$ is selected from ${0.5, 0.2, 0.1, 0.05, 0.01, 0.005, 0.001}$, and the total number of iterations $T$ is chosen from ${10, 20, 50, 100, 200, 1000, 2000}$. We report the best performance (i.e., the smallest error) over all configurations.
    \subsection{Detailed Experimental Setups for MNIST and CIFAR-10 Experiments}
    \textbf{Models.}
    For experiments on the MNIST dataset, we use a two-layer multilayer perceptron (MLP) consisting of one hidden layer of 200 nodes. The input images are flattened into 784-dimensional vectors, passed through a fully connected layer with ReLU activation, and mapped to a fully connected output layer with 10 neurons. The network is trained using the cross-entropy loss function.
    For experiments on the CIFAR-10 dataset, we implement a ResNet-20 deep learning model \cite{cifar}. The batch normalization (BN) layers in the ResNet-20 model are replaced with group normalization layers since BN layers have a poor performance with heterogeneous data across workers.
    
    \textbf{Parameter Settings.} For the MNIST task, we use the following stepsize schedule in training MLP:
    \[\gamma_t=\begin{cases}
        \gamma_0, & \mathrm{if}\;t\leq \frac{T}{15},\\
        0.5\gamma_0, & \mathrm{if}\;\frac{T}{15}<t\leq \frac{T}{6},\\
        0.25\gamma_0, & \mathrm{if}\;\frac{T}{6}<t\leq \frac{T}{3},\\
        0.125\gamma_0, & \mathrm{if}\;\frac{T}{3}<t\leq T,
    \end{cases}\]
    where the initial stepsize $\gamma_0$ is selected from the grid $\{0.5, 0.2,0.1,0.05,0.025, 0.01\}$. We choose the momentum parameter $\beta\in\{0.0, 0.9\}$ for R-DSGD and R-DSGD-M, respectively. The total number of iterations is set to $T=30000$ for a batch size of 1, and $T=15000$ for a batch size of 64. 

    In the CIFAR-10 task, each algorithm is run for $50000$ iterations with a cosine annealing learning rate schedule \cite{loshchilov2017sgdr}. Specifically, the stepsize at iteration $t$ is given by $\gamma_t=\frac{1+\cos(t\pi/50000)}{2}\gamma_0$, where $\gamma_0$ is selected from $\gamma_0\in\{0.2,0.1,0.05,0.025,0.01\}$. For R-DSGD-M, the momentum parameter is tuned over $\beta \in \{0.5, 0.6, 0.7, 0.8, 0.9\}$. We report the best final training loss and top-1 test accuracy across all configurations.

    \textbf{Aggregation Rules.}
    In our experiments, we consider the following robust aggregation rules:
    \begin{itemize}
        \item \textbf{Krum} \cite{blanchard2017machine}. Given $n$ vectors $x_1,x_2,\ldots,x_n\in\rr^d$, Krum selects a vector from from $\{x_1,x_2,\ldots,x_n\}$ that is closest to its $n-b-1$ neighbors. Specifically, for each $i\in[n]$, sort $\{x_1,x_2,\ldots,x_n\}\setminus\{x_i\}$ such that
			\[\norm{x_i-x_{i_1}}\leq\norm{x_i-x_{i_2}}\leq\cdots\leq\norm{x_i-x_{i_{n-1}}},\]
			and define $\hat{d}_i=\sum_{j=1}^{n-b-1}\norm{x_i-x_{i_j}}^2$. The output of Krum is then given by 
			\[\mathrm{Krum}(x_1,x_2,\ldots,x_n)=x_{i^*},\;\text{where}\; i^*\in\argmin_{i \in [n]} \hat{d}_i.\]
            Krum can be extended to Multi-Krum as follows. Given an additional parameter $q$ and suppose that $\hat{d}_{i_1}\leq\hat{d}_{i_2}\leq\cdots\leq\hat{d}_{i_n}$, then the output of Multi-Krum is given by
            \[\mathrm{Multi\text{-}Krum}(x_1,x_2,\ldots,x_n)=\frac{1}{q}\sum_{j=1}^qx_{i_j}.\]
        When $q=1$, Multi-Krum reduces to Krum. In our experiments, we observe that our theories can also explain the performance of Multi-Krum with various choices of $q$ , so we only report results for choosing $q=b$.
        \item \textbf{Coordinate-wise median (CwM)} \cite{yin2018byzantine}. Given $n$ vectors $x_1,x_2,\ldots,x_n\in\rr^d$, CwM outputs a vector $\mathrm{CwM}(x_1,x_2,\ldots,x_n)\in\rr^d$ whose $i$-th coordinate is given by
        \[[\mathrm{CwM}(x_1,x_2,\ldots,x_n)]_i=\mathrm{Median}([x_1]_i,[x_2]_i,\ldots,[x_n]_i),\]
        where $[\cdot]_i$ denotes the $i$-th coordinate of a vector, and $\mathrm{Median}$ represents the median of $n$ real numbers.
        \item \textbf{Coordinate-wise trimmed mean (CwTM)} \cite{yin2018byzantine}. Suppose that $n$ vectors $x_1,x_2,\ldots,x_n\in\rr^d$ and a parameter $q$ representing the estimated number of Byzantine workers are given. Let $\tau_k$ be a permutation on $[n]$ such that 
		$[x_{\tau_k(1)}]_k\leq[x_{\tau_k(2)}]_k\leq\cdots\leq[x_{\tau_k(n)}]_k$. CwTM outputs a vector $\mathrm{CwTM}(x_1,x_2,\ldots,x_n)$ in $\rr^d$, whose $k$-th coordinate is given by
			\[[\mathrm{CwTM}(x_1,x_2,\ldots,x_n)]_k=\frac{1}{n-2q}\sum_{j=q+1}^{n-q}[x_{\tau_k(j)}]_k.\]
        In our experiments, we set $q=b$ since the number of Byzantine workers is known.
        \item \textbf{Geometric median} \cite{chen2017distributed, wu2020federated}. Given $n$ vectors $x_1,x_2,\ldots,x_n\in\rr^d$, GM outputs 
        \[\mathrm{GM}(x_1,x_2,\ldots,x_n)\in\argmin_{x\in\rr^d}\sum_{i=1}^n\norm{x_i-x}.\]
        The solution is unique unless $x_1,x_2,\ldots,x_n$ are on a common line. Since this optimization problem generally does not admit a closed-form solution, we compute an approximate solution using the smoothed Weiszfeld algorithm:\[\begin{cases}
                \beta_i^{(k)} = \frac{1}{\max\{\nu,\norm{v^{(k)}-x_i}\}},\\
            v^{(k+1)} =\frac{\sum_{i=1}^n\beta_i^{(k)}x_i}{\sum_{i=1}^n\beta_i^{(k)}},
            \end{cases}\]
        where we use the superscript $(k)$ here to denote the iteration indices and $\nu$ is a positive small number to avoid division by zero. In our experiments, we run the smoothed Weiszfeld algorithm for 50 iterations per aggregation, which yields an approximation error that is negligible in practice.
    \end{itemize}
    \textbf{Attacks.}
    In our experiments, we evaluate the robustness of the aggregation rules against the following three attacks:
    \begin{itemize}
        \item \textbf{Sign flip} \cite{zhu2021byzantine}. After receiving the model $x^t$ at the $t$-th iteration, each Byzantine workers first generates a honest stochastic gradient $g_i^t$, and then submits the negated gradient $-g_i^t$ to the server.
        \item \textbf{Lable flip} \cite{li2019rsa}. Each sample $(x,y)$ in the local dataset of a Byzantine worker is modified to $(x,9-y)$, where $x$ denotes the features and $y\in \left\lbrace 0,1,\ldots,9 \right\rbrace$ is the label. The Byzantine workers then use these poisoned samples to generate stochastic gradients once they receive the model from the server.
        \item \textbf{A little is enough} \cite{baruch2019little}. The honest updates $\left\lbrace g_i\right\rbrace_{i\in\mh}$ in each round as well as the aggregation rule $\ma$ used by the server are assumed to be known by the Byzantine workers. Each Byzantine worker calculates the mean $\mu$ and the standard deviation $\sigma$ of honest updates as 
        \[\mu=\frac{1}{h}\sum_{i\in\mh}g_i,\;\sigma=\sqrt{\sum_{i\in\mh}(x-\mu)^T(x-\mu)}.\]
        The $j$-th Byzantine worker then constructs its update as $g_j = \mu + \alpha \sigma$, where $\alpha$ is chosen greedily from a candidate set $\{-2,-1,-0.5,0.5,1,2\}$ to maximize the deviation of the aggregated update $\ma(g_1,g_2,\ldots,g_n)$ from the honest mean $\mu$ in terms of $L_2$ distance.
    \end{itemize}
    

\end{document}